\documentclass[
reqno,
10pt,
oneside
]{article}

\usepackage[ngerman, english]{babel}
\usepackage[utf8]{inputenc}
\usepackage[normalem]{ulem}
\usepackage[babel,french=guillemets,german=swiss]{csquotes}
\usepackage[center,font=small]{caption}
\usepackage{cite}
\usepackage{bbm}
\usepackage[raggedright]{titlesec}
\usepackage[dvipsnames]{xcolor}
\usepackage{enumerate}

\titleformat{\section}{\normalfont\large\bfseries}{\thesection}{1em}{}
\titleformat{\subsection}{\normalfont\bfseries}{\thesubsection}{1em}{}

\usepackage{tocbasic}
\DeclareTOCStyleEntry[
beforeskip=.2em plus 1pt,
]{tocline}{section}

\usepackage[%
left=1.5in,
right=1.5in,
bottom=1in,
top=0.8in,
footskip=0.4in,
]{geometry}

\usepackage{color}
\definecolor{LinkColor}{rgb}{0,0,1}
\definecolor{LinkColor2}{rgb}{0,0.5,0}
\definecolor{lbcolor}{rgb}{0.85,0.85,0.85}
\definecolor{FrameColor}{rgb}{0.85,0.85,0.85}
\definecolor{rosso}{rgb}{0.8,0,0}
\definecolor{lightgray}{rgb}{0.5,0.5,0.5}
\definecolor{violet}{rgb}{0.65,0,0.65}
\definecolor{darkgreen}{rgb}{0,0.5,0}

\usepackage{enumitem}

\usepackage{graphicx}
\usepackage{placeins}
\usepackage{overpic}

\usepackage{amsmath}
\usepackage{amssymb}
\usepackage{dsfont}
\usepackage{empheq}
\allowdisplaybreaks
\usepackage{amsthm}

\usepackage[%
pdftitle={Titel},%
pdfauthor={Autor},%
pdfcreator={LaTeX, LaTeX with hyperref and KOMA-Script},
pdfsubject={Betreff}, 
pdfkeywords={Keywords}
]{hyperref} 

\hypersetup{%
	colorlinks	=true,
	linkcolor	=LinkColor,%
	anchorcolor	=LinkColor,%
	citecolor	=LinkColor2,%
	filecolor	=LinkColor,%
	menucolor	=LinkColor,%
	urlcolor	=LinkColor,%
}
\usepackage{marginnote}

\usepackage{verbatim}

\newtheorem{theorem}{Theorem}[section]
\newtheorem{lemma}[theorem]{Lemma}
\newtheorem{proposition}[theorem]{Proposition}
\newtheorem{corollary}[theorem]{Corollary}
\newtheorem{definition}[theorem]{Definition}

\theoremstyle{definition}
\newtheorem{remark}[theorem]{Remark}

\makeatletter
\renewenvironment{proof}[1][\proofname]{%
	\par\pushQED{\qed}\normalfont%
	\topsep6\p@\@plus6\p@\relax
	\trivlist\item[\hskip\labelsep\bfseries#1\@addpunct{.}]%
	\ignorespaces
}{%
	\popQED\endtrivlist\@endpefalse
}
\makeatother

\makeatletter
\renewcommand\paragraph{\@startsection{paragraph}{4}{\z@}%
	{1ex \@plus1ex \@minus.2ex}%
	{-1em}%
	{\normalfont\normalsize\bfseries}}
\renewcommand\subparagraph{\@startsection{paragraph}{4}{\z@}%
	{1ex \@plus1ex \@minus.2ex}%
	{-1em}%
	{\normalfont\normalsize\itshape}}
\makeatother

\newcommand{\abs}[1]{| #1 |}

\newcommand{\Bigabs}[1]{\Big| #1 \Big|}

\newcommand{\norm}[1]{\| #1 \|}
\newcommand{\bignorm}[1]{\big\| #1 \big\|}

\newcommand{\ang}[2]{ \langle #1 , #2  \rangle}
\newcommand{\bigang}[2]{ \big< #1 , #2  \big>}

\newcommand{\scp}[2]{ \left( #1 , #2  \right)}
\newcommand{\bigscp}[2]{\big( #1 , #2 \big)}

\newcommand{\mean}[2]{\textnormal{mean}\scp{#1}{#2}}
\newcommand{\meano}[1]{{\langle #1 \rangle}_{\Omega}}
\newcommand{\meang}[1]{{\langle #1 \rangle}_{\Gamma}}

\newcommand{\R}{\mathbb R}
\newcommand{\N}{\mathbb N}
\newcommand{\n}{\mathbf{n}}

\newcommand{\intO}{\int_\Omega}

\newcommand{\intG}{\int_\Gamma}

\newcommand{\dist}{\textnormal{dist\,}}

\newcommand{\dr}{\;\mathrm dr}
\newcommand{\dx}{\;\mathrm{d}x}

\newcommand{\ds}{\;\mathrm ds}

\newcommand{\dxs}{\;\mathrm{d}x\;\mathrm{d}s}
\newcommand{\dxr}{\;\mathrm{d}x\;\mathrm{d}r}
\newcommand{\dGs}{\;\mathrm{d}\Ga\;\mathrm{d}s}
\newcommand{\dGr}{\;\mathrm{d}\Ga\;\mathrm{d}r}

\newcommand{\dG}{\;\mathrm d\Gamma}

\newcommand{\ddt}{\frac{\mathrm d}{\mathrm dt}}

\newcommand{\del}{\partial}
\newcommand{\delt}{\partial_{t}}
\newcommand{\delth}{\partial_{t}^{h}}

\newcommand{\deln}{\partial_\n}

\newcommand{\Grad}{\nabla}
\newcommand{\Lap}{\Delta}
\newcommand{\Div}{\textnormal{div}}
\newcommand{\Gradg}{\nabla_\Ga}
\newcommand{\Lapg}{\Delta_\Ga}
\newcommand{\Divg}{\textnormal{div}_\Ga}

\newcommand{\emb}{\hookrightarrow}

\newcommand{\suchthat}{\;\ifnum\currentgrouptype=16 \middle\fi|\;}

\newcommand{\HH}{\mathcal{H}}

\newcommand{\VV}{{\mathcal{V}}}

\newcommand{\WW}{\mathcal{W}}
\newcommand{\LL}{\mathcal{L}}
\newcommand{\BLL}{\boldsymbol{\mathcal{L}}}
\newcommand{\Wms}[1]{{\mathcal{W}_{K,L,m}^#1}}

\def\bv{\mathbf v}
\def\bw{\mathbf w}
\def\bL{\mathbf L}

\newcommand{\Om}{\Omega}
\newcommand{\Ga}{\Gamma}

\newcommand{\blk}{\color{black}}

\def\oml{\omega^{K,L}(\phi_0, \psi_0) }

\begin{document}

%
%

\title{\bfseries\Large 
    Long-time dynamics of a\\ 
    bulk-surface convective Cahn--Hilliard system:\\[0.5ex] 
    Pullback attractors and convergence to equilibrium
\\[-1.5ex]$\;$}

\author{
	Patrik Knopf \footnotemark[1] \footnotemark[4]
    \and Andrea Poiatti \footnotemark[2]
    \and Jonas Stange \footnotemark[1]
    \and Sema Yayla \footnotemark[3] 
}
\date{ }

\maketitle

\renewcommand{\thefootnote}{\fnsymbol{footnote}}

\footnotetext[1]{
    Faculty for Mathematics, 
    University of Regensburg, 
    Universitätsstr.~31,
    D-93053 Regensburg, 
    Germany \newline
	\tt(%
        \href{mailto:patrik.knopf@ur.de}{patrik.knopf@ur.de},
        \href{mailto:jonas.stange@ur.de}{jonas.stange@ur.de}%
        ).
}
\footnotetext[2]{
    Faculty of Mathematics, 
    University of Vienna, 
    Oskar-Morgenstern-Platz~1, 
    A-1090 Vienna, Austria \newline
	\tt(%
          \href{mailto:andrea.poiatti@univie.ac.at}{andrea.poiatti@univie.ac.at}%
        ).
}
\footnotetext[3]{
    Department of Mathematics, Faculty of Science,
    Hacettepe University, 
    Beytepe 06800, Ankara, Turkey \newline
	\tt(%
          \href{mailto:semasimsek@hacettepe.edu.tr}{semasimsek@hacettepe.edu.tr}%
        ).
}

\footnotetext[4]{
    Institute for Analysis,
    Department of Mathematics,
    Karlsruhe Institute of Technology (KIT), 
    D-76128 Karlsruhe, 
    Germany \newline
	\tt(%
        \href{mailto:patrik.knopf@kit.edu}{patrik.knopf@kit.edu}%
        ).
}

\begin{center}
    $\phantom{x}$\\[-6ex]
	\scriptsize
	\color{white}
	{
		\textit{This is a preprint version of the paper. Please cite as:} \\  
		P.~Knopf, A.~Poiatti, J.~Stange, S. Yayla
        \textit{[Journal]} \textbf{xx}:xx 000-000 (2024), \\
		\texttt{https://doi.org/...}
	}
\end{center}

\medskip

%
%

\begin{small}
\textbf{Abstract.} We study the long-time dynamics of a bulk--surface convective Cahn--Hilliard system describing phase separation processes with bulk-surface interaction. The presence of convection terms leads to a non-autonomous dynamical system and prevents the associated free energy from being a Lyapunov functional, which makes the analysis of the asymptotic behavior considerably more challenging.
First, we establish an instantaneous regularization property for weak solutions. Next, interpreting the evolution as a continuous two-parameter process, we prove the existence of a minimal pullback attractor. Finally, under suitable decay assumptions on the velocity fields, we show that every solution converges as $t\to\infty$ to a single steady state. 
The proof of this convergence relies on the {\L}ojasiewicz--Simon inequality combined with customized decay estimates that compensate for the lack of a monotone energy functional. 
\\[2ex]
\\[2ex]	
\end{small}

\begin{footnotesize}
\setcounter{tocdepth}{2}
\hypersetup{linkcolor=black}
\tableofcontents
\end{footnotesize}

\smallskip

\setlength\parindent{0ex}
\setlength\parskip{1ex}
\allowdisplaybreaks
\numberwithin{equation}{section}
\renewcommand{\thefootnote}{\arabic{footnote}}

\newpage


\section{Introduction} 
\label{SECT:INTRO}

The Cahn--Hilliard equation was originally introduced to describe spinodal decomposition in binary alloys. Subsequently, it was recognized that the equation also provides an effective framework for modeling later stages of phase separation processes, such as Ostwald ripening. Today, the Cahn--Hilliard equation has become one of the most widely used diffuse-interface models, with applications spanning materials science, life sciences, and image processing.

In its classical formulation in a bounded domain, the Cahn--Hilliard system is typically supplemented with homogeneous Neumann boundary conditions for both the phase-field variable and the chemical potential 
when studying the associated initial-boundary value problem. However, both mathematicians and physicists (see, e.g., \cite{Binder1991,Fischer1997,Fischer1998,Kenzler2001}) have observed that such boundary conditions lead to significant limitations when an accurate description of the dynamics near the boundary is required.

A key drawback is that the homogeneous Neumann condition for the phase-field variable forces diffuse interfaces between the two components to intersect the boundary at a right angle. In many practical situations, this assumption is unrealistic, since the contact angle generally deviates from ninety degrees and may even evolve dynamically over time. Moreover, the homogeneous Neumann condition for the chemical potential prevents any exchange of material between the bulk and the boundary. While this constraint is appropriate in settings where total bulk mass must be conserved, it becomes inadequate when processes such as adsorption or chemical reactions occur at the boundary, where mass transfer between bulk and surface must be taken into account.

Motivated by these limitations of the classical formulation and by the need to more accurately capture short-range interactions between bulk and surface, several Cahn--Hilliard models with dynamic boundary conditions have been proposed and studied in the literature. In particular, dynamic boundary conditions have attracted considerable attention in recent years (see, e.g., \cite{Gal2016,Gal2019,Gal2023,Garcke2020,Lv2024b,Colli2020,Colli2022a,Fukao2021,Miranville2020,Knopf2021b,Stange2025,Lv2024}), as they allow for a consistent description of coupled bulk-surface dynamics, including variables contact angles and the exchange of material between bulk and boundary.  
For more information on the Cahn--Hilliard equation with and without dynamic boundary conditions, we refer to the recent review paper \cite{Wu2022} and the book \cite{Miranville-Book}.

In this paper, we investigate the long-time dynamics of the \textit{bulk-surface convective Cahn--Hilliard model}
\begin{subequations}
\label{conCH*}
\begin{alignat}{2}
       \label{conCH*:1}
    &\delt \phi + \Div(\phi \bv )
    = \Lap \mu
    &&\qquad\text{in $Q_\tau$}, 
    \\
    \label{conCH*:2}
    &\mu = - \Lap \phi + F'(\phi)
    &&\qquad\text{in $Q_\tau$}, 
    \\
    \label{conCH*:3}
    &\delt \psi + \Divg(\psi \bw ) 
    = \Lapg \theta - \beta\deln \mu 
    &&\qquad\text{on $\Sigma_\tau$}, 
    \\
    \label{conCH*:4}
    &\theta = - \Lapg \psi +  G'(\psi)
    + \alpha\deln \phi
    &&\qquad\text{on $\Sigma_\tau$}, 
    \\
    \label{conCH*:5}
    &\begin{cases}
        K \deln \phi = \alpha \psi - \phi &\text{if}\; K\in [0,\infty),\\
        \deln \phi = 0 &\text{if}\; K=\infty
    \end{cases}
    &&\qquad\text{on $\Sigma_\tau$},
    \\
    \label{conCH*:6}
    &\begin{cases}
        L\deln \mu = \beta \theta - \mu &\text{if}\; L\in [0,\infty),\\
        \deln \mu = 0 &\text{if}\; L=\infty
    \end{cases}
    &&\qquad\text{on $\Sigma_\tau$},
    \\
    \label{conCH*:7}
    &\phi\vert_{t=\tau} = \phi_\tau
    &&\qquad\text{in $\Omega$},
    \\
    \label{conCH*:8}
    &\psi\vert_{t=\tau} = \psi_\tau
    &&\qquad\text{on $\Gamma$},
\end{alignat}
\end{subequations}
which was introduced in \cite{Knopf2024} and further analyzed in \cite{Knopf2025,Giorgini2025}. This model represents a convective extension of a general bulk-surface Cahn--Hilliard system that unifies several Cahn--Hilliard formulations with dynamic boundary conditions (e.g., \cite{Goldstein2011,Liu2019,Knopf2021a,Knopf2020}) within a single framework. 
The model is of particular interest as it naturally arises as a subsystem in Cahn--Hilliard-based models for two-phase flows with dynamic boundary conditions; see, for instance, \cite{Giorgini2023,Knopf2025a,Stange2026,Colli2023,Gal2026}.

In system \eqref{conCH*}, $\Om \subset \R^d$, $d\in\{2,3\}$, is a bounded domain with boundary $\Ga\coloneqq  \del \Om$.
For any initial time $\tau\in\R$, we use the notation $Q_\tau \coloneqq \Om\times(\tau,\infty)$ and $\Sigma_\tau \coloneqq \Ga\times(\tau,\infty)$. The outward pointing unit normal vector on $\Ga$ is denoted by $\n$, whereas $\deln$ denotes the corresponding normal derivative on the boundary. Moreover, the symbols $\Gradg$, $\Divg$, and $\Lapg$ stand for the surface gradient, the surface divergence, and the Laplace--Beltrami operator on $\Ga$, respectively.

The functions $\phi : Q_\tau \to \R$ and $\mu : Q_\tau \to \R$ denote the phase field and the chemical potential in the bulk, respectively, while $\psi : \Sigma_\tau \to \R$ and $\theta : \Sigma_\tau \to \R$ represent the corresponding quantities on the boundary.
The phase fields $\phi$ and $\psi$ are coupled through the boundary condition \eqref{conCH*:5}, whereas the chemical potentials $\mu$ and $\theta$ are coupled via the boundary condition \eqref{conCH*:6}. There, the parameters $K, L \in [0,\infty]$ determine the type of coupling between the respective bulk and surface quantities.
The functions $F^\prime$ and $G^\prime$ are the derivatives of double-well potentials $F$ and $G$. 
In this paper, we consider singular potentials whose exact properties are specified in Section~\ref{SUBSEC:ASS}. In particular, the physically relevant \textit{Flory--Huggings potential}
\begin{align*}
	W_{\textup{log}}(s) \coloneqq \frac{\Theta}{2}\big[(1+s)\ln(1+s) + (1-s)\ln(1-s)\big] -\frac{\Theta_c}{2}s^2, \quad s\in[-1,1],
\end{align*}
is an admissible choice for both $F$ and $G$.

System \eqref{conCH*} is associated with the total free energy
\begin{align*}
    E_{\mathrm{free}}(\phi,\psi) 
    \coloneqq E_{\text{free},\Omega}(\phi) + E_{\text{free},\Gamma}(\phi,\psi),
\end{align*}
where
\begin{align*}
    E_{\text{free},\Omega}(\phi) 
    &\coloneqq  \intO \frac12\abs{\Grad \phi}^2 + F(\phi) \dx,
    \\
    E_{\text{free},\Gamma}(\phi,\psi) 
    &\coloneqq  \intG \frac12\abs{\Gradg \psi}^2 + G(\psi) \dG  
        + \chi(K)\intG \frac12(\alpha\psi - \phi)^2 \dG.
\end{align*}
Here, the function
\begin{equation*}
    \chi:[0,\infty]\rightarrow[0,\infty), \qquad r\mapsto  
    \begin{cases}
        \frac1r &\text{if}\; r\in (0,\infty),\\
        0 &\text{if}\; r\in\{0,\infty\},
    \end{cases}
\end{equation*}
is used to distinguish different cases corresponding to the choice of $K\in[0,\infty]$.
Sufficiently regular solutions of system \eqref{conCH*} satisfy the the \textit{energy identity}
\begin{align*}
    \begin{aligned}
    \ddt E_{\mathrm{free}}(\phi,\psi)
    &=  \intO \phi\bv\cdot\Grad \mu \dx
    +\intG \psi\bw\cdot \Gradg \theta  \dG
    - \intO \abs{\Grad \mu}^2 \dx \\
    &\quad - \intG  \abs{\Gradg \theta}^2 \dG 
    - \chi(L) \intG (\beta\theta - \mu)^2 \dG
    \end{aligned}
\end{align*}
on $[\tau,\infty)$, and the \textit{mass conservation law}
\begin{align*}
    \left\{
    \begin{aligned}
    &\beta\intO \phi(t) \dx + \intG \psi(t) \dG 
    = \beta\intO \phi_0 \dx + \intG \psi_0 \dG
    &&\quad\text{if $L\in[0,\infty)$,}
    \\
    &\intO \phi(t) \dx = \intO \phi_0 \dx 
    \quad\text{and}\quad
    \intG \psi(t) \dG = \intG \psi_0 \dG
    &&\quad\text{if $L=\infty$,}
    \end{aligned}
    \right.
\end{align*}
for all $t\in[\tau,\infty)$.

\paragraph{Main results, difficulties, and novelties of this paper.}
In this paper, we investigate the long-time dynamics of system \eqref{conCH*}. 
On the one hand, this question is of interest since insights gained for \eqref{conCH*} may provide information on the long-time behavior of more involved models in which \eqref{conCH*} appears as a subsystem, such as the two-phase flow models studied in \cite{Giorgini2023,Knopf2025a,Stange2026,Colli2023,Gal2026}. 
On the other hand, the problem is also appealing from the viewpoint of mathematical analysis. 

Remarkably, the long-time behavior of bulk-surface Cahn--Hilliard models in the non-convec\-tive case is by now well understood, and an extensive literature is available (see, e.g., \cite{Garcke2022,Miranville2020,Lv2024a,Lv2024b,Stange2025,Liu2019,Lv2025}). 
We also refer to \cite{Chill2006,Gal2008,Gilardi2010,Pruss2006} for the long-time analysis of further Cahn--Hilliard models with dynamic boundary conditions. 
In contrast, only very few results are available for convective Cahn--Hilliard systems.
In \cite{Colli2018a}, the authors study a viscous, convective Cahn--Hilliard system with dynamic boundary conditions. It is shown there that the $\omega$-limit set of this model is non-empty and that each of its elements is a solution to the corresponding stationary problem. We also mention the contribution \cite{Zhao2020}, which studies the long-time dynamics of a convective Cahn--Hilliard model in two dimensions. However, the convective term considered there has a very specific structure and is substantially different from the one appearing in system \eqref{conCH*}. We finally remark that we expect the new results presented in the recent contribution \cite{GrasselliPoiatti} could be extended to the bulk-surface Cahn--Hilliard equation, which will be the objective of future work.

The main difficulty of the long-time analysis of convective Cahn--Hilliard models such as \eqref{conCH*} is that the underlying dynamical system is non-autonomous. Therefore, unless the velocity fields are constant in time, it cannot be described as a semigroup. For this reason, we need to move from the framework of semigroups and global attractors to the more general theory of two-parameter processes and pullback attractors. 
Moreover, due to the presence of the additional convection terms, the total free energy $E_\mathrm{free}$ is in general no longer decreasing along solutions of \eqref{conCH*}. As a consequence, establishing convergence to a single equilibrium becomes significantly more challenging, since, unlike in the non-convective case, $E_\mathrm{free}$ is not a Lyapunov functional. 

The main results of this paper are as follows:
\begin{enumerate}[label=\textnormal{\bfseries(\Roman*)},topsep=0ex,leftmargin=*]
    \item \textbf{Instantaneous regularization of the weak solution.} 
    In Section~3, after recalling from \cite{Knopf2025,Giorgini2025} some well-posedness and regularity results, we establish a new instantaneous regularization property for weak solutions. Namely, the unique weak solution instantly regularizes at times strictly greater than the initial time due to parabolic smoothing.
    
    \item \textbf{Existence of a minimal pullback attractor.}
    In Section~4, we show that the underlying dynamical system of system \eqref{conCH*} can be interpreted as a continuous two-parameter process. We further observe that in the non-convective case (i.e., $(\bv,\bw)\equiv (\mathbf{0},\mathbf{0})$), the corresponding dynamical system is autonomous, and the continuous process can be reduced to a continuous semigroup. Finally, using abstract results from the theory of non-autonomous dynamical systems, we prove that the continuous process associated with \eqref{conCH*} actually admits a minimal pullback attractor.

    At this point, we also want to mention the work \cite{Jiang2015}, where the existence of a minimal pullback attractor to a non-autonomous Cahn–Hilliard–Darcy system with mass source was established by employing similar techniques.
    
    \item \textbf{Convergence to a steady state.}
    In Section~5, we finally show that each unique solution to system \eqref{conCH*} converges as $t\to\infty$ to a single steady state. In other words, this means that the $\omega$-limit set is actually a singleton. 
    Our proof relies on a 
    \textit{{\L}ojasiewicz--Simon inequality} and the fact that stationary solutions are \textit{uniformly strictly separated }from pure phases.
    As our dynamical system is non-autonomous and since the energy functional $E_\mathrm{free}$ is not necessarily decreasing along solutions, we need to devise customized decay estimates for the weak solution to prove its convergence to a steady state. In this context, it is essential that the velocity fields $(\bv,\bw)$ fulfill a certain decay property for large times in a suitable norm (see~\ref{D5}).

    To the best of our knowledge, this is the first work establishing convergence to equilibrium as $t \to \infty$ for a convective Cahn--Hilliard model where prescribed velocity fields cause the system to be non-autonomous.
    In this sense, this technique may have the potential to be applied to a variety of related situations and models, including even the case of multi-component phase-field systems (see, e.g., \cite{GCElliottPoiatti, GGPS,multiCAC}) or problems on evolving surfaces like \cite{CEGP}.
    
    We also refer to \cite{Jiang2015}, where convergence to equilibrium was established for a two-dimensional Cahn--Hilliard--Darcy model with mass source. However, the situation considered there somewhat differs from the one in the present paper since the external quantity rendering the dynamical system non-autonomous enters the Cahn--Hilliard equation through a source term on the right-hand side.
    
\end{enumerate}

\section{Notation, assumptions, and preliminaries} 
\label{SECT:PRELIM}

In this section, we introduce some notation, assumptions, and preliminaries that are supposed to hold throughout the remainder of this paper.

\subsection{Notation}
Let us first introduce some basic notation.

\begin{enumerate}[label=\textnormal{\bfseries(N\arabic*)}, leftmargin=*]
    \item $\N$ denotes the set of natural numbers excluding zero, whereas $\N_0 = \N\cup\{0\}$.
    
    \item Let $\Omega \subset \R^d$ with $d\in\{2,3\}$ be a bounded domain in $\R^d$, and let $\Gamma \coloneqq\del\Omega$ denote its boundary, which is assumed to be sufficiently regular. For any $s\geq 0$ and $p\in[1,\infty]$, the Lebesgue and Sobolev spaces for functions mapping from $\Om$ to $\R$ are denoted as $L^p(\Om)$ and $W^{s,p}(\Om)$. We write $\norm{\cdot}_{L^p(\Om)}$ and $\norm{\cdot}_{W^{s,p}(\Om)}$ to denote the standard norms on these spaces. In the case $p=2$, we use the notation $H^s(\Om) = W^{s,2}(\Om)$. In particular, $H^0(\Om)$ can be identified with $L^2(\Om)$. The Lebesgue and Sobolev spaces on $\Ga$ are denoted by $L^p(\Ga)$ and $W^{s,p}(\Ga)$ along with the corresponding norms $\norm{\cdot}_{L^p(\Ga)}$ and $\norm{\cdot}_{W^{s,p}(\Ga)}$, respectively. 
    For vector-valued functions mapping from $\Om$ into $\R^d$, we use the notation $\mathbf{L}^p(\Om)$, $\mathbf{W}^{s,p}(\Om)$ and $\mathbf{H}^s(\Om)$. The spaces $\mathbf{L}^p(\Ga)$, $\mathbf{W}^{s,p}(\Ga)$ and $\mathbf{H}^s(\Ga)$ are defined analogously. 

    Moreover, for any interval $I\subset\R$, any Banach space $X$, $p\in[1,\infty]$ and $k \in \N$, we write $L^p(I;X)$, $W^{k,p}(I;X)$ and $H^{k}(I;X) = W^{k,2}(I;X)$ to denote the Lebesgue and Sobolev spaces of functions with values in $X$. The standard norms on these spaces are denoted by $\|\cdot\|_{L^p(I;X)}$, $\|\cdot\|_{W^{k,p}(I;X)}$ and $\|\cdot\|_{H^k(I;X)}$, respectively. Further, we define
    \begin{align*}
        L^p_\mathrm{loc}(I;X) 
        &\coloneqq 
        \big\{ 
            u:I\to X \,\big\vert\, u \in L^p(J;X) \;\text{for every compact interval $J\subset I$}
        \big\},
        \\[1ex]
        L^p_\mathrm{uloc}(I;X) 
        &\coloneqq 
        \left\{ u:I\to X \,\middle|\,
        \begin{aligned}
        &u \in L^p_\mathrm{loc}(I;X) \;\text{and}\; \exists C>0\; \forall t\in\R:\\
        &\|u\|_{L^p(I\cap[t,t+1);X)} \le C
        \end{aligned}
        \right\}.
    \end{align*}
    The spaces $W^{k,p}_\mathrm{loc}(I;X)$, $H^k_\mathrm{loc}(I;X)$, $W^{k,p}_\mathrm{uloc}(I;X)$, $H^k_\mathrm{uloc}(I;X)$ are defined analogously. We further define
    \begin{align*}
        \norm{u}_{L^p_{\mathrm{uloc}}(I;X)} \coloneqq \sup_{t\in\R} \norm{u}_{L^p(I\cap [t,t+1);X)}.
    \end{align*}
    \item For any Banach space $X$, its dual space is denoted by $X'$. The corresponding duality pairing of elements $\phi\in X'$ and $\zeta\in X$ is denoted by $\ang{\phi}{\zeta}_X$. If $X$ is a Hilbert space, we write $\scp{\cdot}{\cdot}_X$ to denote its inner product. 
    
    \item For any bounded domain $\Om\subset \R^d$ ($d\in\N$) with Lipschitz boundary $\Ga$, $u\in H^1(\Om)'$ and $v\in H^1(\Ga)'$, we write
    \begin{align*}
        \meano{u}\coloneqq 
        \frac{1}{\abs{\Om}}\ang{u}{1}_{H^1(\Om)},
        \qquad
        \meang{v}\coloneqq 
        \frac{1}{\abs{\Ga}}\ang{v}{1}_{H^1(\Ga)}
    \end{align*}
    to denote the generalized means of $u$ and $v$, respectively. Here, $\abs{\Om}$ denotes the $d$-dimensional Lebesgue measure of $\Omega$, whereas $\abs{\Ga}$ denotes the $(d-1)$-dimensional Hausdorff measure of $\Gamma$. If $u\in L^1(\Om)$ or $v\in L^1(\Ga)$, the generalized mean can be expressed as
    \begin{align*}
        \meano{u} = \frac{1}{\abs{\Om}} \intO u \dx,
        \qquad
        \meang{v} = \frac{1}{\abs{\Ga}} \intG v \dG,
    \end{align*}
    respectively.

    \item For any bounded domain $\Om\subset \R^d$ ($d\in\N$) with Lipschitz boundary $\Gamma\coloneqq\partial\Omega$ and any $2\leq p <\infty$, we introduce the spaces
    \begin{align*}
        \mathbf{L}^p_\Div(\Om) &\coloneqq \{\bv\in\mathbf{L}^p(\Om) : \Div\;\bv = 0 \ \text{in~} \Om, \ \bv\cdot\mathbf{n} = 0 \ \text{on~} \Ga\},
        \\
        \mathbf{L}^p_\mathrm{tan}(\Ga)&\coloneqq\{\bw\in\mathbf{L}^p(\Ga) : \bw\cdot\n = 0 \ \text{on~}\Ga\},
        \\
        \mathbf{L}^p_\Div(\Ga)&\coloneqq\{\bw\in\mathbf{L}^p_\mathrm{tan}(\Ga) : \Divg\;\bw = 0 \ \text{on~}\Ga\}.
    \end{align*}
    We point out that in the definition of $\mathbf{L}^p_\Div(\Om)$  the expression $\Div\;\bv$ in $\Om$ is to be understood in the sense of distributions. If $\bv\in \mathbf{L}^p(\Om)$ ($p\ge 2$) with $\Div\;\bv = 0$ in $\Om$, we already know that $\bv\cdot\n \in H^{-1/2}(\Ga)$, and therefore, the relation $\bv\cdot\n = 0$ on $\Ga$ is well-defined.
    \item Let $(X,d)$ be a metric space. Then for any sets $A,B\subset X$, the \emph{Hausdorff semidistance} is defined as
    \begin{align*}
    	\dist_{X}(A,B) \coloneqq \underset{a\in A}{\phantom{|}\sup\phantom{|}} \!\! \underset{b\in B}{\phantom{|}\inf\phantom{|}} d \left( a,b\right).
    \end{align*} 
\end{enumerate}

\subsection{Assumptions} \label{SUBSEC:ASS}
We make the following general assumptions.

\begin{enumerate}[label=\textnormal{\bfseries(A\arabic*)}, leftmargin=*]
    \item  \label{ASSUMP:1} $\Omega$ is a non-empty, bounded domain in $\R^d$ with $d\in\{2,3\}$, whose boundary $\Ga\coloneqq\del\Om$ is at least of class $C^{2,1}$. Moreover, $\tau\in \R$ is a prescribed initial time, and we use the notation
    \begin{align*}
        Q_\tau\coloneqq \Om\times(\tau,\infty), \quad\Sigma_\tau \coloneqq\Ga\times(\tau,\infty).
    \end{align*}
    
    \item \label{ASSUMP:2} The constants in system \eqref{conCH*} satisfy $K,L \in [0,\infty]$, $\alpha\in[-1,1]$ and $\beta\in \R$ with $\alpha\beta\abs{\Omega} + \abs{\Gamma} \neq 0$. (The latter condition is required to apply a certain bulk-surface Poincar\'e inequality, see Lemma~\ref{lemma:poinc}).

\end{enumerate}
    
For the potentials $F,G :\R\rightarrow\R$ we assume the decompositions $F = F_1 + F_2$ and $G = G_1 + G_2$ with the following properties:
\begin{enumerate}[label=\textnormal{\bfseries(S\arabic*)}, leftmargin=*]
\item \label{S1} It holds $F_1,G_1\in C([-1,1])\cap C^2(-1,1)$ and
\begin{align*}
    \lim_{s\searrow -1} F_1^\prime(s) = \lim_{s\searrow -1} G_1^\prime(s) = -\infty \quad\text{and}\quad
    \lim_{s\nearrow 1} F_1^\prime(s) = \lim_{s\nearrow 1} G_1^\prime(s) = +\infty.
\end{align*}
Moreover, there exists a constant $\Theta > 0$ such that
\begin{align*}
    F_1^{\prime\prime}(s) \geq \Theta \quad\text{and}\quad G_1^{\prime\prime}(s)\geq \Theta
    \quad\text{for all $s\in(-1,1)$.}
\end{align*}
As usual, we extend $F_1(s) = G_1(s) = +\infty$ for all $s\not\in[-1,1]$
Without loss of generality, we assume $F_1(0) = G_1(0) = 0$ and $F_1^\prime(0) = G_1^\prime(0) = 0$. In particular, those entail
\begin{align*}
    F_1(s) \geq 0 \quad\text{and}\quad G_1(s) \geq 0
    \quad\text{for all~}s\in[-1,1].
\end{align*}
\item \label{S2} There exists positive constants $\kappa_1, \kappa_2 > 0$ such that
\begin{align}\label{DominationProperty}
    \abs{F_1^\prime(\alpha s)} \leq \kappa_1 \abs{G_1^\prime(s)} + \kappa_2 \quad\text{for all~}s\in(-1,1).
\end{align}
\item \label{S3} $F_2,G_2:\R\rightarrow\R$ are continuously differentiable with Lipschitz continuous derivatives.
\end{enumerate}
    
  \medskip

\subsection{Preliminaries}

\begin{enumerate}[label=\textnormal{\bfseries(P\arabic*)}, leftmargin=*]
    \item For any real numbers $s\geq 0$ and $p\in[1,\infty]$, we set
    \begin{align*}
        \mathcal{L}^p \coloneqq L^p(\Om)\times L^p(\Ga), \quad
        \mathcal{H}^s\coloneqq H^s(\Om)\times H^s(\Ga), \quad
        \mathcal{W}^{s,p}\coloneqq W^{s,p}(\Om)\times W^{s,p}(\Ga),
    \end{align*}
    provided that the boundary $\Gamma$ is sufficiently regular.
    As usual, we identify $\mathcal{L}^2$ with $\mathcal{H}^0$. Note that $\mathcal{H}^s$ is a Hilbert space with respect to the inner product
    \begin{align*}
        \bigscp{\scp{\phi}{\psi}}{\scp{\zeta}{\xi}}_{\mathcal{H}^s} \coloneqq \scp{\phi}{\zeta}_{H^s(\Om)} + \scp{\psi}{\xi}_{H^s(\Ga)} \quad\text{for all } \scp{\phi}{\psi}, \scp{\zeta}{\xi}\in\mathcal{H}^s
    \end{align*}
    and its induced norm $\norm{\cdot}_{\mathcal{H}^s} \coloneqq \scp{\cdot}{\cdot}_{\mathcal{H}^s}^{1/2}$. We recall that the duality pairing can be expressed as
    \begin{align*}
        \ang{\scp{\phi}{\psi}}{\scp{\zeta}{\xi}}_{\mathcal{H}^1} = \scp{\phi}{\zeta}_{L^2(\Om)} + \scp{\psi}{\xi}_{L^2(\Ga)}
    \end{align*}
    if $\scp{\phi}{\psi}\in \mathcal{L}^2$ and $\scp{\zeta}{\xi}\in \mathcal{H}^1$.
    
    \item\label{PRE:HKA} Let $L\in[0,\infty]$ and $\beta\in\R$. We introduce the closed linear subspace
    \begin{align*}
        \mathcal{H}_{L,\beta}^1 \coloneqq
        \begin{cases}
            \{(\phi,\psi)\in\mathcal{H}^1 : \phi = \beta\psi \text{ a.e.~on } \Ga\} &\text{if } L=0, \\
            \mathcal{H}^1 &\text{if } L \in (0,\infty] , \\
        \end{cases}
    \end{align*}
    Endowed with the inner product $\scp{\cdot}{\cdot}_{\mathcal{H}_{L,\beta}^1} \coloneqq \scp{\cdot}{\cdot}_{\mathcal{H}^1}$ and its induced norm, the space $\mathcal{H}_{L,\beta}^1$ is a Hilbert space. Moreover, we define the product
    \begin{align*}
        \ang{\scp{\phi}{\psi}}{\scp{\zeta}{\xi}}_{\mathcal{H}_{L,\beta}^1} \coloneqq \scp{\phi}{\zeta}_{L^2(\Om)} + \scp{\psi}{\xi}_{L^2(\Ga)}
    \end{align*}
    for all $\scp{\phi}{\psi}, \scp{\zeta}{\xi}\in\mathcal{L}^2$. By means of the Riesz representation theorem, this product can be extended to a duality pairing on $(\mathcal{H}_{L,\beta}^1)^\prime\times\mathcal{H}_{L,\beta}^1$, which will also be denoted as $\ang{\cdot}{\cdot}_{\mathcal{H}_{L,\beta}^1}$.
    
    \item Let $L\in[0,\infty]$ and $\beta\in\R$ be real numbers. We define for $\scp{\phi}{\psi}\in(\mathcal{H}^1_{L,\beta})^\prime$ the generalized bulk-surface mean 
    \begin{align*}
        \mean{\phi}{\psi} \coloneqq \frac{\ang{\scp{\phi}{\psi}}{\scp{\beta}{1}}_{\mathcal{H}^1_{L,\beta}}}{\beta^2\abs{\Om} + \abs{\Ga}},
    \end{align*}
    which reduces to
    \begin{align*}
        \mean{\phi}{\psi} = \frac{\beta\abs{\Om}\meano{\phi} + \abs{\Ga}\meang{\psi}}{\beta^2\abs{\Om} + \abs{\Ga}}
    \end{align*}
    if $(\phi,\psi)\in\mathcal{L}^2$. We then define the closed linear subspaces
    \begin{align*}
        \mathcal{V}_{L,\beta}^1 &\coloneqq \begin{cases} 
        \{\scp{\phi}{\psi}\in\mathcal{H}^1_{L,\beta} : \mean{\phi}{\psi} = 0 \} &\text{if~} L\in[0,\infty), \\
        \{\scp{\phi}{\psi}\in\mathcal{H}^1: \meano{\phi} = \meang{\psi} = 0 \} &\text{if~}L=\infty.
        \end{cases} 
    \end{align*}
    Note that these subspaces are Hilbert spaces with respect to the inner product $\scp{\cdot}{\cdot}_{\mathcal{H}^1}$ and its induced norm.
    For any $s> 1$, we further set
    \begin{align*}
      \VV_{L,\beta}^s \coloneqq   \mathcal{V}_{L,\beta}^1 \cap \HH^{s}.
    \end{align*} 
    
    \item Let $L\in[0,\infty]$ and $\beta\in\R$. We define a bilinear form on $\mathcal{H}^1\times\mathcal{H}^1$ by
    \begin{align*}
         \bigscp{\scp{\phi}{\psi}}{\scp{\zeta}{\xi}}_{L,\beta} \coloneqq &\intO\Grad\phi\cdot\Grad\zeta \dx + \intG\Gradg\psi\cdot\Gradg\xi \dG \\  
         &\quad + \chi(L)\intG (\beta\psi-\phi)(\beta\xi-\zeta)\dG
    \end{align*}
    for all $ \scp{\phi}{\psi}, \scp{\zeta}{\xi}\in\mathcal{H}^1$. Moreover, we set 
    \begin{align*}
        \norm{\scp{\phi}{\psi}}_{L,\beta} \coloneqq \bigscp{\scp{\phi}{\psi}}{\scp{\phi}{\psi}}_{L,\beta}^{1/2}
    \end{align*}
    for all $\scp{\phi}{\psi}\in\mathcal{H}^1$. The bilinear form $\scp{\cdot}{\cdot}_{L,\beta}$ defines an inner product on $\mathcal{V}^1_{L,\beta}$, and $\norm{\cdot}_{L,\beta}$ defines a norm on $\mathcal{V}^1_{L,\beta}$ that is equivalent to the norm $\norm{\cdot}_{\mathcal{H}^1}$ (see \cite[Corol\-lary~A.2]{Knopf2021}). Moreover, the space $\big(\mathcal{V}^1_{L,\beta}, \scp{\cdot}{\cdot}_{L,\beta}, \norm{\cdot}_{L,\beta}\big)$ is a Hilbert space.
    
    \item \label{PRELIM:bulk-surface-elliptic} For any $L\in[0,\infty]$ and $\beta\in\R$, we define the space
    \begin{align*}
        \mathcal{V}_{L,\beta}^{-1} \coloneqq \begin{cases} 
        \{\scp{\phi}{\psi}\in(\mathcal{H}^1_{L,\beta})^\prime : \beta\abs{\Om}\meano{\phi} + \abs{\Ga}\meang{\psi} = 0 \} &\text{if~} L\in[0,\infty), \\
        \{\scp{\phi}{\psi}\in(\mathcal{H}^1)^\prime: \meano{\phi} = \meang{\psi} = 0 \} &\text{if~}L=\infty.
        \end{cases}
    \end{align*}
    Using the Lax-Milgram theorem, one can show that for any $\scp{\phi}{\psi}\in\mathcal{V}_{L,\beta}^{-1}$, there exists a unique weak solution $\mathcal{S}_{L,\beta}(\phi,\psi) = \bigscp{\mathcal{S}_{L,\beta}^\Om(\phi,\psi)}{\mathcal{S}_{L,\beta}^\Ga(\phi,\psi)}\in\mathcal{V}^1_{L,\beta}$ to the following elliptic problem with bulk-surface coupling:
    \pagebreak[2]
        \begin{alignat*}{2}
            -\Lap\mathcal{S}_{L,\beta}^\Om &= - \phi \qquad &&\text{in } \Om, \\
            -\Lapg\mathcal{S}_{L,\beta}^\Ga + \beta\deln\mathcal{S}_{L,\beta}^\Om &= -\psi &&\text{on } \Ga, \\
            L\deln\mathcal{S}_{L,\beta}^\Om &= \beta\mathcal{S}_{L,\beta}^\Ga - \mathcal{S}_{L,\beta}^\Om \qquad&&\text{on } \Ga.
        \end{alignat*}
    This means that $\mathcal{S}_{L,\beta}(\phi,\psi)$ satisfies the weak formulation
    \begin{align*}
        \bigscp{S_{L,\beta}(\phi,\psi)}{\scp{\zeta}{\xi}}_{L,\beta} = - \ang{\scp{\phi}{\psi}}{\scp{\zeta}{\xi}}_{\mathcal{H}^1_{L,\beta}}
    \end{align*}
    for all test functions $\scp{\xi}{\zeta}\in\mathcal{H}^1_{L,\beta}$. Consequently, we have
    \begin{align*}
        \norm{\mathcal{S}_{L,\beta}(\phi,\psi)}_{\mathcal{H}^1} \leq C\norm{\scp{\phi}{\psi}}_{(\mathcal{H}^1_{L,\beta})^\prime}
    \end{align*}
    for all $\scp{\phi}{\psi}\in\mathcal{V}^{-1}_{L,\beta}$, for a constant $C\geq 0$ depending only on $\Omega$, $L$ and $\beta$. We can thus define the solution operator
    \begin{align*}
        \mathcal{S}_{L,\beta}:\mathcal{V}_{L,\beta}^{-1}\rightarrow\mathcal{V}^1_{L,\beta}, \quad \scp{\phi}{\psi}\mapsto\mathcal{S}_{L,\beta}(\phi,\psi) = \bigscp{\mathcal{S}_{L,\beta}^\Om(\phi,\psi)}{\mathcal{S}_{L,\beta}^\Ga(\phi,\psi)},
    \end{align*}
    as well as an inner product and its induced norm on $\mathcal{V}^{-1}_{L,\beta}$ by
    \begin{align*}
        \bigscp{\scp{\phi}{\psi}}{\scp{\zeta}{\xi}}_{L,\beta,\ast} &\coloneqq \bigscp{\mathcal{S}_{L,\beta}(\phi,\psi)}{\mathcal{S}_{L,\beta}(\zeta,\xi)}_{L,\beta}, \\
        \norm{\scp{\phi}{\psi}}_{L,\beta,\ast} &\coloneqq \bigscp{\scp{\phi}{\psi}}{\scp{\phi}{\psi}}_{L,\beta,\ast}^{1/2},
    \end{align*}
    for $\scp{\phi}{\psi}, \scp{\zeta}{\xi}\in\mathcal{V}_{L,\beta}^{-1}$. This norm is equivalent to the norm $\norm{\cdot}_{(\mathcal{H}^1_{L,\beta})^\prime}$ on $\mathcal{V}_{L,\beta}^{-1}$. For the case $L\in (0,\infty)$, we refer the reader to \cite[Theorem 3.3 and Corollary 3.5]{Knopf2021} for a proof of these statements. In the other cases, the results can be proven analogously.

    \item \label{PRELIM:SGRSPACE}
    Let $\alpha,\beta\in\R$, $K,L\in [0,\infty]$ as well as $m\in\R$ if $L\in[0,\infty)$ and $m=(m_1,m_2)\in\R^2$ if $L=\infty$. We define 
    \begin{align*}
       \mathcal{W}_{K,L,m}^1 \coloneqq  
        \begin{cases}
            \big\{ (\phi,\psi) \in \mathcal{H}_{K,\alpha}^1 \,\big|\,  
                \beta\abs{\Omega}\meano{\phi} + \abs{\Gamma} \meang{\psi} = m \big\}
            &\text{if}\; L\in[0,\infty),
            \\[1ex]
            \left\{ (\phi,\psi) \in \mathcal{H}_{K,\alpha}^1 \,\middle|\, 
                \begin{aligned}
                    \abs{\Omega}\meano{\phi} &= m_1,\\
                    \abs{\Gamma} \meang{\psi} &= m_2
                \end{aligned}
                \right\}
            &\text{if}\; L=\infty,    
        \end{cases}
    \end{align*}
    which is a closed subset of $\mathcal{H}^1$. Hence, it is a complete metric space with respect to the norm-induced metric on $\mathcal{H}^1$. For any $s> 1$, we further set
    \begin{align*}
      \Wms{s} \coloneqq   \mathcal{W}_{K,L,m}^1 \cap \HH^{s}.
    \end{align*} 
    Additionally, we define
    \begin{align*}
        \mathcal{X}_{K,L,m} \coloneqq \big\{(\phi,\psi)
        \in\mathcal{W}_{K,L,m}^1 : \norm{\phi}_{L^\infty(\Om)} \leq 1, \; \norm{\psi}_{L^\infty(\Ga)} \leq 1 \big\}.
    \end{align*}
    The metric $\mathrm{d}_{\mathcal{X}_{K,L,m}}$ is defined as follows:
    \begin{align*}
        \mathrm{d}_{\mathcal{X}_{K,L,m}}\big((\phi,\psi),(\zeta,\xi)\big) &\coloneqq \norm{(\phi,\psi) - (\zeta,\xi)}_{\mathcal{H}^1} + \left| \intO F_1(\phi)\dx - \intO F_1(\zeta)\dx \right|
        \\
        &\qquad + \left|\intG G_1(\psi)\dG - \intG G_1(\xi)\dG \right|,
    \end{align*}
    and $\big(\mathcal{X}_{K,L,m},\mathrm{d}_{\mathcal{X}_{K,L,m}}(\cdot,\cdot)\big)$ is thus a complete metric space.
\end{enumerate} 

\subsection{Useful inequalities}

We first state a \textit{bulk-surface Poincar\'{e} inequality}, which has been established in \cite[Lemma A.1]{Knopf2021}.

\begin{lemma}\label{lemma:poinc}
 Let $K\in[0,\infty)$ and $\alpha,\beta\in\mathbb{R}$ with $\alpha\beta\abs{\Om} + \abs{\Ga} \neq 0$ be arbitrary. Then there exists a constant $C_P >0$ depending only on $K, \alpha, \beta$ and $\Omega$ such that
\begin{align}\label{EQ:BSPI1}
	\norm{\scp{\phi}{\psi}}_{\mathcal{L}^2} \leq C_P \norm{\scp{\phi}{\psi}}_{K,\alpha}
\end{align}
for all $\scp{\phi}{\psi}\in\mathcal{H}^1_{K,\alpha}$ satisfying $\beta\abs{\Om}\meano{\phi} + \abs{\Ga}\meang{\psi} = 0$. 
\end{lemma}

Furthermore, we recall the following variant of Gronwall's lemma, which is sometimes referred to as the \textit{Uniform Gronwall Lemma}. It can be found, e.g., in \cite[Chapter~III., Lemma~1.1]{Temam1997}.

\begin{lemma}\label{Lemma:Gronwall}
    Let $g,h,y$ be three non-negative locally integrable functions on $[t_0,\infty)$ such that $y^\prime$ is locally integrable on $(t_0,\infty)$ and which satisfy
    \begin{align*}
        \ddt y(t) &\leq g(t)y(t) + h(t) \qquad\text{for a.e.~}t\geq t_0, 
    \end{align*}
    and
    \begin{align*}
        \int_t^{t+r} g(s)\ds \leq a_1, \quad \int_t^{t+r} h(s)\ds &\leq a_2, \quad \int_t^{t+r} y(s)\ds \leq a_3 \quad\text{for all~}t\geq t_0,
    \end{align*}
    where $r,a_1,a_2,a_3$ are positive constants. Then it holds
    \begin{align*}
        y(t+r) \leq \left(\frac{a_3}{r} + a_2\right)\mathrm{e}^{a_1} \qquad\text{for all~}t\geq t_0.
    \end{align*}
\end{lemma}

We conclude this section by stating another variant of Gronwall's lemma, which follows from \cite[Lemma~2.5]{Giorgi1999}, see also \cite[Corollary~A.1]{Jiang2015}.

\pagebreak[3]

\begin{lemma}\label{Lemma:Gronwall:2}
    Let $g,h,y$ be three non-negative locally integrable functions on $[t_0,\infty)$ which satisfy, for some $\gamma > 0$,
    \begin{align*}
        \ddt y(t) + \gamma y(t) \leq g(t)y^{\frac12}(t) + h(t) \qquad\text{for a.e.~}t\geq t_0.
    \end{align*}
    Additionally, assume that there exist positive constants $A_1$ and $A_2$ such that
    \begin{align*}
        \int_t^{t+1} g(s)\ds \leq A_1, \quad \int_t^{t+1} h(s)\ds \leq A_2 \qquad\text{for all~}t\geq t_0.
    \end{align*}
    Then it holds that
    \begin{align*}
        y(t) \leq 2y(t_0)\mathrm{e}^{-\gamma(t-t_0)} + Q(\gamma,A_1,A_2),
    \end{align*}
    where
    \begin{align*}
        Q(\gamma,A_1,A_2) = \Big( \frac{\mathrm{e}^{\frac{\gamma}{2}}}{1 - \mathrm{e}^{-\frac{\gamma}{2}}}A_1\Big)^2 + \frac{2 \mathrm{e}^\gamma}{1 - \mathrm{e}^{-\gamma}}A_2.
    \end{align*}
\end{lemma}

\section{Well-posedness and instantaneous regularization}

In the remainder of this paper, we will require the following assumptions on the initial data and the velocity fields:
\begin{enumerate}[label=\textnormal{\bfseries(D\arabic*)}, leftmargin=*]
    \item \label{D1} If $L\in[0,\infty)$, let $m\in\R$ satisfy
    \begin{align*}
        \beta m, \; m \in (-1,1),
    \end{align*}
    while in the case $L = \infty$, let $m = (m_1, m_2)\in\R^2$ be given such that
    \begin{align*}
        m_1,\; m_2 \in (-1,1).
    \end{align*}

    \item \label{D2} The velocity fields have the regularities
    \begin{equation*}
        (\bv,\bw) \in L^2(\R;\mathbf{L}^2_\Div(\Omega) \times \mathbf{L}^2_\mathrm{tan}(\Gamma)) .
    \end{equation*}

    \item \label{D3}
   At least one of the following conditions is fulfilled:
    \begin{enumerate}[label=\textnormal{\bfseries(D3.\arabic*)}, topsep=0.5ex]
        \item\label{D3:1} $\bv$ and $\bw$ have the regularities
        \begin{align*}
        \bv &\in H^1_\mathrm{uloc}(\R;\mathbf{L}^{6/5}(\Om))
            \cap L^\infty(\R;\mathbf{L}^2_\Div(\Om))
            \cap L^2_\mathrm{uloc}(\R;\mathbf{L}^{3}(\Om)),
        \\
        \bw &\in H^1_\mathrm{uloc}(\R;\mathbf{L}^{1+\omega}(\Ga))
            \cap L^\infty(\R;\mathbf{L}^2_\Div(\Ga))
            \cap L^2_\mathrm{uloc}(\R;\mathbf{L}^{2+\omega}(\Ga))
        \end{align*}
        for some $\omega>0$.
        
        \item\label{D3:2} $\bv$ and $\bw$ have the regularities
        \begin{align*}
        \bv &\in L^\infty(\R;\mathbf{L}^2_\Div(\Om))
            \cap L^2_\mathrm{uloc}(\R;\mathbf{H}^1(\Omega)),
        \\
        \bw &\in L^\infty(\R;\mathbf{L}^2_\Div(\Ga))
            \cap L^2_\mathrm{uloc}(\R;\mathbf{H}^1(\Gamma)),
        \end{align*}
        and it holds that
        $K\in [0,\infty]$ and $L\in [0,\infty]$. Additionally, if $K = 0$, it holds
        \begin{align*}
            \bv\vert_\Ga = \bw \qquad\text{a.e.~on~}\Ga\times\R.
        \end{align*}
    \end{enumerate}
\end{enumerate}
The assumptions \ref{D1} and \ref{D2} are imposed to ensure the well-posedness of system~\eqref{conCH*} in a suitable weak sense (see Proposition~\ref{PROP:WS}), and \ref{D3} is required to show that weak solutions instantaneously regularize for
positive times (see Theorem~\ref{thm:highreg}).

\subsection{Well-posedness and energy estimates}

The notion of weak solutions to system~\eqref{conCH*} is defined in the following. Since in the sequel we will deal with a non-autonomous dynamical system, here we consider the initial condition at a generic time $\tau\in \R$, which does not necessarily coincide with $\tau=0$.

\begin{definition}[Weak solution of system \eqref{conCH*}]
    \label{DEF:WS}
    Suppose that the assumptions \ref{ASSUMP:1}--\ref{ASSUMP:2}, \ref{S1}--\ref{S3} and \ref{D1} hold, and let 
    $(\bv,\bw) \in L^2_\mathrm{uloc}(\R;\mathbf{L}^2_\Div(\Omega) \times \mathbf{L}^2_\mathrm{tan}(\Gamma))$ and $(\phi_\tau,\psi_\tau) \in \mathcal{X}_{K,L,m}$ be arbitrary. 
    
    A quadruplet $(\phi,\psi,\mu,\theta)$ is called a \textbf{weak solution of system \eqref{conCH*} on $[\tau,\infty)$} if the following properties hold:
    \begin{enumerate}[label = \textnormal{(\roman*)}, leftmargin=*, topsep=0pt]
        \item \label{DEF:WS:REG}
        The functions $\phi$, $\psi$, $\mu$ and $\theta$ have the regularities
        \begin{subequations}
        \label{WS:REG}
        \begin{align}
            (\phi,\psi) &\in
                H^1_\mathrm{loc}(\tau,\infty;(\mathcal{H}_{L,\beta}^1)')
                \cap L^\infty(\tau,\infty;\mathcal{H}_{K,\alpha}^1),
            \\
            (\mu,\theta) &\in  L^2_\mathrm{loc}(\tau,\infty;\mathcal{H}^1_{L,\beta}), \\
            (F^\prime(\phi),G^\prime(\psi))&\in L^2_{\mathrm{loc}}(\tau,\infty;\mathcal{L}^2),
        \end{align}
        \end{subequations}
        and it holds $|\phi|<1$ a.e.~in $Q_\tau$ and $|\psi|<1$ a.e.~on $\Sigma_\tau$.
        \item \label{DEF:WS:INI}
        The functions $\phi$ and $\psi$ satisfy the initial conditions
        \begin{align}
            \label{INI}
           \quad \phi\vert_{t=\tau} = \phi_\tau
            \quad\text{a.e.~in $\Omega$},
            \quad\text{and}\quad
            \psi\vert_{t=\tau} = \psi_\tau
            \quad\text{a.e.~on $\Gamma$}.
        \end{align}
        \item \label{DEF:WS:WF}
        The functions  $\phi$, $\psi$, $\mu$ and $\theta$ satisfy the weak formulation
        \begin{subequations}
        \label{WF}
        \begin{align}
       \label{WF:1}
        &\begin{aligned}
        &\bigang{(\delt \phi,\delt \psi)}{(\zeta,\xi)}_{\mathcal{H}^1_{L,\beta}}
            - \intO \phi\bv \cdot \Grad\zeta \dx
            - \intG \psi\bw \cdot \Gradg\xi \dG
        \\     
        &\quad= 
        - \intO \Grad \mu \cdot \Grad \zeta \dx
            -\intG \Gradg \theta \cdot \Gradg \xi \dG
            - \chi(L) \intG (\beta\theta-\mu) (\beta\xi-\zeta) \dG,
        \end{aligned}	
        \\[1ex]
        \label{WF:2}
        &\begin{aligned}
        &\intO \mu \, \eta \dx     
            + \intG \theta \, \vartheta \dG 
        \\
        &\quad 
        = \intO \Grad\phi \cdot \Grad \eta 
            + F'(\phi)\eta\dx
            + \intG \Gradg\psi \cdot \Gradg\vartheta 
            + G'(\psi)\vartheta \dG
        \\
        &\qquad
            + \chi(K) \intG (\alpha\psi-\phi) (\alpha\vartheta-\eta) \dG
        \end{aligned}
        \end{align}
        \end{subequations}
        a.e.~on $[\tau,\infty)$ for all test functions $(\zeta,\xi)\in\mathcal{H}^1_{L,\beta}$, $(\eta,\vartheta)\in\mathcal{H}^1_{K,\alpha}$.
    \item \label{DEF:WS:MASS}
    The functions $\phi$ and $\psi$ satisfy the mass conservation law
    \begin{align}
    \label{WS:MASS}
        \left\{
        \begin{aligned}
        &\beta\intO \phi(t) \dx + \intG \psi(t) \dG 
        = \beta\intO \phi_0 \dx + \intG \psi_0 \dG
        &&\quad\text{if $L\in[0,\infty)$,}
        \\
        &\intO \phi(t) \dx = \intO \phi_0 \dx 
        \quad\text{and}\quad
        \intG \psi(t) \dG = \intG \psi_0 \dG
        &&\quad\text{if $L=\infty$,}
        \end{aligned}
        \right.
    \end{align}
    for all $t\in[\tau,\infty)$.
    \item \label{DEF:WS:ENERGY}
    The functions $\phi$, $\psi$, $\mu$ and $\theta$ satisfy the weak energy inequality
        \begin{align}
            \label{WS:ENERGY}
            &E_{\mathrm{free}}\big(\phi(t),\psi(t)\big) + \int_s^t\intO\abs{\Grad \mu}^2 \dxr
             + \int_s^t\intG\abs{\Gradg \theta}^2 \dGr
            \nonumber \\
            &\quad+\chi(L)\int_s^t \intG (\beta\theta-\mu)^2 \dGr
            -\int_s^t\intO \phi\bv \cdot \Grad\mu \dxr
            - \int_s^t\intG \psi\bw \cdot \Gradg\theta \dGr
            \nonumber \\[1ex]
            & \le E_{\mathrm{free}}(\phi(s),\psi(s))
        \end{align}
        for almost all $s\in[\tau,\infty)$ including $s=\tau$ and all $t\in[s,\infty)$.
    \end{enumerate}
\end{definition}

\medskip

In the following lemma, we derive some a priori bounds from the energy inequality \eqref{WS:ENERGY}.

\begin{lemma}
	\label{LEM:UNIFORMBOUNDEDNESS}
	Suppose that \ref{ASSUMP:1}--\ref{ASSUMP:2}, \ref{S1}--\ref{S3}, \ref{D1} are fulfilled, and let 
    $(\bv,\bw) \in L^2_\mathrm{uloc}(\R;\mathbf{L}^2_\Div(\Omega) \times \mathbf{L}^2_\mathrm{tan}(\Gamma))$ and $(\phi_\tau,\psi_\tau) \in \mathcal{X}_{K,L,m}$ be arbitrary.  
	
	Then, for any weak solution $(\phi,\psi,\mu,\theta)$ of system \eqref{conCH*}, the  estimates 
	   \begin{subequations}
		\label{In:UNIFORMBOUNDEDNESS}
		\begin{alignat}{2}
			\label{s0.1}
			E_{\mathrm{free}}\big(\phi(t),\psi(t)\big)
            &\leq E_{\mathrm{free}}(\phi_\tau,\psi_\tau) + \frac12 \int_\tau^t \norm{(\bv(s),\bw(s))}_{\LL^2}^2 \ds,\\
			\label{s0.2}
			\int_\tau^{t}\norm{(\mu(s),\theta(s))}_{L,\beta}^{2}\ds
            &\leq2 E_{\mathrm{free}}(\phi_\tau,\psi_\tau) + \int_\tau^t \norm{(\bv(s),\bw(s))}_{\LL^2}^2 \ds,\\
			\label{s0.3}
			\int_\tau^{t} \norm{(\delt \phi(s),\delt \psi(s)\big)}_{L,\beta,*}^{2}\ds
            &\leq C E_{\mathrm{free}}(\phi_\tau,\psi_\tau) + C\int_\tau^t \norm{(\bv(s),\bw(s))}_{\LL^2}^2 \ds
		\end{alignat}
	\end{subequations}
    hold for all $t\ge \tau$, where $C$ is a constant that depends only on $L, \beta$ and $\Om$.
\end{lemma}

\begin{proof}
First, we recall the weak energy inequality
\begin{align}
	\label{WS:ENERGY2}
	\begin{aligned}
		&E_{\mathrm{free}}\big(\phi(t),\psi(t)\big) + \int_\tau^t\intO\abs{\Grad \mu}^2 \dxs
			+ \int_\tau^t\intG\abs{\Gradg \theta}^2 \dGs\\
		& \quad+\chi(L)\int_\tau^t \intG (\beta\theta-\mu)^2 \dGs		
			\\
		& \le E_{\mathrm{free}}(\phi_\tau,\psi_\tau)	+\int_\tau^t\intO \phi\bv \cdot \Grad\mu \dxs
		+ \int_\tau^t\intG \psi\bw \cdot \Gradg\theta\dGs.
	\end{aligned}
\end{align} 
Since $\abs{\phi}\leq 1$ a.e.~in $Q_\tau$ and $\abs{\psi}\leq 1$ a.e.~on $\Sigma_\tau$, we use Hölder's and Young's inequalities to estimate the terms containing the velocity fields on the right-hand side of \eqref{WS:ENERGY2} as
\begin{align}\label{s2}
	&\Bigabs{\int_\tau^t\intO \phi\bv \cdot \Grad\mu \dx\ds + \int_\tau^t\intG \psi\bw \cdot \Gradg\theta \dG\ds}
    \nonumber\\
	&\quad\leq\int_\tau^t\norm{\bv}_{\mathbf{L}^2(\Om)}\norm{\Grad \mu}_{\mathbf{L}^{2}(\Om)} \ds
    +\int_\tau^t\norm{\bw}_{\mathbf{L}^{2}(\Ga)}\norm{\Gradg\theta}_{\mathbf{L}^2(\Ga)}\ds \\
	&\quad\leq \dfrac{1}{2}\int_\tau^t\norm{\Grad \mu}_{\mathbf{L}^{2}(\Om)}^{2}
    +\norm{\Gradg\theta}_{\mathbf{L}^{2}(\Ga)}^{2}\ds
    +\frac12 \int_\tau^t \norm{(\bv,\bw)}_{\LL^2}^2 \ds. \nonumber
\end{align}
Thus, we readily infer from \eqref{WS:ENERGY2} that
\begin{align*}
    E_{\mathrm{free}}(\phi(t),\psi(t)) \leq E_{\mathrm{free}}(\phi_\tau,\psi_\tau) + \frac12 \int_\tau^t \norm{(\bv,\bw)}_{\LL^2}^2 \ds,
\end{align*}
as well as
\begin{align}\label{est:mt:lb}
    \int_\tau^t \norm{(\mu,\theta)}^2_{L,\beta}\ds \leq 2E_{\mathrm{free}}(\phi_\tau,\psi_\tau) 
    + \int_\tau^t \norm{(\bv,\bw)}_{\LL^2}^2 \ds
\end{align}
for almost every $t \geq \tau$, which proves \eqref{s0.1} and \eqref{s0.2}. Next, in view of the weak formulation \eqref{WF:1}, we find
\begin{align*}
    \abs{\ang{(\delt\phi,\delt\psi)}{(\zeta,\xi)}_{\mathcal{H}^1_{L,\beta}}} \leq \norm{(\bv,\bw)}_{\mathcal{L}^2}\norm{(\zeta,\xi)}_{\mathcal{H}^1} + C\norm{(\mu,\theta)}_{L,\beta}\norm{(\zeta,\xi)}_{\mathcal{H}^1}
\end{align*}
for all $(\zeta,\xi)\in\mathcal{H}^1_{L,\beta}$. Consequently, taking the supremum over all $(\zeta,\xi)\in\mathcal{H}^1_{L,\beta}$ with $\norm{(\zeta,\xi)}_{\mathcal{H}^1}\leq 1$, squaring the resulting inequality and using \eqref{est:mt:lb}, we readily deduce that
\begin{align*}
    \int_\tau^t\norm{(\delt\phi,\delt\psi)}_{L,\beta,\ast}^2\ds 
    &\leq 2\int_\tau^t \norm{(\bv,\bw)}_{\LL^2}^2 \ds + C\int_\tau^t \norm{(\mu,\theta)}^2_{L,\beta}\ds
    \\[1ex]
    &\leq C E_{\mathrm{free}}(\phi_\tau,\psi_\tau) + C\int_\tau^t\norm{(\bv,\bw)}_{\LL^2}^2 \ds
\end{align*}
for almost all $t\geq \tau$, which proves \eqref{s0.3}.
\end{proof}

\medskip

For any final time $T>0$, the well-posedness of system \eqref{conCH*} in the sense of weak solutions was shown in \cite{Knopf2025}. In the following proposition, we extend this result to the case $T=\infty$ .

\begin{proposition}
    \label{PROP:WS}
    Suppose that the assumptions \ref{ASSUMP:1}--\ref{ASSUMP:2}, \ref{S1}--\ref{S3} and \ref{D1} are fulfilled and let 
    $(\bv,\bw) \in L^2_\mathrm{uloc}(\R;\mathbf{L}^2_\Div(\Omega) \times \mathbf{L}^2_\mathrm{tan}(\Gamma))$ and
    $(\phi_\tau,\psi_\tau) \in \mathcal{X}_{K,L,m}$ be arbitrary.   
    
    Then, there exists a unique weak solution of system \eqref{conCH*} on $[\tau,\infty)$ in the sense of Definition~\ref{DEF:WS},
    which has the additional regularities
    \begin{align*}
        (\phi,\psi) &\in H^1(\tau,\infty;(\mathcal{H}^1_{L,\beta})^\prime)\cap BC\big([\tau,\infty);\mathcal{H}^s\big)
        \cap L^2_\mathrm{uloc}(\tau,\infty;\mathcal{W}^{2,6}),
        \\
        (\mu,\theta) &\in L^2_{\mathrm{uloc}}(\tau,\infty;\mathcal{H}^1_{L,\beta}\big), \\
        \big(F'(\phi),G'(\psi)\big) &\in L^2_{\mathrm{uloc}}(\tau,\infty;\LL^6)
    \end{align*}
    for all $s < 1$ if $L=0$ and $s=1$ if $L\in (0,\infty]$. 
    Moreover, the equations
    \begin{subequations}
    \label{EQ:MUTH:STRG}
    \begin{alignat}{2}
        \mu &= -\Lap \phi + F'(\phi) 
        &&\qquad\text{a.e.~in $Q_\tau$},\\
        \theta &= -\Lapg \psi + G'(\psi) + \alpha\deln\phi 
        &&\qquad\text{a.e.~on $\Sigma_\tau$}, \\
        &\begin{cases}
            K\deln\phi = \alpha\psi - \phi &\text{if~}K\in[0,\infty), \\
            \deln\phi = 0 &\text{if~} K = \infty
        \end{cases}
        &&\qquad\text{a.e.~on $\Sigma_\tau$}
    \end{alignat}
    \end{subequations}
    are fulfilled in the strong sense. 
    In particular, there exists a constant $C_*>0$ depending only on the system parameters, $E_\mathrm{free}(\phi_\tau,\psi_\tau)$ and $\norm{(\bv,\bw)}_{L^2_\mathrm{uloc}(\R;\mathbf{L}^2(\Omega)\times \mathbf{L}^2(\Gamma))}$ such that
    \begin{align*}
        &\norm{(\phi,\psi)}_{H^1(\tau,\infty;(\mathcal H^1_{L,\beta})')}
            + \norm{(\phi,\psi)}_{L^\infty(\tau,\infty;\mathcal H^1)}
            + \norm{(\phi,\psi)}_{L^2_\mathrm{uloc}(\tau,\infty;\mathcal W^{2,6})}
        \\
        &\quad + \norm{(F'(\phi),G'(\psi))}_{L^2_\mathrm{uloc}(\tau,\infty;\mathcal L^6)}
            + \norm{(\mu,\theta)}_{L^2_\mathrm{uloc}(\tau,\infty;\mathcal H^1_{L,\beta})}
        \leq C_*.
    \end{align*}
    
    Furthermore, let $(\phi_\tau^1,\psi_\tau^1)$ and $(\phi_\tau^2,\psi_\tau^2)$ be two pairs of admissible initial data satisfying \ref{D1} as well as
    \begin{align}\label{initial-data-mean-value}
        \begin{cases}
            \mean{\phi_\tau^1}{\psi_\tau^1} = \mean{\phi_\tau^2}{\psi_\tau^2} \quad&\text{if~} L\in[0,\infty), \\
            \meano{\phi_\tau^1} = \meano{\phi_\tau^2} \quad\text{and}\quad \meang{\psi_\tau^1} = \meang{\psi_\tau^2} &\text{if~} L=\infty,
        \end{cases}
    \end{align}
    and let 
    \begin{align*}
        (\bv_1,\bw_1) 
        \in L^2_{\mathrm{uloc}}(\tau,\infty;\mathbf{L}^2_\Div(\Om)\times\mathbf{L}^2_\mathrm{tan}(\Ga))
        \quad\text{and}\quad
        (\bv_2,\bw_2) 
        \in L^2_{\mathrm{uloc}}(\tau,\infty;\mathbf{L}^3_\Div(\Om)\times\mathbf{L}^{2+\omega}_\tau(\Ga))
    \end{align*}
    for some $\omega > 0$ be prescribed velocity fields, and let $(\phi_{1},\psi_{1},\mu_{1},\theta_{1})$, $(\phi_{2},\psi_{2},\mu_{2},\theta_{2})$ denote the weak solutions corresponding to the initial data $(\phi_{\tau}^{1},\psi_{\tau}^{1})$, $(\phi_{\tau}^{2},\psi_{\tau}^{2}) $ and the velocity fields $(\bv_{1},\bw_{1})$, $(\bv_{2},\bw_{2})$, respectively. Then, for any $T > 0$, there exists a constant $C > 0$ such that
    \begin{align}\label{est:cont:dep}
        \begin{split}
        	&\bignorm{\big(\phi_{1}(t)-\phi_{2}(t),\psi_{1}(t)-\psi_{2}(t)\big)}^{2}_{L,\beta,*}\\
        	&\quad\leq \bignorm{\big(\phi^{1}_{\tau}-\phi^{2}_{\tau},\psi^{1}_{\tau}-\psi^{2}_{\tau}\big)}^{2}_{L,\beta,*}\exp\Big(C \, \int_{\tau}^{t}\mathcal{F}(r)\dr\Big)\\
        	&\qquad+\int_{\tau}^{t}\bignorm{\big(\bv_{1}(s)-\bv_{2}(s),\bw_{1}(s)-\bw_{2}(s)\big)}_{\mathcal{L}^2}^{2}\exp\Big(C \, \int_{s}^{t}\mathcal{F}(r)\dr \Big)\ds
     \end{split}
    \end{align}
    for all $t \in [\tau,T]$, where $\mathcal{F}\coloneqq 1 + \norm{(\bv_{2},\bw_{2})}^{2}_{\mathbf{L}^3(\Om)\times\mathbf{L}^{2+\omega}(\Ga)}$. Here, the constant $C$ depends only on the system parameters, the final time $T$, and the initial energies $E_{\mathrm{free}}(\phi_\tau^1,\psi_\tau^1)$ and $E_{\mathrm{free}}(\phi_\tau^2,\psi_\tau^2)$.
\end{proposition}

\begin{proof}
    For any $T\in (\tau,\infty)$, the analogous statement was established in \cite[Theorems~3.2 and 3.5]{Giorgini2025} (see also \cite{Knopf2025}). Note that the regularity $(\phi,\psi) \in C\big([\tau,T);\mathcal{H}^s\big)$ directly follows from the Aubin--Lions--Simon lemma. As this holds for arbitrary $T\in (\tau,\infty)$ and the solution is unique on $[\tau,T)$, it can be extended onto the interval $[\tau,\infty)$ and possesses the desired regularities.
\end{proof}

\subsection{Instantaneous regularization}

\begin{theorem}
    \label{thm:highreg}
    Suppose that \ref{ASSUMP:1}--\ref{ASSUMP:2}, \ref{S1}--\ref{S3}, \ref{D1}--\ref{D3} are fulfilled, let  $(\phi_\tau,\psi_\tau) \in \mathcal{X}_{K,L,m}$ be arbitrary, and let $(\phi,\psi,\mu,\theta)$ be the unique weak solution to \eqref{conCH*} in the sense of Definition~\ref{DEF:WS}. Then, for any $\varsigma > \tau$, it holds that
    \begin{align*}
        (\delt\phi,\delt\psi)&\in L^\infty(\varsigma,\infty;(\mathcal{H}^1_{L,\beta})^\prime)\cap L^2_{\mathrm{uloc}}(\varsigma,\infty;\mathcal{H}^1), \\
        (\phi,\psi)&\in L^\infty(\varsigma,\infty;\mathcal{W}^{2,6}), \\
        \mu,\theta)&\in L^\infty(\varsigma,\infty;\mathcal{H}^1_{L,\beta})\cap L^2_{\mathrm{uloc}}(\varsigma,\infty;\mathcal{H}^2), \\
        (F^\prime(\phi),G^\prime(\psi))&\in L^\infty(\varsigma,\infty;\mathcal{L}^6).
    \end{align*}
\end{theorem}

\begin{remark}
    Regarding the regularities established in Theorem~\ref{thm:highreg}, we note the following (see also \cite[Remark 3.8]{Knopf2025}):
    \begin{enumerate}[label=\textnormal{(\alph*)},topsep=0ex, leftmargin=*]
        \item We even get
        \begin{align*}
            G^\prime(\psi)\in L^\infty(\varsigma,\infty;L^s(\Ga)) \qquad\text{~and~}\qquad \psi\in L^\infty(\varsigma,\infty;W^{2,s}(\Ga))
        \end{align*}
        for any $\varsigma > \tau$ and $s\in[1,\infty)$.
        \item If the spatial dimension is two (i.e., $d=2$), we also obtain
        \begin{alignat*}{2}
            (F^\prime(\phi), G^\prime(\psi))\in L^\infty(\varsigma,\infty;\mathcal{L}^r), \qquad (\phi,\psi)\in L^\infty(\varsigma,\infty;\mathcal{W}^{2,r})
        \end{alignat*}
        for any $\varsigma > \tau$ and $r\in[1,\infty)$.
    \end{enumerate}
\end{remark}

\begin{proof}[Proof of Theorem~\ref{thm:highreg}]
    In this proof, we deal with the cases corresponding to \ref{D3:1} or \ref{D3:2} separately since both cases require different approximation techniques.
    Let $\varsigma>\tau$ be arbitrary. In the following, let $C$ denote a generic positive constant depending only on the prescribed velocity fields, the quantities specified in \ref{ASSUMP:1} and \ref{ASSUMP:2}, $F$, $G$, $\tau$, $\varsigma$, and $\norm{(\phi_\tau,\psi_\tau)}_{\HH^1}$, which may change its value from line to line.

    \textbf{Case 1:} We assume that the velocity fields $(\bv,\bw)$ satisfy \ref{D3:1}. \\
    For any $t\ge \tau$ and $h>0$, we write
    \begin{align*}
    	\partial_{t}^{h}f(t)= \frac 1h \big(f(t+h)-f(t)\big)
    \end{align*}
    to denote the forward-in-time difference quotient of a function $f$ at time $t$.	

    Now, let $h\in (0,1)$ be arbitrary.
	From \eqref{WF:1}, we deduce that
	\begin{align}
		\big\langle(\partial_{t}^{h}\delt \phi,\partial_{t}^{h}\delt \psi),(\zeta,\xi) \big\rangle_{\mathcal{H}^1_{L,\beta}}
		&= \big(\partial_{t }^{h}\phi\,\bv(\cdot+h),\Grad\zeta\big)_{L^{2}(\Om)}+\big(\partial_{t }^{h}\bv\,\phi,\Grad\zeta\big)_{L^{2}(\Om)}\nonumber\\
		 &\quad +\big(\partial_{t }^{h}\psi\,\bw(\cdot+h),\Gradg\xi\big)_{L^{2}(\Ga)}+\big(\partial_{t }^{h}\bw\,\psi,\Gradg\xi\big)_{L^{2}(\Ga)}\nonumber\\
		 &\quad -\big(\big(\partial_{t }^{h}\mu,\partial_{t }^{h}\theta\big),(\zeta,\xi) \big)_{L,\beta} \label{sumweak2} 
	\end{align} 
	for every $(\zeta,\xi)\in \mathcal{H}^1_{L,\beta}$, where 
    \begin{subequations}
    	\begin{alignat}{3}
    		&\partial_{t }^{h}\mu 
    		=-\Delta\partial_{t }^{h}\phi+\partial_{t }^{h}F^{\prime}(\phi)\label{mu} 
    		&&\quad\text{a.e.~in~}Q_\tau, \\
    		&\partial_{t }^{h}\theta
    		=-\Delta_{\Gamma}\partial_{t}^{h}\psi+\partial_{t }^{h}G^{\prime}(\psi)
    		+\alpha\partial_{n }\partial_{t }^{h}\phi \label{teta}
    		&&\quad\text{a.e.~on~}\Sigma_\tau,\\
    		&\begin{cases}
    			K \deln\partial_{t }^{h} \phi = \alpha \partial_{t }^{h}\psi - \partial_{t }^{h}\phi &\text{if~}\; K\in [0,\infty),\\
    			\deln \partial_{t }^{h}\phi = 0 &\text{if~}\; K=\infty
    		\end{cases}\label{muteta}
    		&&\quad\text{a.e.~on~}\Sigma_\tau,
    	\end{alignat}
    \end{subequations}
	due to \eqref{EQ:MUTH:STRG}. 
	Now, testing \eqref{sumweak2} with the solution operator $\mathcal{S}_{L,\beta}\big(\partial_{t }^{h}\phi,\partial_{t }^{h}\psi\big)$ and exploiting the fact that
	\begin{align*}
	   \big( \big(\partial_{t }^{h}\mu,\partial_{t }^{h}\theta\big),\mathcal{S}_{L,\beta}\big(\partial_{t }^{h}\phi,\partial_{t }^{h}\psi\big)\big)_{L,\beta}=-\big( \big(\partial_{t }^{h}\mu,\partial_{t }^{h}\theta\big),\big(\partial_{t }^{h}\phi,\partial_{t }^{h}\psi\big)\big)_{\mathcal{L}^{2}}
	\end{align*}
	and
	\begin{align*}
        \mathcal{S}_{L,\beta}(\partial_{t }^{h}\phi,\partial_{t }^{h}\psi) = \bigscp{\mathcal{S}_{L,\beta}^\Om(\partial_{t }^{h}\phi,\partial_{t }^{h}\psi)}{\mathcal{S}_{L,\beta}^\Ga(\partial_{t }^{h}\phi,\partial_{t }^{h}\psi)}\in\mathcal{V}^1_{L,\beta},
	\end{align*} 
    we have
	\begin{align*}
		&\dfrac{1}{2}\ddt\big\| \big(\partial_{t }^{h}\phi,\partial_{t }^{h}\psi \big)\big\|_{L,\beta,*}^{2}
		+\big( \big(\partial_{t }^{h}\mu,\partial_{t }^{h}\theta\big),\big(\partial_{t }^{h}\phi,\partial_{t }^{h}\psi\big)\big)_{\mathcal{L}^{2}}\\
		&\quad=-\big(\partial_{t}^{h}\phi\,\bv(\cdot+h),\Grad\mathcal{S}_{L,\beta}^\Om(\partial_{t }^{h}\phi,\partial_{t }^{h}\psi) \big)_{L^{2}(\Om)}-\big(\partial_{t }^{h}\bv\,\phi,\Grad\mathcal{S}_{L,\beta}^\Om(\partial_{t }^{h}\phi,\partial_{t }^{h}\psi)\big)_{L^{2}(\Om)}\nonumber\\
		&\quad\quad-\big(\partial_{t }^{h}\psi\,\bw(\cdot+h),\Gradg\mathcal{S}_{L,\beta}^\Ga(\partial_{t }^{h}\phi,\partial_{t }^{h}\psi)\big)_{L^{2}(\Ga)}-\big(\partial_{t }^{h}\bw\,\psi,\Gradg\mathcal{S}_{L,\beta}^\Ga(\partial_{t }^{h}\phi,\partial_{t }^{h}\psi)\big)_{L^{2}(\Ga)}\\
		&\quad\eqqcolon\sum_{k=1}^{4}\Pi_{k}.
	\end{align*} 
    Then, from \eqref{mu} and \eqref{teta} it follows that
	\begin{align*}
		&\dfrac{1}{2}\ddt\big\|\big(\partial_{t }^{h}\phi,\partial_{t }^{h}\psi \big) \big\|_{L,\beta,*}^{2} 
		+\big\|(\partial_{t }^{h}\phi,\partial_{t }^{h}\psi) \big\|^{2} _{K,\alpha} +\intO \partial_{t }^{h}[F_{1}^{\prime}(\phi)]\partial_{t }^{h}\phi\dx +\intG \partial_{t }^{h}[G_{1}^{\prime}(\psi)]\partial_{t }^{h}\psi\dG
		\\
		&\quad=-\intO\partial_{t }^{h}[F_{2}^{\prime}(\phi)]\partial_{t }^{h}\phi \dx-\intG\partial_{t }^{h}[G_{2}^{\prime}(\psi)]\partial_{t }^{h}\psi \dG+\sum_{k=1}^{4}\Pi_{k}.
	\end{align*}
	Now, recalling that $ F_{2} $, $ G_{2} $ are Lipschitz continuous and $ F_{1}^\prime $ and $ G_{1}^\prime$ are monotone, we obtain that
	\begin{align}\label{est0}
		\dfrac{1}{2}\ddt\big\|\big(\partial_{t }^{h}\phi,\partial_{t }^{h}\psi \big) \big\|_{L,\beta,*}^{2} 
		+\big\|(\partial_{t }^{h}\phi,\partial_{t }^{h}\psi) \big\|^{2} _{K,\alpha} \leq C\big\|(\partial_{t }^{h}\phi,\partial_{t }^{h}\psi) \big\|^{2} _{\mathcal{L}^{2}}
		 +\sum_{k=1}^{4}\Pi_{k}.
	\end{align} 
    If $K\in[0,\infty)$, we use the bulk-surface Poincar\'{e} inequality combined with Ehrling's Lemma for the first term on the right-hand side of \eqref{est0} to obtain
	\begin{align}\label{es01}
	   \big\|(\partial_{t }^{h}\phi,\partial_{t }^{h}\psi) \big\|^{2} _{\mathcal{L}^{2}}\leq \frac 14 \big\|(\partial_{t }^{h}\phi,\partial_{t }^{h}\psi) \big\|^{2} _{K,\alpha}+ C\norm{(\partial_{t}^{h}\phi,\partial_{t}^{h}\psi)}_{L,\beta,*}^{2}.
	\end{align} 
    If $K = \infty$, we first use Ehrling's Lemma to obtain
    \begin{align*}
        \norm{(\delth\phi,\delth\psi)}_{\mathcal{L}^2}^2 
        &\leq \frac{1}{5}\norm{(\delth\phi,\delth\psi)}_{\mathcal{H}^1}^2 + C\norm{(\delth\phi,\delth\psi)}_{L,\beta,\ast}^2 
        \\
        &\leq \frac{1}{5}\norm{(\delth\phi,\delth\psi)}_{\mathcal{L}^2}^2 
        + \frac{1}{5}\norm{(\delth\Grad\phi,\delth\Gradg\psi)}_{\mathcal{L}^2}^2 + C\norm{(\delth\phi,\delth\psi)}_{L,\beta,\ast}^2.
    \end{align*}
    Consequently, absorbing the corresponding term by the left-hand side, we deduce that
    \begin{align}
        \label{es01*}
        \begin{split}
        \norm{(\delth\phi,\delth\psi)}_{\mathcal{L}^2}^2 
        &\leq \frac{1}{4}\norm{(\delth\Grad\phi,\delth\Gradg\psi)}_{\mathcal{L}^2}^2
        + C\norm{(\delth\phi,\delth\psi)}_{L,\beta,\ast}^2
        \\
        &\le \frac{1}{4}\norm{(\delth\phi,\delth\psi)}_{K,\alpha}^2
        + C\norm{(\delth\phi,\delth\psi)}_{L,\beta,\ast}^2.
        \end{split}
    \end{align}
    Now, we estimate the terms $\Pi_{k}$ on the right-hand side of $\eqref{est0}$. By using Sobolev embedding theorems, Young's inequality, and bulk-surface Poincar\'{e}'s inequality (Lemma~\ref{lemma:poinc}), we infer that
	\begin{align}
        \label{es02}
		\abs{\Pi_{1}}+\abs{\Pi_{3}}&\leq \norm{\nabla\partial_{t}^{h}\phi}_{\mathbf{L}^{2}(\Om)}\norm{\bv(\cdot+h)}_{\mathbf{L}^{3}(\Om)}\norm{\mathcal{S}_{L,\beta}^{\Om}(\partial_{t}^{h}\phi,\partial_{t}^{h}\psi)}_{L^{6}(\Om)}
         \nonumber\\
		&\quad+\norm{\nabla_{\Ga}\partial_{t}^{h}\psi}_{\mathbf{L}^{2}(\Ga)}\norm{\bw(\cdot+h)}_{\mathbf{L}^{2+\omega}(\Ga)}\norm{\mathcal{S}_{L,\beta}^{\Ga}(\partial_{t}^{h}\phi,\partial_{t}^{h}\psi)}_{L^{\frac{4+2\omega}{\omega}}(\Ga)}
        \\ \nonumber
		&\leq\dfrac{1}{4}\big\|(\partial_{t }^{h}\phi,\partial_{t }^{h}\psi) \big\|^{2} _{K,\alpha}+C\norm{(\bv(\cdot + h),\bw(\cdot + h))}_{\mathbf{L}^{3}(\Om)\times \mathbf{L}^{2+\omega}(\Ga)}^{2}\norm{(\partial_{t}^{h}\phi,\partial_{t}^{h}\psi)}_{L,\beta,*}^{2},
        \\[1ex]
        \label{es03}
        \begin{split}
		\abs{\Pi_{2}}+\abs{\Pi_{4}}&\leq \norm{\nabla\phi}_{\mathbf{L}^{\infty}(\Om)}\norm{\partial_{t}^{h}\bv}_{\mathbf{L}^{6/5}(\Om)}\norm{\mathcal{S}_{L,\beta}^{\Om}(\partial_{t}^{h}\phi,\partial_{t}^{h}\psi)}_{L^{6}(\Om)}
        \\
		&\quad+\norm{\nabla_{\Gamma}\psi}_{\mathbf{L}^{\infty}(\Ga)}\norm{\partial_{t}^{h}\bw}_{\mathbf{L}^{1+\omega}(\Ga)}\norm{\mathcal{S}_{L,\beta}^{\Ga}(\partial_{t}^{h}\phi,\partial_{t}^{h}\psi)}_{L^{\frac{1+\omega}{\omega}}(\Ga)}
        \\
		&\leq \norm{(\partial_{t}^{h}\bv,\partial_{t}^{h}\bw)}_{\mathcal{X}^{\omega}}^{2}+C\norm{(\phi,\psi)}_{\mathcal{W}^{1,\infty}}^{2}\norm{(\partial_{t}^{h}\phi,\partial_{t}^{h}\psi)}_{L,\beta,*}^{2}.
        \end{split}
	\end{align}
	Here, $\omega$ is the quantity introduced in \ref{D3:1}, and
    $\mathcal{X}^{\omega}\coloneqq \mathbf{L}^{\frac65}(\Om)\times \mathbf{L}^{1+\omega}(\Ga)$.

	Considering \eqref{es01} and the above estimates in \eqref{est0} and using Young's inequality, we deduce that
	\begin{align}\label{est02}
        \begin{split}
		&\ddt\norm{(\partial_{t }^{h}\phi,\partial_{t }^{h}\psi)}_{L,\beta,*}^{2} + \frac12\norm{(\delth\phi,\delth\psi)}_{K,\alpha}^2
        \\
        &\quad
		\leq2\norm{(\partial_{t}^{h}\bv,\partial_{t}^{h}\bw)}_{\mathcal{X}^{\omega}}^{2}+ \Upsilon\big\|\big(\partial_{t }^{h}\phi,\partial_{t }^{h}\psi \big) \big\|_{L,\beta,*}^{2} \,,
        \end{split}
	\end{align}
    where 
    \begin{equation*}
        \Upsilon\coloneqq C\Big(1+\norm{(\bv(\cdot + h),\bw(\cdot + h))}_{\mathbf{L}^{3}(\Om)\times \mathbf{L}^{2+\omega}(\Ga)}^{2}+\norm{(\phi,\psi)}_{\mathcal{W}^{1,\infty}}^{2}\Big).
    \end{equation*}

    Now, we fix an arbitrary $r\in (0,1]$.
    To apply the Uniform Gronwall Lemma (Lemma~\ref{Lemma:Gronwall}) with $t_0 \coloneqq \tau$, we need to show that there exist positive constants $a_1,a_2,a_3>0$ such that 
    \begin{align}
        \label{EST:GW:A1}
        \int_t^{t + r} \Upsilon(s) \ds &\le a_1,
        \\
        \label{EST:GW:A2}
        \int_t^{t+r} 2\norm{(\delth\bv(s),\delth\bw(s))}_{\mathcal{X}^\omega}^2 \ds 
        &\le a_2,
        \\
        \label{EST:GW:A3}
        \int_t^{t+r} \norm{(\delth\phi(s),\delth\psi(s))}_{L,\beta,\ast}^2\ds &\le a_3
    \end{align}
    for every $t\ge \tau$.

    To this end, let $t\ge \tau$ be arbitrary. 
    Using a standard estimate for difference quotients, which can be found in \cite[p.~317]{Fukao2021} (see also \cite[Lemma~7.23]{GilbargTrudinger2001}), we obtain
    \begin{align*}
        \int_t^{t+r} \norm{(\delth\phi(s),\delth\psi(s))}_{L,\beta,\ast}^2\ds 
        &\leq \int_t^{t+r+h}\norm{(\delt\phi(s),\delt\psi(s))}_{L,\beta,\ast}^2\ds.
    \end{align*} 
    Hence, applying Lemma~\ref{LEM:UNIFORMBOUNDEDNESS} (with $t+2$ instead of $t$ and $t$ instead of $\tau$), and recalling \ref{D2} and $h\in (0,1)$, we find that
    \begin{align*}
        &\int_t^{t+r} \norm{(\delth\phi(s),\delth\psi(s))}_{L,\beta,\ast}^2\ds
        \le \int_\tau^{t+2}\norm{(\delt\phi(s),\delt\psi(s))}_{L,\beta,\ast}^2\ds
        \\
        &\quad 
        \le C E_{\mathrm{free}}(\phi_\tau,\psi_\tau) 
            + C\int_\tau^{t+2} \norm{(\bv(s),\bw(s))}_{\LL^2}^2 \ds
        \\[1ex]
        &\quad
        \le C E_{\mathrm{free}}(\phi_\tau,\psi_\tau) 
            + C\, \norm{(\bv,\bw)}_{L^2(\R;\LL^2)}^2 
        \eqqcolon a_3. 
    \end{align*}
    This verifies \eqref{EST:GW:A3}.
    Proceeding in a similar manner, we obtain
    \begin{align*}
        &\int_t^{t+r} 2\norm{(\delth\bv(s),\delth\bw(s))}_{\mathcal{X}^\omega}^2 \ds 
        \leq C \int_t^{t+r+h}\norm{(\delt\bv(s),\delt\bw(s))}_{\mathcal{X}^\omega}^2 \ds 
        \\[1ex]
        &\quad
        \le C\, \norm{(\delt\bv,\delt\bw)}_{L^2_\mathrm{uloc}(\R;\mathcal{X}^\omega)}^2 \eqqcolon a_2,
    \end{align*}
    which verifies \eqref{EST:GW:A2}.
    Moreover, recalling $r\le 1$, \ref{D3:1} and $(\phi,\psi)\in L^2_{\mathrm{uloc}}(\tau,\infty;\mathcal{W}^{2,6})\emb L^2_{\mathrm{uloc}}(\tau,\infty;\mathcal{W}^{1,\infty})$, we deduce that
    \begin{align*}
        \int_t^{t + r} C\Upsilon(s) \ds 
        &\le Cr 
            + C\int_{t+h}^{t+h+1} \norm{(\bv(s),\bw(s))}_{\mathbf{L}^{3}(\Om)\times \mathbf{L}^{2+\omega}(\Ga)}^{2} \ds
        \\
        &\qquad
            + C\int_t^{t+1} \norm{(\phi(s),\psi(s))}_{\mathcal{W}^{1,\infty}}^{2} \ds
        \\[1ex]
        &\le C\norm{(\bv,\bw)}_{L^2_\mathrm{uloc}(\R;\mathbf{L}^3(\Omega)\times \mathbf{L}^{2+\omega}(\Gamma))}^2 
            + C\norm{(\phi,\psi)}_{L^2_\mathrm{uloc}(\R;\WW^{2,6})}^2
        \eqqcolon a_1.
    \end{align*}
    This means that \eqref{EST:GW:A1} is fulfilled.

    Now, using the Uniform Gronwall Lemma (Lemma~\ref{Lemma:Gronwall}) with $t_0 \coloneqq \tau$, we deduce that
    \begin{align*}
        \norm{(\partial_{t }^{h}\phi(t+r),\partial_{t }^{h}\psi(t+r))}_{L,\beta,*}^{2}
        \leq \left(\frac{a_3}{r} + a_2\right)\mathrm{e}^{a_1} \qquad\text{for all~}t\geq \tau.
    \end{align*}
    Consequently, for any $s\ge \varsigma$, choosing $r\coloneqq \min\{1,\varsigma-\tau\}$ and $t=s-r$ yields
    \begin{align}
        \label{EST:GW:4}
        \norm{(\partial_{t }^{h}\phi(s),\partial_{t }^{h}\psi(s))}_{L,\beta,*}^{2}
        \leq \left(\frac{a_3}{\min\{1,\varsigma-\tau\}} + a_2\right)\mathrm{e}^{a_1} 
        \le C.
    \end{align}
    Integrating \eqref{est02} with respect to time from an arbitrary $t\ge \varsigma$ to $t+1$ , and using the estimates \eqref{EST:GW:A1}, \eqref{EST:GW:A2} and \eqref{EST:GW:4}, we get 
    \begin{align*}
        \int_t^{t+1}\norm{(\delth\phi(s),\delth\psi(s))}_{K,\alpha}^2\ds 
        \le C.
    \end{align*}
    Consequently, we have
    \begin{align*}
        \norm{(\delth\phi,\delth\psi)}_{L^\infty(\varsigma,\infty;(\mathcal{H}^1_{L,\beta})^\prime)} + \norm{(\delth\phi,\delth\psi)}_{L^2_{\mathrm{uloc}}(\varsigma,\infty;\mathcal{H}^1)} 
        \leq C.
    \end{align*}
    Sending $h\rightarrow 0$, we finally conclude that
    \begin{align*}
        (\delt\phi,\delt\psi)\in L^\infty(\varsigma,\infty;(\mathcal{H}^1_{L,\beta})^\prime)\cap L^2_{\mathrm{uloc}}(\varsigma,\infty;\mathcal{H}^1).
    \end{align*}
    The remaining regularities can be established by proceeding as in \cite{Knopf2025} (see also \cite{Giorgini2025}).
    
    \textbf{Case 2:} Now, we assume that the velocity fields $(\bv,\bw)$ satisfy \ref{D3:2}. \\
    We introduce families of regular potentials $(F_\lambda)_{\lambda>0}$ and $(G_\lambda)_{\lambda > 0}$ to approximate the singular potentials $F$ and $G$, respectively. To be precise, for $\lambda > 0$, we define
    \begin{align*}
        F_\lambda(s) = F_{1,\lambda}(s) + F_2(s), \qquad G_\lambda(s) = G_{1,\lambda}(s) + G_2(s) \qquad\text{for all~}s\in\R,
    \end{align*}
    where $F_{1,\lambda}$ and $G_{1,\lambda}$ denote, for any $\lambda\in(0,\infty)$, the Moreau--Yosida approximations of $F_1$ and $G_1$, respectively. These are defined as
    \begin{align*}
        F_{1,\lambda}(s) = \inf_{r\in\R}\left\{\frac{1}{2\lambda}\abs{s-r}^2 + F_1(r)\right\}, \qquad G_{1,\lambda}(s) = \inf_{r\in\R}\left\{\frac{1}{2\lambda}\abs{s-r}^2 + G_1(r)\right\}
    \end{align*}
    for all $s\in\R$, respectively. Then, the derivatives are given as
    \begin{align*}
        F_{1,\lambda}^\prime(s) = \frac1\lambda\Big( s - (I + \lambda F_1^\prime)^{-1}(s)\Big), \qquad G_{1,\lambda}^\prime(s) = \frac1\lambda\Big(s - (I + \lambda G_1^\prime)^{-1}(s)\Big)
    \end{align*}
    for all $s\in\R$.
    For any $\lambda>0$, $F_\lambda$ has the following important properties:
    \begin{enumerate}[label=\textbf{(M\arabic*)},topsep=0ex,leftmargin=*]
        \item \label{Yosida:Reg} For every $\lambda > 0$, $F_{1,\lambda}\in C^{2,1}_{\textup{loc}}(\R)$ with $F_{1,\lambda}(0) = F_{1,\lambda}^\prime(0) = 0$;
        \item \label{Yosida:Convexity}  $F_{1,\lambda}$ is convex with
        \begin{align*}
            F_{1,\lambda}^{\prime\prime}(s) \geq \frac{\Theta_\Om}{1 + \Theta_\Om} \quad\text{for all~}s\in\R;
        \end{align*}
        \item \label{Yosida:Lipschitz}  For every $\lambda > 0$, $F_{1,\lambda}^\prime$ is Lipschitz continuous on $\R$ with constant $\frac1\lambda$.
        \item \label{Yosida:Growth} There exists $\overline\lambda\in(0,1)$ and $C > 0$ such that
        \begin{align*}
            F_{1,\lambda}(s) \geq \frac{1}{4\overline\lambda}s^2 - C \qquad\text{for all~}s\in\R\text{~and~}\lambda\in(0,\overline\lambda);
        \end{align*}
        \item \label{Yosida:Convergence}  
        For every $\lambda > 0$, it holds $F_{1,\lambda}(s)\le  F_1(s)$ for all $s\in[-1,1]$.
        Moreover, as $\lambda \rightarrow 0$, we have $F_{1,\lambda}(s)\rightarrow F_1(s)$ for all $s\in[-1,1]$ and $\vert F_{1,\lambda}^\prime(s)\vert \rightarrow \abs{F_1^\prime(s)}$ for $s\in(-1,1)$.
    \end{enumerate}
    Analogous properties hold true for $G_{1,\lambda}$ instead of $F_{1,\lambda}$.
    Additionally, as a consequence of the domination property \eqref{DominationProperty}, 
    we obtain an analogous estimate for the corresponding Yosida regularizations.
    Namely, for any $\lambda>0$, it holds that
    \begin{align}
        \label{DOMINATION:YOSIDA}
        \abs{F_{1,\lambda}^\prime(\alpha s)} \leq \kappa_1\abs{G_{1,\lambda}^\prime(s)} + \kappa_2 \qquad\text{for all~}s\in\R
    \end{align}
    (see, e.g., \cite[Lemma 4.4]{Calatroni2013}), where $\kappa_1, \kappa_2$ are the same constants as in \eqref{DominationProperty}.
    Besides, using standard mollification, we find a sequence $\{(\bv_\lambda,\bw_\lambda)\}_{\lambda\in(0,1)}\subset C_c^\infty(\R;\mathcal{H}^1\cap\BLL^2_\Div)$ such that 
    \begin{align}\label{Conv:vw:lambda}
        \begin{split}
            (\bv_\lambda,\bw_\lambda) \rightarrow (\bv,\bw) &\qquad\text{strongly in~} L^2_{\mathrm{uloc}}(\R;\mathcal{H}^1), \\
            &\qquad\text{weakly-$\ast$ in~} L^\infty(\R;\BLL^2),
        \end{split}
    \end{align}
    as $\lambda\rightarrow0$. For more details, we refer to \cite[S.~4.13--4.15]{Alt2016}.
    In particular, according to \cite[S.~4.13 (2)]{Alt2016}, we have 
    \begin{align}
        \label{Est:vw:lambda}
        \norm{(\bv_\lambda,\bw_\lambda)}_{L^2_\mathrm{uloc}(\R;\mathcal{H}^1)} 
        &\leq \norm{(\bv,\bw)}_{L^2_\mathrm{uloc}(\R;\mathcal{H}^1)} 
    \end{align}
    for all $\lambda > 0$. 
    According to Proposition~\ref{PROP:WS}, for every $\lambda\in (0,1)$, there exists a unique solution $(\phi_\lambda,\psi_\lambda,\mu_\lambda,\theta_\lambda)$ of \eqref{conCH*} in the sense of Definition~\ref{DEF:WS} where $(F,G)$ and $(\bv,\bw)$ are replaced by $(F_\lambda, G_\lambda)$ and $(\bv_\lambda,\bw_\lambda)$, respectively. 
    Arguing now for instance as in~\cite{Giorgini2025}, one can show that 
    \begin{align*}
        &\norm{(\delt\phi_\lambda,\delt\psi_\lambda)}_{L^2(\tau,\infty;(\mathcal{H}^1_{L,\beta})^\prime)} + \norm{(\phi_\lambda,\psi_\lambda)}_{L^\infty(\tau,\infty;\mathcal{H}^1)\cap L^2_{\mathrm{uloc}}(\tau,\infty;\mathcal{W}^{2,6})} \\
        &\qquad + \norm{(\mu_\lambda,\theta_\lambda)}_{L^2_{\mathrm{uloc}}(\tau,\infty;\mathcal{H}^1)} + \norm{(F_\lambda^\prime(\phi_\lambda),G_\lambda^\prime(\psi_\lambda))}_{L^2_{\mathrm{uloc}}(\tau,\infty;\mathcal{L}^6)} \leq C.
    \end{align*}
    Thus, employing the Banach--Alaoglu theorem and the Aubin--Lions--Simon lemma, we deduce that there exists a quadruplet $(\phi_\ast,\psi_\ast,\mu_\ast,\theta_\ast)$ such that, for any $T\in(0,\infty)$,
    \begin{alignat*}{2}
        (\delt\phi_\lambda,\delt\psi_\lambda) &\rightarrow (\delt\phi_\ast,\delt\psi_\ast) &&\qquad\text{weakly in~} L^2(\tau,T;(\mathcal{H}^1_{L,\beta})^\prime), \\
        (\phi_\lambda,\psi_\lambda) &\rightarrow (\phi_\ast,\psi_\ast) &&\qquad\text{weakly-star in~} L^\infty(\tau,T;\mathcal{H}^1_{K,\alpha}), \\
        & &&\qquad\text{stongly in~} C([\tau,T];\mathcal{H}^s) \text{~for any~}s\in[0,1), \\
        (\mu_\lambda,\theta_\lambda) &\rightarrow (\mu_\ast,\theta_\ast) &&\qquad\text{weakly in~} L^2(\tau,T;\mathcal{H}^1_{L,\beta}), \\
        (F_\lambda^\prime(\phi_\lambda),G_\lambda^\prime(\psi_\lambda)) &\rightarrow (F^\prime(\phi_\ast), G^\prime(\psi_\ast)) &&\qquad\text{weakly in~} L^2(\tau,T;\mathcal{L}^2)
    \end{alignat*}
    as $\lambda\rightarrow 0$, up to a subsequence extraction. It is then straightforward to show that $(\phi_\ast,\psi_\ast,\mu_\ast,\theta_\ast)$ is a weak solution to \eqref{conCH*} in the sense of Definition~\ref{DEF:WS}, and thus, by uniqueness, $(\phi_\ast,\psi_\ast,\mu_\ast,\theta_\ast) = (\phi,\psi,\mu,\theta)$.
    Then, slightly modifying the proof in Case~1, we deduce that
    \begin{align*}
        (\delt\phi_\lambda,\delt\psi_\lambda)&\in L^\infty(\varsigma,\infty;(\mathcal{H}^1_{L,\beta})^\prime)\cap L^2_{\mathrm{uloc}}(\varsigma,\infty;\mathcal{H}^1), \\
        (\phi_\lambda,\psi_\lambda)&\in L^\infty(\varsigma,\infty;\mathcal{W}^{2,6}), \\
        (\mu_\lambda,\theta_\lambda)&\in L^\infty(\varsigma,\infty;\mathcal{H}^1_{L,\beta})\cap  L^2_{\mathrm{uloc}}(\tau,\infty;\mathcal{H}^2), \\
        (F_\lambda^\prime(\phi_\lambda),G_\lambda^\prime(\psi_\lambda))&\in L^\infty(\varsigma,\infty;\mathcal{L}^6)
    \end{align*}
    for any $\varsigma > \tau$.
    Next, following \cite{Giorgini2025}, we infer that
    \begin{align}
        \label{EST:GW2:1}
        \begin{split}
            &\ddt\norm{(\mu_\lambda,\theta_\lambda)}_{L,\beta}^2 
                + \frac12\norm{(\delt\phi_\lambda,\delt\psi_\lambda)}_{K,\alpha}^2 
            \\
            &\qquad\leq C\big(\norm{(\bv_\lambda,\bw_\lambda)}_{\mathcal{H}^1}^2 
                + \norm{(\mu_\lambda,\theta_\lambda)}_{L,\beta}^2\big) 
                + C\norm{(\bv_\lambda,\bw_\lambda)}_{\mathcal{H}^1}^2
                    \norm{(\mu_\lambda,\theta_\lambda)}_{L,\beta}^2
            \\
            &\qquad\leq C\big( 1 + \norm{(\bv_\lambda,\bw_\lambda)}_{\mathcal{H}^1}^2 \big) 
                + C\big( 1 + \norm{(\bv_\lambda,\bw_\lambda)}_{\mathcal{H}^1}^2 \big)
                    \norm{(\mu_\lambda,\theta_\lambda)}_{L,\beta}^2
        \end{split}
    \end{align}
    on $[\tau,\infty)$.

    Now, we fix an arbitrary $r\in (0,1]$, and we intend to 
    apply the Uniform Gronwall Lemma (Lemma~\ref{Lemma:Gronwall}) with $t_0 \coloneqq \tau$.
    This requires showing that there exist positive constants $a_1,a_2,a_3>0$ such that 
    \begin{align}
        \label{EST:GW:A12*}
        \int_t^{t + r} C\big( 1 + \norm{(\bv_\lambda(s),\bw_\lambda(s))}_{\mathcal{H}^1}^2 \big) \ds 
        &\le a_1 = a_2,
        \\
        \label{EST:GW:A3*}
        \int_t^{t+r} \norm{(\mu_\lambda(s),\theta_\lambda(s))}_{L,\beta}^2 \ds 
        &\le a_3
    \end{align}
    for every $t\ge \tau$.

    In view of \eqref{Est:vw:lambda}, we obviously have
    \begin{align*}
        &\int_t^{t + r} C\big( 1 + \norm{(\bv_\lambda(s),\bw_\lambda(s))}_{\mathcal{H}^1}^2 \big) \ds 
        \le C + C \int_t^{t + 1} \norm{(\bv_\lambda(s),\bw_\lambda(s))}_{\mathcal{H}^1}^2  \ds
        \\[1ex]
        &\quad\le C + C \norm{(\bv,\bw)}_{L^2_\mathrm{uloc}(\R;\mathcal{H}^1)}^2
        \eqqcolon a_1 \eqqcolon a_2.
    \end{align*}
    This verifies \eqref{EST:GW:A12*}.
    Moreover, applying Lemma~\ref{LEM:UNIFORMBOUNDEDNESS}, and 
    using \ref{Yosida:Convergence} (and the analogous statement for $G_{1,\lambda}$ and $G_1$) along with \eqref{EST:GW:A12*}, we deduce that
    \begin{align*}
        &\int_t^{t+r} \norm{(\mu_\lambda(s),\theta_\lambda(s))}_{L,\beta}^2 \ds
        \le 2 E_{\mathrm{free}}(\phi_\tau,\psi_\tau) 
            + \int_\tau^{t+1} \norm{(\bv_\tau(s),\bw_\tau(s))}_{\LL^2}^2 \ds \ds
        \\[1ex]
        &\quad
        \le 2 E_{\mathrm{free}}(\phi_\tau,\psi_\tau) + C\, \norm{(\bv,\bw)}_{L^2(\R;\LL^2)}^2 \eqqcolon a_3.
    \end{align*}
    Now, using the Uniform Gronwall Lemma (Lemma~\ref{Lemma:Gronwall}) with $t_0 \coloneqq \tau$, we deduce that
    \begin{align*}
        \norm{(\mu_\lambda(t+r),\theta_\lambda(t+r))}_{L,\beta}
        \leq \left(\frac{a_3}{r} + a_2\right)\mathrm{e}^{a_1} \qquad\text{for all~}t\geq \tau.
    \end{align*}
    As in Case~1, we infer that
    \begin{align}
        \label{EST:GW2:2}
        \norm{(\mu_\lambda(s),\theta_\lambda(s))}_{L,\beta} \le C
        \quad\text{for all $s\ge \varsigma$.}
    \end{align}
    After integrating \eqref{EST:GW2:1} from $t$ to $t+1$ with respect to time, we use \eqref{Est:vw:lambda} and \eqref{EST:GW2:2} to obtain
    \begin{align*}
        &\int_t^{t+1}\norm{(\delt\phi_\lambda(s),\delt\psi_\lambda(s))}_{K,\alpha}^2\ds 
        \leq C + C\int_t^{t+1}\norm{(\bv_\lambda(s),\bw_\lambda(s))}_{\mathcal{H}^1}^2 \ds
        \\
        &\quad
        \leq C + C\norm{(\bv,\bw)}_{L^2_\mathrm{uloc}(\R;\mathcal{H}^1)}^2. 
    \end{align*}
    for all $t\ge \varsigma$.
    Now, sending $\lambda\rightarrow 0$, we conclude that
    \begin{align*}
        &t\mapsto \norm{(\mu(t),\theta(t))}_{L,\beta}\in L^\infty(\varsigma,\infty), \\
        &t\mapsto \norm{(\delt\phi(t),\delt\psi(t))}_{K,\alpha}\in L^2_{\mathrm{uloc}}(\varsigma,\infty).
    \end{align*}
    for any $\varsigma > \tau$.
    Finally, testing \eqref{WF:1} with $(\zeta,\xi)\in\mathcal{H}^1_{L,\beta}$ satisfying $\norm{(\zeta,\xi)}_{\mathcal{H}^1} \leq 1$ and taking the supremum, we infer
    \begin{align*}
        \norm{(\delt\phi(t),\delt\psi(t))}_{(\mathcal{H}^1_{L,\beta})^\prime} \leq \norm{(\bv(t),\bw(t))}_{\mathcal{L}^2} + C\norm{(\mu(t),\theta(t))}_{L,\beta}
    \end{align*}
    for all $t\geq\varsigma$. Consequently, $(\delt\phi,\delt\psi)\in L^\infty(\varsigma,\infty;(\mathcal{H}^1_{L,\beta})^\prime)$. As in Case~1, the remaining regularity properties can be verified as in \cite{Knopf2025, Giorgini2025}.
    \blk%
\end{proof}

\medskip\pagebreak[3]

Using Theorem~\ref{thm:highreg}, we easily obtain the following regularity property. 

\begin{corollary}\label{COR:H3}
     Suppose that the assumptions of Theorem \ref{thm:highreg} are fulfilled. Moreover, assume that there exists a time $T>\tau$ such that $\phi$ and $\psi$ are uniformly strictly separated on $[T,\infty)$, that is, there exists $\delta>0$ such that 
     \begin{align}
        \label{PROP:SEP}
         \sup_{t\geq T}\norm{(\phi(t),\psi(t))}_{L^\infty(\Omega\times \Gamma)}\leq 1-\delta.
     \end{align}
     Then it holds 
     \begin{align}\label{H3}
         \norm{(\phi,\psi)}_{L^\infty(T,\infty;\mathcal H^3)}\leq C(T)
     \end{align}
     for some constant $C(T)>0$ that may depend on $T$.
\end{corollary}

\medskip

\begin{proof}
    Recalling \ref{S1} and \ref{S3}, and using the regularity $(\phi,\psi)\in L^\infty(\varsigma,\infty;\mathcal{W}^{2,6})$ established in Theorem~\ref{thm:highreg}, the separation property \eqref{PROP:SEP} readily implies that
    \begin{align*}
        \norm{(F^\prime(\phi(t)), G^\prime(\psi(t))}_{\mathcal{H}^1} \leq C
    \end{align*}
    for all $t\geq T$. In addition, we know from Theorem~\ref{thm:highreg} that $(\mu,\theta)\in L^\infty(T,\infty;\mathcal{H}^1)$. Consequently, regularity theory for elliptic systems with bulk-surface coupling (see \cite[Theorem 3.3]{Knopf2021} or \cite[Proposition A.1]{Knopf2025}) yields
    \begin{align*}
        \norm{(\phi(t),\psi(t))}_{\mathcal{H}^3} \leq C\big(1 + \norm{(\mu(t) - F^\prime(\phi(t)), \theta(t) - G^\prime(\psi(t)))}_{\mathcal{H}^1}\big) \leq C(T)
    \end{align*}
    for all $t\geq T$.
\end{proof}

\section{Existence of a minimal pullback attractor}\label{SEC:LTD}

In contrast to previous works on the long-time dynamics of non-convective Cahn--Hilliard systems with or without dynamic boundary conditions (see, e.g., \cite{AbelsWilke,Garcke2022,Lv2024a,GrasselliPoiatti}), additional challenges in our model arise because of the additional convection terms. As the prescribed velocity fields $\bv$ and $\bw$ may be time-dependent, the underlying dynamical system is non-autonomous. If the velocity fields are not constant in time, the dynamical system can no longer be interpreted as a semigroup, and the discussion of a possible attractor becomes more interesting. 
Therefore, the non-autonomous nature of our model requires moving from the framework of semigroups and global attractors to the more general setting of two-parameter processes and pullback attractors.

\pagebreak[3]

To prove the existence of a pullback attractor, we introduce an additional assumption:
\begin{enumerate}[label=\textnormal{\bfseries(D4)}, leftmargin=*]
    \item \label{D4} The velocity fields have the additional regularities
    \begin{equation*}
        (\bv,\bw) \in L^2_\mathrm{uloc}(\R;\mathbf{L}^6(\Omega)\times \mathbf{L}^6(\Gamma)),
    \end{equation*}
    and if $K=0$, we demand that $\alpha=\beta$.
\end{enumerate}

\subsection{Formulation as a process}

We first introduce the concept of a two-parameter continuous process, which generalizes the notion of semigroups.

\begin{definition}[Continuous process]\label{Def:Process}
    A \textbf{continuous process} on a metric space $X$ is a family $\{S(t,\tau)\}_{t\geq\tau}$ of continuous mappings $S(t,\tau):X\to X$ with the following properties:
    \begin{enumerate}[label=\textnormal{(\roman*)},topsep=0.5ex, leftmargin=*]
    \item For all $\tau\in\R$ and $ x\in X$, it holds $S(\tau,\tau)x=x$.
    \item For all $t\ge s \ge \tau$, it holds $S(t,\tau) = S(t,s)\, S(s,\tau)$.
    \item The mapping 
    \begin{equation*}
        \{(t,\tau) \in \R^2 \,:\, t\ge \tau\} \times X \ni
        (t,\tau,x)\mapsto S(t,\tau)x \in X
    \end{equation*}
    is continuous.
\end{enumerate}
\end{definition}

\medskip

The following lemma shows that the dynamical system associated with \eqref{conCH*} can actually be formulated as a process.

\begin{lemma} \label{THM:PROC}
Suppose that \ref{ASSUMP:1}--\ref{ASSUMP:2}, \ref{S1}--\ref{S3}, \ref{D1}--\ref{D3} are fulfilled, and for any $(\phi_\tau,\psi_\tau) \in \mathcal{X}_{K,L,m}$, let
\begin{equation*}
    [\tau,\infty) \ni t \mapsto \big(\phi^{K,L},\psi^{K,L},\mu^{K,L},\theta^{K,L}\big)(t;\tau,\phi_\tau,\psi_\tau) \in \mathcal{X}_{K,L,m}
\end{equation*}
denote the unique weak solution of system \eqref{conCH*} corresponding to the velocity fields $(\bv,\bw)$ on $[\tau,\infty)$, which satisfies the initial condition
\begin{equation*}
    \big(\phi^{K,L},\psi^{K,L}\big)(\tau;\tau,\phi_\tau,\psi_\tau) = (\phi_\tau,\psi_\tau).
\end{equation*} 
For $\tau\in\R$ and $t\ge \tau$, we define
\begin{align*}
    &S_m^{K,L}(t,\tau): \mathcal{X}_{K,L,m} \to \mathcal{X}_{K,L,m}\,,
    \\
    &S_m^{K,L}(t,\tau)(\phi_\tau,\psi_\tau) \coloneqq  \big(\phi^{K,L},\psi^{K,L}\big)(t;\tau,\phi_\tau,\psi_\tau).
\end{align*}

Then the family $\{S_m^{K,L}(t,\tau)\}_{t\ge\tau}$ defines a continuous process on $\mathcal{X}_{K,L,m}$.
\end{lemma}

\begin{proof}
   The dynamical system is well-defined by Proposition~\ref{PROP:WS}. It is straightforward to check that conditions (i) and (ii) in Definition \ref{Def:Process} are satisfied.  We now show that the map 
    \begin{equation*}
        [\tau,\infty)\times \R \times \mathcal{X}_{K,L,m} 
        \ni (t;\tau,\phi_\tau,\psi_\tau)
        \to S_m^{K,L}(t,\tau)(\phi_\tau,\psi_\tau)
        \in \mathcal{X}_{K,L,m} 
    \end{equation*}
    is continuous.
    Continuity with respect to $\tau$ and $(\phi_*,\psi_*)$ follows from \eqref{est:cont:dep}. We further need to show that
    \begin{equation*}
        t\mapsto \big(\phi^{K,L},\psi^{K,L}\big)(t;\tau,\phi_\tau,\psi_\tau) \in C([\tau,\infty);\HH^1).
    \end{equation*}
    In view of the regularities established in Theorem~\ref{thm:highreg}, we use the Aubin--Lions--Simon lemma to infer that $(\phi^{K,L},\psi^{K,L})(\,\cdot\,;\tau,\phi_\tau,\psi_\tau) \in C((\tau,\infty);\HH^1)$. To prove continuity at $t = \tau$, we consider any sequence $(t_k)_{k\in\N} \subset (\tau,\infty)$ with $t_k\to \tau$. As the sequence $((\phi^{K,L},\psi^{K,L})(t_k))_{k\in\N}$ is bounded in $\HH^1,$ we deduce that, after subsequence extraction,
    \begin{equation*}
        (\phi^{K,L},\psi^{K,L})(t_k) \to (\phi_\tau,\psi_\tau)
        \quad\text{weakly in $\HH^1$, strongly in $\LL^2$, and a.e.~in $\Omega\times\Gamma$.}
    \end{equation*}
    Moreover, we have 
    \begin{align*}
        &\tfrac 12 \norm{(\phi^{K,L}(t_k),\psi^{K,L}(t_k))}_{K,\alpha}^2
            - \tfrac 12 \norm{(\phi_\tau,\psi_\tau)}_{K,\alpha}^2
        \\[1ex]
        &\quad = E_{\mathrm{free}}(\phi^{K,L}(t_k),\psi^{K,L}(t_k)) - E_{\mathrm{free}}(\phi_\tau,\psi_\tau)
        \\
        &\qquad
            + \intO F(\phi_\tau) - F\big(\phi^{K,L}(t_k)\big) \dx
            + \intG G(\psi_\tau) - G\big(\psi^{K,L}(t_k)\big) \dG.
    \end{align*}
    Since $F$ and $G$ are continuous, the corresponding integrals converge to zero by the dominated convergence theorem. Furthermore, it is easy to see that 
    \begin{align*}
        E_{\mathrm{free}}(\phi_\tau,\psi_\tau) \le \underset{k\to\infty}{\liminf}\; E_{\mathrm{free}}(\phi^{K,L}(t_k),\psi^{K,L}(t_k))
    \end{align*}
    as all terms in the energy are either convex (and thus weakly lower semicontinuous) or continuous (by dominated convergence).
    However, by taking the limit superior in the energy inequality \eqref{WS:ENERGY}, we infer that 
    \begin{align*}
         \underset{k\to\infty}{\limsup}\; E_{\mathrm{free}}(\phi^{K,L}(t_k),\psi^{K,L}(t_k))
         \le E_{\mathrm{free}}(\phi_\tau,\psi_\tau).
    \end{align*}
    Thus, we conclude that $E_{\mathrm{free}}(\phi^{K,L}(t_k),\psi^{K,L}(t_k)) \to E_{\mathrm{free}}(\phi_*,\psi_*)$, and this directly yields
    \begin{align*}
        \tfrac 12 \norm{(\phi^{K,L}(t_k),\psi^{K,L}(t_k))}_{K,\alpha}^2
        \to \tfrac 12 \norm{(\phi_\tau,\psi_\tau)}_{K,\alpha}^2.
    \end{align*}
    Together with the weak convergence in $\HH^1$, this implies that
    \begin{align*}
        (\Grad\phi^{K,L}(t_k),\Gradg\psi^{K,L}(t_k)) \to (\Grad\phi_\tau,\Gradg\psi_\tau)
        \quad\text{in $\LL^2$}.
    \end{align*}
    As the limit is independent of the previously extracted subsequence, this convergence actually remains true for the whole sequence $(t_k)_{k\in\N}$.
    Hence, we conclude that $(\Grad\phi^{K,L},\Gradg\psi^{K,L}) \in C([\tau,\infty);\mathbf{L}^2(\Omega)\times \mathbf{L}^2(\Omega))$.
    In combination with $(\phi^{K,L},\psi^{K,L}) \in C([\tau,\infty);\LL^2)$, this proves the claim.    
\end{proof}

In the special case of constant velocity fields, the continuous process $\{S_m^{K,L}(t,\tau)\}_{t\ge \tau}$ in fact coincides with a continuous semigroup. However, in view of our assumptions, the only choice of constant velocity fields that agrees with \ref{D2} is $(\bv,\bw)\equiv (\mathbf{0},\mathbf{0})$.

\begin{corollary}\label{COR:Semigroup}
    Suppose that the assumptions of Theorem~\ref{THM:PROC} hold with $(\bv,\bw)\equiv (\mathbf{0},\mathbf{0})$. 

    Then, for all $s\in\R$ and $t\ge s$, we have $S_m^{K,L}(t,s) = S_m^{K,L}(t-s,0)$. This means that the process is autonomous. For $t\ge 0$, we define
    \begin{align*}
    T_m^{K,L}(t) \coloneqq  S_m^{K,L}(t,0): \mathcal{X}_{K,L,m} \to \mathcal{X}_{K,L,m}.
    \end{align*}
    This family $T_m^{K,L}(t)$ has the following properties:
    \begin{enumerate}[label=\textnormal{(\roman*)}, leftmargin=*]
    \item $T_m^{K,L}(0)$ is the identity map on $\mathcal{X}_{K,L,m}$.
    \item (Semigroup property) For all $t,s \ge 0$, $T_m^{K,L}(t+s) = T_m^{K,L}(t)\, T_m^{K,L}(s)$.
    \item The mapping 
    \begin{equation*}
        [0,\infty)\times \mathcal{X}_{K,L,m} 
        \ni (t,\phi_0,\psi_0)
        \to T_m^{K,L}(t)(\phi_0,\psi_0)
        \in \mathcal{X}_{K,L,m} 
    \end{equation*}
    is continuous.
    \end{enumerate}
    This means that the family $\{T_m^{K,L}(t)\}_{t\ge 0}$ defines a continuous semigroup on $\mathcal{X}_{K,L,m}$.
\end{corollary}

\pagebreak[3]

\begin{proof}
   The identity $S_m^{K,L}(t,s) = S_m^{K,L}(t-s,0)$ holds since the time evolution is independent of the initial time. Property (ii) holds since 
        \begin{align*}
            T_m^{K,L}(t+s) 
            &= S_m^{K,L}(t+s,0)
            = S_m^{K,L}(t+s,s) \, S_m^{K,L}(s,0)
            \\
            &= S_m^{K,L}(t+s-s,0) \, S_m^{K,L}(s,0)
            = T_m^{K,L}(t)\, T_m^{K,L}(s).
        \end{align*}
       The remaining properties can be easily verified.
\end{proof}

\subsection{Existence of a unique pullback attractor}

Now, we are going to investigate the long-time behavior of the non-autonomous dynamical system $ \big(S_{m}^{K,L}(t,\tau),\mathcal{X}_{K,L,m}\big)$ through the theory of pullback attractors. As a preliminary step, we first recall the fundamental framework, including the basic definitions and a key theorem concerning pullback attractors, whose assumptions will be verified for our system in the subsequent analysis. In the following, for any metric space $X$, we write $\mathcal{P}(X)$ to denote the family of all non-empty subsets of $X$.

We start by recalling the main concepts related to pullback dynamics (see, e.g.,~\cite{marinrubio2009}), which will be used to establish our main result.

\medskip

\begin{definition}[Pullback absorbing sets]
  Let $\{S(t,\tau)\}_{t \geq \tau}$ be a continuous process on a metric space $X$.
    A family 
    $\widehat{D}_0 = \{ D_0(t) \}_{t\in\R}$ of non-empty sets is called
    \textbf{pullback absorbing} for the process $\{S(t,\tau)\}_{t \geq \tau}$, if for any bounded set ${D}\in \mathcal P(X)$ and any $t \in \R$, 
    there exists a time $\tau_0 (t,D) \leq t$ such that
    \begin{align*}
      S(t,\tau)D \subset D_0(t) \quad \text{for all } \tau \leq \tau_0 (t,{D}).  
    \end{align*}
    In this case, $\tau_0 (\cdot,\cdot)$ is referred to as the corresponding \textbf{pullback absorption time} for the absorbing family $\widehat{D}_0$.
\end{definition}

\medskip

\begin{definition}[Pullback attractor] \label{Definition:PullbackAttractor}
    Let $\{S(t,\tau)\}_{t \geq \tau}$ be a continuous process on a metric space $X$. 
    A family $\mathcal{A} = \{ \mathcal{A}(t) \}_{t \in \R}$ of non-empty subsets of $X$ 
    is called a \textbf{pullback attractor} 
    for the process $\{S(t,\tau)\}_{t \geq \tau}$ 
    in $X$ if
    \begin{enumerate}[label=\textnormal{(\roman*)}, topsep=0.5ex, leftmargin=*]
        \item $ \mathcal{A}(t) $ is compact in $X$ for every $t \in \R$,
        \item $ \mathcal{A}$ is invariant, i.e., $S(t,\tau)\mathcal{A}(\tau) = \mathcal{A}(t)$ for all $t \geq \tau$,
        \item $ \mathcal{A}$ is pullback attracting, i.e., for any $t\in \R$ and any bounded set $D$ it holds 
        \begin{equation*}
            \lim_{\tau \to -\infty} \dist_{X}(S(t,\tau)D, \mathcal{A}(t)) = 0.
        \end{equation*}
    \end{enumerate}
    A pullback attractor $\mathcal{A} = \{ \mathcal{A}(t) \}_{t \in \R}$ is referred to as \textbf{minimal}, if for every family $\widehat{C} = \{C(t)\}_{t\in\R} \in \mathcal{P}(X)$ of closed sets, which satisfies
    \begin{equation*}
        \lim_{\tau \to -\infty} \dist_{X}(S(t,\tau)D, C(t)) = 0
        \qquad\text{for all bounded sets $D$,}
    \end{equation*}
    it holds that $\mathcal{A}(t)\subset C(t)$ for all $t\in\R$.
\end{definition}

\medskip

\begin{remark}\label{REM:MPBA}
    We point out that in the case of autonomous systems, the definition of the global attractor is simply that of a compact invariant and attracting set. 
    Indeed, these two properties are enough to conclude that the global attractor is \textit{minimal closed attracting} (i.e., the smallest closed attracting set) and \textit{maximal invariant} (i.e., the largest invariant set). 
    For non-autonomous systems, these properties become more subtle. Specifically, unlike in autonomous dynamical systems, the global pullback attractor defined here is neither the maximal invariant set attracting all bounded sets nor the minimal closed attracting set. 
    In particular, the definition above does not guarantee the crucial property that the attractor attracts itself, meaning that 
    \begin{align*}
        \lim_{\tau\to-\infty}\dist(S(t,\tau)\mathcal A(\tau),\mathcal A(t))=0
        \quad\text{for all $t\in\R$}.
    \end{align*}
    This explains why the minimality property must be assumed \textit{a priori} in the definition above (leading to the \textit{minimal} pullback attractor), as it does not follow from the other properties. This assumption guarantees the uniqueness of $\mathcal A$.
    In order for the minimal and maximal properties to follow from the other assumptions, the attracting property must hold not just for bounded sets independent of time, but also for time-dependent families of sets belonging to a specific universe. This results in alternative definitions of pullback attractors (see, e.g., \cite{marinrubio2009}), which can be larger than the ones considered in the present paper.
\end{remark}

\medskip

We will use the following proposition to prove the existence of a pullback attractor for the process $S_m^{K,L}$ (see \cite[Theorem 7]{marinrubio2009} and \cite[Proposition 3.7]{garcialuengo2012}).

\begin{proposition}\label{PROP:pullback_att.}
    Let $\{S(t,\tau)\}_{t \geq \tau}$ be a continuous process on a metric space $X$.  
	Assume that there exists a compact attracting family of sets $\widehat K=\{K(t)\}_{t\in\R}$ for $\{S(t,\tau)\}_{t \ge \tau}$, where $K(t)$ is compact in $X$ for every $t\in \R$, such that
    \begin{align*}
        \lim_{\tau \to-\infty}\dist(S(t,\tau)D, K(t))=0,
    \end{align*}
    for any bounded set $D\in \mathcal P(X)$.

	Then there exists a minimal pullback attractor 
	$\mathcal{A} = \{\mathcal{A}(t)\}_{t\in\R}$ in $X$, 
	which is given by
	\begin{align*}
	\mathcal{A}(t) = 
	\overline{\bigcup_{\substack{D\in \mathcal P(X)\\ \text{bounded}\;}} \Lambda({D},t)}^{\,X},
		\end{align*}
	where
	\begin{align*}
    	\Lambda({D},t)
    	= \bigcap_{s \le t} \overline{\bigcup_{\tau \le s} S(t,\tau) D}^{\,X}
    	\qquad \text{for any bounded set ${D} \in \mathcal P(X)$}.
	\end{align*}
    Additionally, if there exists a family 
		$\widehat{D}_0 = \{D_0(t)\}_{t\in\R} \subset \mathcal{P}(X)$ 
		such that $\widehat{D}_0$ is pullback absorbing 
		for $\{S(t,\tau)\}_{t \ge \tau}$ and if ${D}_0(t)\equiv D_0$ for all $t\in \R$, then
    \begin{equation}
        \mathcal{A}(t) = \Lambda(D_0,t) 
        \subset \overline{D_0}^X
        \quad\text{for all $t\in\R$.}\label{bound}
    \end{equation}
\end{proposition}
\begin{remark}
    Notice that if the pullback absorbing set is time-independent and $\mathcal A$ exists, then \eqref{bound} shows that the system's dissipative properties suffice to ensure that the pullback attractor $\{\mathcal A(t)\}_{t\in \R}$ is uniformly bounded in time, which is not guaranteed in general.
\end{remark}

\subsubsection{Existence of a pullback absorbing set}

We start by proving the existence of a pullback $\mathcal{D}^\mathcal{X}$-absorbing family of sets for the process $\{S^{K,L}_m(t,\tau)\}_{t\geq\tau}$.

\begin{proposition}\label{Prop:PBAbsorbing}
    Suppose that \ref{ASSUMP:1}--\ref{ASSUMP:2}, \ref{S1}--\ref{S3}, \ref{D1}--\ref{D4} are fulfilled.
    
    Then there is a family $\widehat{D}_0=\{D_0(t)\}_{t\in\R}$ of compact sets in $\mathcal{X}_{K,L,m}$ that is pullback absorbing for the process $\{S^{K,L}_m(t,\tau)\}_{t\geq\tau}$ associated with system \eqref{conCH*} with pullback absorption time $\tau_0({D},t) \le t-2$ for all bounded sets ${D}$ and $t\in\R$. More precisely, the family $\widehat D_0$ is constant in time, and there exists a compact set $D_0$ such that $D_0(t)\equiv D_0$ for any $t\in \R$.
\end{proposition}

\begin{proof}
    In this proof, let $C$ denote a generic positive constant depending only on the prescribed velocity fields, the quantities specified in \ref{ASSUMP:1} and \ref{ASSUMP:2}, $F$ and $G$, which may change its value from line to line.
    We test $\big(\eqref{conCH*:1},\eqref{conCH*:3}\big)$ with $(\mu,\theta)$ and $\big(\eqref{conCH*:2},\eqref{conCH*:4}\big)$ with $-(\delt\phi,\delt\psi)$. Performing integration by parts and adding the resulting equations leads to
    \begin{align}\label{PB:ddt:Energy}
        &\ddt E_{\mathrm{free}}(\phi,\psi) + \norm{(\mu,\theta)}_{L,\beta}^2 = \intO\phi\bv\cdot\Grad\mu\dx + \intG \psi\bw\cdot\Gradg\theta\dG.
    \end{align}
    Using the Gagliardo-Nirenberg's inequalities and the bulk-surface Poincar\'{e} inequality, we find that
    \begin{align*}
        &\Bigabs{\intO\phi\bv\cdot\Grad\mu\dx + \intG \psi\bw\cdot\Gradg\theta\dG} \\
        &\quad\leq \norm{(\mu,\theta)}_{L,\beta}\norm{(\bv,\bw)}_{\mathcal{L}^6}\norm{(\phi,\psi)}_{\mathcal{L}^3} \\
        &\quad\leq C\norm{(\mu,\theta)}_{L,\beta}\norm{(\bv,\bw)}_{\mathcal{L}^6}\norm{(\phi,\psi)}_{\mathcal{L}^2}^{\frac12}\norm{(\phi,\psi)}_{\mathcal{H}^1}^{\frac12} \\
        &\quad\leq\frac12\norm{(\mu,\theta)}_{L,\beta}^2 + C\big(1 + \norm{(\bv,\bw)}_{\mathcal{L}^6}^2\big)\norm{(\phi,\psi)}_{K,\alpha} \,.
    \end{align*}
    Thus, we have
    \begin{align}\label{PB:ddt:Energy:1}
        \ddt E_{\mathrm{free}}(\phi,\psi) + \frac12\norm{(\mu,\theta)}_{L,\beta}^2 \leq C\big(1 + \norm{(\bv,\bw)}_{\mathcal{L}^6}^2\big)\norm{(\phi,\psi)}_{K,\alpha} \,.
    \end{align}
    For the following computations, we assume that $L\in[0,\infty)$. The case $L = \infty$ can be handled similarly. Setting $(\phi^\ast,\psi^\ast) = (\phi - \beta\mean{\phi_\tau}{\psi_\tau}, \psi - \mean{\phi_\tau}{\psi_\tau})$, we have   
    \begin{align}\label{PB:ddt:Lb*}
        &\ddt\frac12\norm{(\phi^\ast,\psi^\ast)}_{L,\beta,\ast}^2 
            = \ddt\frac12\norm{\mathcal{S}_{L,\beta}(\phi^\ast,\psi^\ast)}_{L,\beta}^2 
        \nonumber\\
        &\quad = \big(\mathcal{S}_{L,\beta}(\delt\phi^\ast,\delt\psi^\ast),\mathcal{S}_{L,\beta}(\phi^\ast,\psi^\ast)\big)_{L,\beta} 
        \nonumber\\
        &\quad = - \big\langle (\delt\phi^\ast,\delt\psi^\ast),\mathcal{S}_{L,\beta}(\phi^\ast,\psi^\ast)\big\rangle_{\mathcal{H}^1_{L,\beta}} 
        \\
        &\quad =  \big((\mu,\theta),\mathcal{S}_{L,\beta}(\phi^\ast,\psi^\ast)\big)_{L,\beta} 
            - \big(\phi\bv,\Grad\mathcal{S}^\Omega_{L,\beta}(\phi^\ast,\psi^\ast)\big)_{\mathbf{L}^2(\Omega)}
            - \big(\psi\bw,\Gradg\mathcal{S}^\Gamma_{L,\beta}(\phi^\ast,\psi^\ast)\big)_{\mathbf{L}^2(\Gamma)}
        \nonumber\\
        &\quad = - \big((\mu,\theta),(\phi^\ast,\psi^\ast)\big)_{\mathcal{L}^2} 
            - \big(\phi\bv,\Grad\mathcal{S}^\Omega_{L,\beta}(\phi^\ast,\psi^\ast)\big)_{\mathbf{L}^2(\Omega)}
            - \big(\psi\bw,\Gradg\mathcal{S}^\Gamma_{L,\beta}(\phi^\ast,\psi^\ast)\big)_{\mathbf{L}^2(\Gamma)}.
        \nonumber    
    \end{align}
    Then, using the weak formulation \eqref{WF:2}, we have
    \begin{align}\label{PB:MT:PP}
        \begin{split}
            - \big((\mu,\theta),(\phi^\ast,\psi^\ast)\big)_{\mathcal{L}^2} &= - \norm{(\phi,\psi)}_{K,\alpha}^2 - \intO F^\prime(\phi)\phi^\ast\dx - \intG G^\prime(\psi)\psi^\ast\dG \\
            &\quad + \mathbf{1}_{\{0\}}(K)(\alpha - \beta)\mean{\phi_\tau}{\psi_\tau}\intG (\alpha\psi - \phi)\dG,
        \end{split}
    \end{align}
    where $\mathbf{1}_{\{0\}}(\cdot)$ denotes the characteristic function of the set $\{0\}$, which is added to deal with all the cases $K\in[0,\infty]$ simultaneously. 
    In view of assumption \ref{D4}, where $\alpha = \beta$ if $K = 0$ was demanded, we readily see that the last term on the right-hand side of \eqref{PB:MT:PP} vanishes.
    Now, exploiting the convexity of $F_1$ and $G_1$, respectively, as well as the Lipschitz continuity of $F_2^\prime$ and $G_2^\prime$, respectively, we find that
    \begin{align}\label{PB:Pot:1}
        \begin{split}
            &\intO F_1^\prime(\phi)\big(\phi - \beta\mean{\phi_\tau}{\psi_\tau}\big)\dx + \intG G_1^\prime(\psi)\big(\psi - \mean{\phi_\tau}{\psi_\tau}\big)\dG 
            \\
            &\quad\geq \intO F_1(\phi) - F_1(\beta\mean{\phi_\tau}{\psi_\tau})\dx + \intG G_1(\psi) - G_1(\mean{\phi_\tau}{\psi_\tau})\dG 
            \\
            &\quad\geq \intO F_1(\phi)\dx + \intG G_1(\psi)\dG - \abs{\Om}F_1(\beta\mean{\phi_\tau}{\psi_\tau}) - \abs{\Ga}G_1(\mean{\phi_\tau}{\psi_\tau}),
        \end{split}
    \end{align}
    and
    \begin{align}\label{PB:Pot:2}
        \intO F_2(\phi)\phi^\ast\dx + \intG G_2(\psi)\psi^\ast\dx \leq C.
    \end{align}
    For the advection-related terms on the right-hand side of \eqref{PB:ddt:Lb*}, we again use Gagliardo--Nirenberg's inequalities to obtain
    \begin{align}\label{PB:Conv}
        \begin{split}
            &\Bigabs{\intO\phi\bv\cdot\Grad\mathcal{S}_{L,\beta}^\Om(\phi^\ast,\psi^\ast)\dx 
                + \intG \psi\bw\cdot\Gradg\mathcal{S}_{L,\beta}^\Ga(\phi^\ast,\psi^\ast)\dG} 
            \\
            &\quad= \Bigabs{\intO \mathcal{S}_{L,\beta}^\Om(\phi^\ast,\psi^\ast)\bv\cdot\Grad\phi\dx + \intG \mathcal{S}_{L,\beta}^\Ga(\phi^\ast,\psi^\ast)\bw\cdot\Gradg\psi\dG} 
            \\
            &\quad\leq \norm{\mathcal{S}_{L,\beta}(\phi^\ast,\psi^\ast)}_{\mathcal{L}^3}\norm{(\bv,\bw)}_{\mathcal{L}^6}\norm{(\phi,\psi)}_{K,\alpha} 
            \\
            &\quad\leq C\norm{\mathcal{S}_{L,\beta}(\phi^\ast,\psi^\ast)}_{\mathcal{L}^2}^{\frac12}\norm{(\phi^\ast,\psi^\ast)}_{L,\beta,\ast}^{\frac12}\norm{(\bv,\bw)}_{\mathcal{L}^6}\norm{(\phi,\psi)}_{K,\alpha} 
            \\
            &\quad\leq\frac12\norm{(\phi,\psi)}_{K,\alpha}^2 + C\norm{(\bv,\bw)}_{\mathcal{L}^6}^2\norm{(\phi^\ast,\psi^\ast)}_{L,\beta,\ast}.
        \end{split}
    \end{align}
    Thus, collecting the estimates \eqref{PB:Pot:1}-\eqref{PB:Conv} to bound the right-hand side of \eqref{PB:ddt:Lb*}, we deduce
    \begin{align}\label{PB:ddt:Lb*:1}
        \begin{split}
            &\ddt\frac12\norm{(\phi^\ast,\psi^\ast)}_{L,\beta,\ast}^2 + \frac12\norm{(\phi,\psi)}_{K,\alpha}^2 + \intO F_1(\phi)\dx + \intG G_1(\psi)\dG \\
            &\quad\leq C + C\norm{(\bv,\bw)}_{\mathcal{L}^6}^2\norm{(\phi^\ast,\psi^\ast)}_{L,\beta,\ast} \,.
        \end{split}
    \end{align}
    Consequently, adding \eqref{PB:ddt:Energy:1} and \eqref{PB:ddt:Lb*:1}, and again using the fact that $F_2^\prime$ and $G_2^\prime$ are both Lipschitz continuous, leads to
    \begin{align}\label{PB:DiffIneq:0}
        \begin{split}
            &\ddt \Big(E_{\mathrm{free}}(\phi,\psi) + \frac12\norm{(\phi^\ast,\psi^\ast)}_{L,\beta,\ast}^2\Big) + E_{\mathrm{free}}(\phi,\psi) + \frac12\norm{(\mu,\theta)}_{L,\beta}^2 + \frac12\norm{(\phi,\psi)}_{K,\alpha}^2 \\
            &\quad\leq C + C\norm{(\bv,\bw)}_{\mathcal{L}^6}^2\Big(\norm{(\phi,\psi)}_{K,\alpha} + \norm{(\phi^\ast,\psi^\ast)}_{L,\beta,\ast}\Big).
        \end{split}
    \end{align}
    Recalling the definition of $E_{\mathrm{free}}$ and the fact that $F$ and $G$ are bounded from below on $[-1,1]$, it easy to see that there exist constants $C^\prime$ and $C^{\prime\prime}$ (which may depend on the same quantities as $C$) such that
    \begin{align}
        \label{EST:ABS:1}
        \begin{split}
        \frac12\norm{(\phi,\psi)}_{K,\alpha}^2 + \frac12\norm{(\phi^\ast,\psi^\ast)}_{L,\beta,\ast}^2 - C^\prime &\leq E_{\mathrm{free}}(\phi,\psi) + \frac12\norm{(\phi^\ast,\psi^\ast)}_{L,\beta,\ast}^2 
        \\
        & \leq \frac12\norm{(\phi,\psi)}_{K,\alpha}^2 + \frac12\norm{(\phi^\ast,\psi^\ast)}_{L,\beta,\ast}^2 + C^{\prime\prime}.
        \end{split}
    \end{align}
    Therefore, defining $\Upsilon(t) \coloneqq C^\prime + E_{\mathrm{free}}(\phi(t),\psi(t)) + \frac12\norm{(\phi^\ast(t),\psi^\ast(t))}_{L,\beta,\ast}^2$, we readily infer from Minkowski's inequality that
    \begin{align*}
        \Upsilon^{\frac12}(t)\geq\frac12\big(\norm{(\phi(t),\psi(t))}_{K,\alpha} + \norm{(\phi^\ast(t),\psi^\ast(t))}_{L,\beta,\ast}\big).
    \end{align*}
    Thus, we obtain from \eqref{PB:DiffIneq:0} the differential inequality
    \begin{align}\label{PB:DiffIneq}
        \ddt\Upsilon(t) + \Upsilon(t) + \frac12\norm{(\mu(t),\theta(t))}_{L,\beta}^2 \leq C + C\norm{(\bv(t),\bw(t))}_{\mathcal{L}^6}^2\Upsilon^{\frac12}(t)
    \end{align}
    for all $t\geq\tau$.
    Hence, applying the variant of Gronwall's inequality presented in Lemma~\ref{Lemma:Gronwall:2}, we infer that
    \begin{align}
        \label{EST:ABS:1.5}
        \Upsilon(t) 
        &\leq 2 \Upsilon(\tau) \mathrm{e}^{-\frac{1}{2}(t-\tau)}
            + \Big(
                C\norm{(\bv,\bw)}_{L^2_\mathrm{uloc}(\R;\mathbf{L}^6(\Omega)\times \mathbf{L}^6(\Gamma))}^2
            \Big)^2
            + C
    \end{align}
    for all $t\geq\tau$.
    Now, we recall that $(\phi(s),\psi(s)) \in \mathcal{X}_{K,L,m}$ for all $s\ge \tau$. If $K\in [0,\infty)$, we use the bulk-surface Poincar\'e inequality (Lemma~\ref{lemma:poinc}) to deduce that
    \begin{align}
        \label{EST:ABS:2}
        \norm{(\phi(s),\psi(s))}_{\mathcal{H}^1} 
        \le C_1 \norm{(\phi(s),\psi(s))}_{K,\alpha}
        \quad\text{for all $s\ge \tau$,}
    \end{align}
    where $C_1$ is a positive constant depending only on $K$, $\alpha$, $\beta$ and $\Omega$.
    If $K=\infty$, we use the same argumentation as for \eqref{es01*} to get
    \begin{align*}
        \norm{(\phi(s),\psi(s))}_{\mathcal{H}^1} \leq C_1\big(\norm{(\phi(s),\psi(s))}_{K,\alpha} + \norm{(\phi(s),\psi(s))}_{L,\beta,\ast}\big) \qquad\text{for all }s\geq\tau.
    \end{align*}
    This allows us to conclude that
    \begin{align}\label{PB:Gronwall:Final}
        \norm{(\phi(t),\psi(t))}_{\mathcal{H}^1}^2 \leq C_*\big(1 + \norm{(\phi_\tau,\psi_\tau)}_{\mathcal{H}^1}\big)\mathrm{e}^{\frac 12 (\tau-t)} + C_{**}
    \end{align}
    for all $t\geq\tau$, where $C_*$ and $C_{**}$ are positive constants which may depend on the same quantities as constants denoted by $C$. In particular, this means that the constant $C_*$ and $C_{**}$ are independent of $\tau$ and $(\phi_\tau,\psi_\tau)$.

    Now, let $t\in\R$ and $D\subset \mathcal{X}_{K,L,m}$ be bounded. Thus, there exists $R=R(D)>0$ such that 
    \begin{align*}
        \bignorm{\big(\tilde\phi,\tilde\psi\big)}_{\mathcal{H}^1} \le R
        \quad\text{for all $\big(\tilde\phi,\tilde\psi\big)\in D$.}
    \end{align*}
    We define
    \begin{align*}
        \tau_0({D},t) \coloneqq t - 2(1+R) \le t - 2.
    \end{align*}
    By this definition, for any $\tau \le \tau_0({D},t)$, we have $\mathrm{e}^{\frac 12 (\tau-t)} \le (1+R)^{-1}$. In view of the estimates \eqref{EST:ABS:1.5}--\eqref{PB:Gronwall:Final}, this implies that
    \begin{align}
        \label{EST:ABSORB}
        \norm{(\phi(t),\psi(t))}_{\mathcal{H}^1}^2 
        \leq C_* + C_{**}
        \quad\text{and}\quad
        \Upsilon(t) \le C
    \end{align}
    for all $\tau \le \tau_0({D},t)$ and initial datum $(\phi_\tau,\psi_\tau) \in D$. 
    Now, we set $\rho \coloneqq C_* + C_{**}$,
    \begin{align*}
        \tilde D_0 \coloneqq 
            \Big\{
                \big(\tilde\phi,\tilde\psi\big) 
                \in \mathcal{X}_{K,L,m}
                \,:\, \bignorm{\big(\tilde\phi,\tilde\psi\big)}_{\mathcal{H}^1}^2 \le \rho 
            \Big\}.
    \end{align*}
    Then, for any $\tau \le \tau_0({D},t)$, it follows from \eqref{EST:ABSORB} that
    \begin{align}
        \label{EST:ABSORB:0}
        S^{K,L}_m(t,\tau)D\subset \tilde D_0 
        \qquad\text{for all~}\tau \leq \tau_0(D, t).
    \end{align}
    Since $D$ was assumed to be an arbitrary bounded subset of the phase space, this shows that the (constant in time) family $\widehat{ {D}}_0:=\{\tilde D_0(t)\equiv \tilde D_0\}_{t\in\R}$ is pullback absorbing for the process 
    $\{S^{K,L}_m(t,\tau)\}_{t\ge \tau}$. We now aim at showing that this set can be improved to be a compact set of the phase space.
    To this end, given $t\in \R$, let us consider $\tau_0(t):=\tau_0(\tilde D_0,t)\leq t-2$ as to be the pullback absorbing time of $\tilde D_0$ into itself. Let now $r\in[t-1,t]$, $t\in \R$, be arbitrary.
    After integrating \eqref{PB:DiffIneq:0} in time from $r-1$ to $r$, we use \eqref{EST:ABSORB} along with the definition of $\rho$ to obtain, for any $\tau\leq \tau_0(t)$ and initial datum $(\phi_\tau,\psi_\tau)\in\tilde  D_0$,
    \begin{align*}
        &\int_{r-1}^r \norm{(\mu,\theta)}_{L,\beta}^2 \ds 
        \\
        &\quad\leq C + E_{\mathrm{free}}(\phi(r-1),\psi(r-1)) + \frac12\norm{(\phi^\ast(r-1),\psi^\ast(r-1))}_{L,\beta,\ast}^2 \\
        &\qquad + C\int_{r-1}^r\norm{(\bv(s),\bw(s))}_{\mathcal{L}^6}^2\Big(\norm{(\phi(s),\psi(s))}_{K,\alpha} + \norm{(\phi^\ast(s),\psi^\ast(s))}_{L,\beta,\ast}\Big)\ds 
        \\[1ex]
        &\quad\leq C + C\rho^2 + C\rho\, \norm{(\bv,\bw)}_{L^2_\mathrm{uloc}(\R;\mathbf{L}^6(\Omega)\times \mathbf{L}^6(\Gamma))}^2
        \eqqcolon \rho^\prime\,.
    \end{align*}
    Taking the supremum over all $r\in[t-1,t]$ yields
    \begin{align}
        \label{PB:MT:LB:UniformIC}
        \sup_{r\in[t-1,t]}\int_{r-1}^r \norm{(\mu,\theta)}_{L,\beta}^2 \ds &\leq \rho^\prime.
    \end{align}
 Now, using the estimates  \eqref{PB:PP:LB:UniformIC} and \eqref{EST:ABSORB} instead of the energy estimates and Lemma~\ref{LEM:UNIFORMBOUNDEDNESS}, we can follow the argument of the proof of Theorem~\ref{thm:highreg} to verify that,  for any $\tau\leq \tau_0(t)$ and initial datum $(\phi_\tau,\psi_\tau)\in \tilde D_0$, 
   \begin{align}
        \label{PB:PP:LB:UniformIC}
        \sup_{r\in[t-2,t]}\norm{(\phi(r),\psi(r))}_{\mathcal{H}^2}^2 + \sup_{r\in[t-1,t]}\int_{r-1}^r \norm{(\delt\phi,\delt\psi)}_{\mathcal{L}^2}^2\ds 
        &\leq \rho^{\prime\prime},
    \end{align}
    where $\rho''>0$ is a constant with the same dependencies of $\rho'$.
We omit the full details for the sake of brevity. We can then introduce the compact set
 \begin{align*}
        D_0 \coloneqq 
            \Big\{
                \big(\tilde\phi,\tilde\psi\big) 
                \in \mathcal{X}_{K,L,m}
                \,:\, \bignorm{\big(\tilde\phi,\tilde\psi\big)}_{\mathcal{H}^2}^2 \le \rho''
            \Big\}.
    \end{align*}
From \eqref{PB:PP:LB:UniformIC} we thus conclude that 
 \begin{align}
        \label{EST:ABSORB:1b}
        S^{K,L}_m(t,\tau)\tilde D_0\subset  D_0 
        \qquad\text{for all~}\tau \leq \tau_0(t).
    \end{align}
 As a consequence, since $\tilde D_0$ is a pullback absorbing set for the dynamical system, also $D_0$, which is of course compact in $\mathcal{X}_{K,L,m}$, is pullback absorbing, concluding the proof.  
\end{proof}

\subsubsection{Existence of a minimal pullback attractor}

We are now in a position to straightforwardly prove the following theorem, which establishes the existence of a minimal pullback attractor for the processes $\{S^{K,L}_m(t,\tau)\}_{t\geq\tau}$.
\begin{theorem}
    Suppose that \ref{ASSUMP:1}--\ref{ASSUMP:2}, \ref{S1}--\ref{S3}, \ref{D1}--\ref{D4} are fulfilled.
    
    Then the continuous process $\{S^{K,L}_m(t,\tau)\}_{t\geq\tau}$ associated with system \eqref{conCH*} admits a minimal pullback attractor $\mathcal{A}$ in the sense of Definition~\ref{Definition:PullbackAttractor}, which is given by
    \begin{equation}
        \label{REP:PBA}
        \mathcal{A}(t) = \Lambda(D_0,t)
        = \bigcap_{s\leq t}\overline{\bigcup_{\tau\leq s}  
            S^{K,L}_m(t,\tau)D_0}^{\mathcal{X}_{K,L,m}}
        \subset \overline{D_0}^{\mathcal{X}_{K,L,m}}
        \quad\text{for all $t\in\R$,}
    \end{equation}
    where $\widehat{D}_0 = \{D_0(t) \equiv D_0\}_{t\in\R}$ is the compact pullback absorbing set of Proposition~\ref{Prop:PBAbsorbing}.
\end{theorem}

\begin{proof}
    In view of Proposition~\ref{Prop:PBAbsorbing}, we can apply the abstract result provided by Proposition~\ref{PROP:pullback_att.}, since the dynamical system admits a compact pullback attracting (in this case even absorbing) set. We thus conclude that the process $\{S^{K,L}_m(t,\tau)\}_{t\geq\tau}$ admits a minimal pullback attractor $\mathcal{A} = \{\mathcal{A}(t)\}_{t\in\R}$ in $\mathcal{X}_{K,L,m}$, which is given by \eqref{REP:PBA} since the time-indexed family $\widehat{D}_0$ is constant in time.
\end{proof}

\section{Convergence to a steady state}

Finally, we want to prove that any weak solution of system \eqref{conCH*} converges as $t\to\infty$ to a single equilibrium solution. However, in order to show this, we require an additional decay condition on our velocity fields:

\begin{enumerate}[label=\textnormal{\bfseries(D5)}, leftmargin=*]
    \item \label{D5} There exist $T_\mathrm{dec}>0 $ and $a\geq 0$ such that the mapping
    \begin{equation*}
     (T_\mathrm{dec},\infty) \ni 
     t \mapsto \bignorm{\big(\bv(t),\bw(t) \big)}_{\LL^2}
    \end{equation*}
    is non-increasing, and
      \begin{equation*}
     \int_{T_\mathrm{dec}}^\infty 
     \mathrm{e}^{as}\bignorm{\big(\bv(s),\bw(s) \big)}_{\LL^2} 
     \ds < \infty.
    \end{equation*}
\end{enumerate}

\medskip

\begin{remark}
    We point out that that condition \ref{D5} is fulfilled, for instance, if $\norm{(\bv(t),\bw(t) )}_{\LL^2}$ decays exponentially fast as $t\to\infty$. 
    If $a>0$, thanks to Hölder's inequality, condition \ref{D5} further implies that, for any $\gamma\in (0,1)$, it holds that
    \begin{align*}
        &\int_{T_\mathrm{dec}}^\infty 
        \bignorm{\big(\bv(s),\bw(s) \big)}_{\LL^2}^\gamma 
        \ds
        \\
        &\quad\leq \left(\int_{T_\mathrm{dec}}^\infty 
        \mathrm{e}^{as} \bignorm{\big(\bv(s),\bw(s) \big)}_{\LL^2} 
        \ds\right)^\gamma \left(\int_{T_\mathrm{dec}}^\infty 
        \mathrm{e}^{-a\frac{\gamma}{1-\gamma}s} 
        \ds\right)^{1-\gamma}<\infty.
    \end{align*}
    \label{control}
\end{remark}

\medskip

In the following, we simply set $\tau=0$. This does not mean any loss of generality, as the concrete value of $\tau$ does not have any impact on the mathematical analysis of the longtime behavior of trajectories. As a first step, we fix an arbitrary $\sigma\in (\tfrac d2,2)$ and recall the definition of the $\omega$-limit set.

\begin{definition}[$\omega$-limit set] 
Suppose that \ref{ASSUMP:1}--\ref{ASSUMP:2}, \ref{S1}--\ref{S3}, and \ref{D1}--\ref{D3} are fulfilled and let  $(\phi_0,\psi_0) \in \mathcal{X}_{K,L,m}$ be arbitrary.

Then the set
\begin{align*}
    \omega^{K,L}(\phi_0, \psi_0) \coloneqq \Bigg\{(\phi_\ast, \psi_\ast)\in\mathcal{W}^\sigma_{K,L,\alpha,\beta, m} \Bigg\vert
    \begin{aligned}
        &\exists (t_k)_{k\in\N}\subset[0,\infty) \text{~with~} t_k\rightarrow\infty\text{~such that~} \\
        &S^{K,L}_m(t_k,0)(\phi_0,\psi_0)\rightarrow (\phi_\ast, \psi_\ast) \text{~strongly~in~}
        \mathcal{H}^\sigma
    \end{aligned}
    \Bigg\},
\end{align*}
is called the \textbf{$\omega$-limit set} to the initial data $(\phi_0, \psi_0)$. 
\end{definition}

\medskip

\begin{lemma}\label{Lemma:oml}
    Suppose that \ref{ASSUMP:1}--\ref{ASSUMP:2}, \ref{S1}--\ref{S3}, and \ref{D1}--\ref{D3} are fulfilled and let  $(\phi_0,\psi_0) \in \mathcal{X}_{K,L,m}$ be arbitrary.
    Then $\oml$ is a non-empty, compact and connected subset of $\mathcal{W}^\sigma_{K,L,m}$. Moreover, it holds that
    \begin{align*}
        \mathrm{dist}_{\mathcal{W}^\sigma_{K,L,m}}\big(S_m^{K,L}(t,0)(\phi_0,\psi_0),\oml\big) \rightarrow 0 \qquad\text{as~}t\rightarrow\infty.
    \end{align*}
\end{lemma}

\begin{proof}
The claim follows directly by employing the propagation of regularity from Theorem~\ref{thm:highreg} as well as \cite[Theorem~9.1.8]{cazenave}.
\end{proof}

\medskip

\begin{definition}[Set of stationary points]
Suppose that \ref{ASSUMP:1}--\ref{ASSUMP:2}, \ref{S1}--\ref{S3} and \ref{D1} are fulfilled, and let 
$\{T_{m}^{K,L}\}_{t\ge 0}$ be the strongly continuous semigroup on $\mathcal{X}_{K,L,m}$ corresponding to trivial velocity fields $(\bv_\infty,\bw_\infty) = (\mathbf{0},\mathbf{0})$, which was introduced in Corollary~\ref{COR:Semigroup}. 

Then the set
\begin{align*}
	\mathcal{N}_{m}^{K,L}=\left\lbrace  (\phi,\psi) \in \mathcal{W}_{K,L,m}^1:T_{m}^{K,L}(t) (\phi,\psi) = (\phi,\psi) \;\text{for all}\; t\geq 0\right\rbrace.
\end{align*}
is called the \textbf{set of stationary points}.
\end{definition}

\medskip

The set of stationary points $\mathcal{N}_{m}^{K,L}$ owes its name to the fact that its elements are stationary solutions of system \eqref{conCH*}. This is shown by the following lemma.

\begin{lemma}
    Suppose that \ref{ASSUMP:1}--\ref{ASSUMP:2}, \ref{S1}--\ref{S3}, and \ref{D1}--\ref{D3} are fulfilled and let $(\phi_\infty,\psi_\infty)\in\mathcal{N}^{K,L}_m$. 
    
    Then there exist constants
    $\mu_\infty$ and $\theta_\infty$ given by
    \begin{alignat}{2}
        \label{Id:MT_infty:L}
        &\mu_\infty = \beta\theta_\infty = \frac{\beta}{\alpha\beta\abs{\Om} + \abs{\Ga}}\Big(\alpha\intO F^\prime(\phi_\infty)\dx + \intG G^\prime(\psi_\infty)\dG\Big)
        &&\quad\text{if $L\in [0,\infty)$},
        \\[1ex]
        \label{Id:MT_infty:infty}
        &\left\{
        \begin{aligned}
            \mu_\infty &= \frac{1}{\abs{\Om}}\Big(\intO F^\prime(\phi_\infty)\dx - \intG \deln\phi_\infty\dG\Big), \\
            \theta_\infty &= \frac{1}{\abs{\Ga}}\Big(\intG G^\prime(\psi_\infty) + \alpha\deln\phi_\infty\dG\Big)
        \end{aligned}
        \right.
        &&\quad\text{if $L=\infty$,}
    \end{alignat}
    such that the quadruplet $(\phi_\infty,\psi_\infty,\mu_\infty,\theta_\infty)$, where $\mu_\infty$ and $\theta_\infty$ are now interpreted as constant functions, is the unique weak solution of 
    system \eqref{conCH*} on the interval $[0,\infty)$ corresponding to the initial data $(\phi_\infty,\psi_\infty)$.
    In particular, this solution is stationary as each of the functions $\phi_\infty$, $\psi_\infty$, $\mu_\infty$, and $\theta_\infty$ is constant in time.
    
    Furthermore, the pair $(\phi_\infty,\psi_\infty)$ belongs to $\mathcal{W}^{2,6} \cap \mathcal{W}_{K,L,m}^\sigma$ and is a strong solution of the semilinear bulk-surface elliptic system
    \begin{subequations}\label{sta3}
    \begin{alignat}{2}
        -\Lap\phi_\infty + F^\prime(\phi_\infty) &= \mu_\infty 
        &&\qquad\text{in~}\Om, \label{sta3:1} 
        \\
        - \Lapg\psi_\infty + G^\prime(\psi_\infty) + \alpha\deln\phi_\infty &= \theta_\infty
        &&\qquad\text{on~}\Ga, \label{sta3:2} 
        \\
        K\deln\phi_\infty &= \alpha\psi_\infty - \phi_\infty 
        &&\qquad\text{on~}\Ga. \label{sta3:3}
    \end{alignat}
    \end{subequations}
\end{lemma}

\begin{proof}
Suppose that $(\phi_\infty,\psi_\infty) \in\mathcal{N}^{K,L}_m$. This means that $T_{m}^{K,L}(t) (\phi_\infty,\psi_\infty) = (\phi_\infty,\psi_\infty)$
for all $t\ge 0$. Hence, by the definition of the semigroup $\{T_{m}^{K,L}\}_{t\ge 0}$, there exist chemical potentials $(\mu_\infty,\theta_\infty) \in L^2_\mathrm{uloc}(0,\infty;\mathcal{H}^1_{L,\beta})$ such that 
the quadruplet $(\phi_\infty,\psi_\infty,\mu_\infty,\theta_\infty)$ is the weak solution of 
system \eqref{conCH*} with trivial velocity fields on the interval $[0,\infty)$ associated with the initial data $(\phi_\infty,\psi_\infty)$.

Moreover, the energy inequality corresponding to $(\phi_\infty,\psi_\infty,\mu_\infty,\theta_\infty)$ reads as
\begin{align}
    \label{WS:ENERGY:STAT}
    &E_{\mathrm{free}}\big(\phi_\infty,\psi_\infty\big) + \int_s^t\intO\abs{\Grad \mu_\infty}^2 \dxr
     + \int_s^t\intG\abs{\Gradg \theta_\infty}^2 \dGr
    \nonumber \\
    &\qquad +\chi(L)\int_s^t \intG (\beta\theta_\infty-\mu_\infty)^2 \dGr
    \\[1ex]
    &\quad \le E_{\mathrm{free}}(\phi_\infty,\psi_\infty)
    \nonumber 
\end{align}
for almost all $s\ge 0$ and all $t\ge s$. This directly implies that
\begin{align*}
    \intO \abs{\Grad\mu_\infty}^2\dx + \intG \abs{\Gradg\theta_\infty}^2\dG + \chi(L)\intG (\beta\theta_\infty - \mu_\infty)^2\dG = 0
\end{align*}
a.e.~in $[0,\infty)$. Hence, in all cases of $L\in[0,\infty]$, it follows that $\mu_\infty$ and $\theta_\infty$ are constant. If $L\in[0,\infty)$, we additionally have the relation $\mu_\infty = \beta\theta_\infty$.

Furthermore, by means of Theorem~\ref{thm:highreg}, we deduce that $(\phi_\infty,\psi_\infty) \in \mathcal{W}^{2,6} \cap \mathcal{W}_{K,L,m}^\sigma$. Consequently, $(\phi_\infty,\psi_\infty,\mu_\infty,\theta_\infty)$ actually satisfies the equations \eqref{conCH*:2}, \eqref{conCH*:4} and \eqref{conCH*:5} even in the strong sense.
This means precisely that the pair $(\phi_\infty,\psi_\infty)$ is a strong solution of system \eqref{sta3}.
\end{proof}

Next, we show that every element of the $\omega$-limit set is actually a stationary point of $T^{K,L}_m$. In particular, this implies that the set of stationary points $\mathcal{N}^{K,L}_m$ is non-empty.

\begin{lemma}\label{LEM:oml}
    Suppose that \ref{ASSUMP:1}--\ref{ASSUMP:2}, \ref{S1}--\ref{S3}, \ref{D1}--\ref{D3} and \ref{D5} are fulfilled and let  $(\phi_0,\psi_0) \in \mathcal{X}_{K,L,m}$ be arbitrary.
    
    Then it holds that 
    \begin{equation*}
        \oml\subset  \mathcal{N}_{m}^{K,L},
    \end{equation*}
    which means that $\oml$ consists only of stationary points. Moreover, there exists $E_*\in\R$ such that
    \begin{equation*}
        E_{\mathrm{free}}(\phi_*,\psi_*)=E_* \quad\text{for all $(\phi_*,\psi_*)\in \oml$}.
    \end{equation*}
    In particular, $E_*$ may depend on $(\phi_0,\psi_0)$, but is independent of the choice of $(\phi_*,\psi_*)\in \oml$.
\end{lemma}

\medskip

\begin{proof}
    Our proof is largely inspired by the line of argument in \cite{AbelsWilke,P,GGPS,GrasselliPoiatti}. However, compared to these previous results, we have to overcome additional difficulties since, due to the involved velocity fields, our energy functional $E_\mathrm{free}$ is usually not decreasing along trajectories. 
    
    Let $(\phi,\psi,\mu,\theta)$ denote the weak solution of \eqref{conCH*} on $[0,\infty)$ that corresponds to the initial data $(\phi_0,\psi_0)$,
    let $(\phi_*,\psi_*)\in \oml$, and let $(t_n)_{n\in\N}$ with $t_n\to\infty$ as $n\to\infty$ be a corresponding sequence
    with 
    \begin{equation}
        \label{CONV:INI}
        (\phi(t_n),\psi(t_n))\to (\phi_*,\psi_*) \quad \text{in $\mathcal{H}^\sigma$ as $n\to \infty$.}
    \end{equation}
    Without loss of generality, we assume that $t_n>1$ for all $n\in\N$. Now, for any $n\in\N$ and any $t\in [0,\infty)$, we define 
    \begin{align*}
        (\phi_n(t),\psi_n(t),\mu_n(t),\theta_n(t)) 
        &\coloneqq (\phi(t+t_n),\psi(t+t_n),\mu(t+t_n),\theta(t+t_n)),
        \\
        (\bv_n(t),\bw_n(t))
        &\coloneqq (\bv_n(t+t_n),\bw_n(t+t_n)).
    \end{align*}
    This means that, for any $n\in\N$, the quadruplet $(\phi_n,\psi_n,\mu_n,\theta_n)$ is a solution to system \eqref{conCH*}, where the velocity fields $(\bv,\bw)$ are replaced by $(\bv_n,\bw_n)$ and the initial conditions $(\phi_0,\psi_0)$ are replaced by $(\phi(t_n),\psi(t_n))$.
    
    As strongly convergent sequences are bounded, we deduce that 
    \begin{equation*}
        \{(\phi(t_n),\psi(t_n))\}_{n\in\N} \subset \mathcal{W}^\sigma_{K,L,\alpha,\beta, m}
        \quad\text{is bounded.}
    \end{equation*} 
    Recalling Proposition~\ref{PROP:WS}, we infer that there exists a constant 
    $C_*>0$ (which may depend on the system parameters, $E_\mathrm{free}(\phi_0,\psi_0)$ and $\norm{(\bv,\bw)}_{L^2_\mathrm{uloc}(\R;\mathbf{L}^2(\Omega)\times \mathbf{L}^2(\Gamma))}$, but not on $n$) such that
    \begin{align*}
        &\norm{(\phi_n,\psi_n)}_{H^1(0,\infty;(\mathcal H^1_{L,\beta})')}
            + \norm{(\phi_n,\psi_n)}_{L^2_\mathrm{uloc}(0,\infty;\mathcal W^{2,6})}
            \\
        &\quad 
            + \norm{(F'(\phi_n),G'(\psi_n))}_{L^2_\mathrm{uloc}(0,\infty;\mathcal L^6)}
            + \norm{(\mu_n,\theta_n)}_{L^2_\mathrm{uloc}(0,\infty;\mathcal H^1_{L,\beta})}
        \leq C_*
    \end{align*}
    for all $n\in \N$. 
    Moreover, thanks to assumption \ref{D2}, we have 
    \begin{align}
        \label{CONV:VEL}
        \norm{(\bv_n,\bw_n)}_{L^2(0,\infty;\mathcal{L}^2)}^2 = \int_{t_n}^{\infty}\norm{(\bv(s),\bw(s))}_{\mathcal{L}^2}^2\ds\to 0 \qquad\text{as~}n\to \infty.
    \end{align}
    Thus, proceeding similarly to \cite[Proof of Theorem~3.4]{Knopf2025}, we employ the Banach--Alaoglu theorem and the Aubin--Lions--Simon lemma, to conclude that the quadruplet $(\phi_n,\psi_n,\mu_n,\theta_n)$ converges as $n\to\infty$ to the weak solution $(\widetilde\phi,\widetilde\psi,\widetilde\mu,\widetilde\theta)$ of system \eqref{conCH*} with the velocity fields $(\bv,\bw)=(\mathbf{0},\mathbf{0})$ (due to \eqref{CONV:VEL}) and the initial data $(\phi_*,\psi_*)$ at initial time zero (due to \eqref{CONV:INI}) in the following sense:
    \begin{alignat}{2}
        \label{CONV:PP}
        (\phi_n,\psi_n) &\rightarrow (\widetilde\phi,\widetilde\psi) 
        &&\qquad\text{weakly in~} H^1(0,\infty;(\mathcal{H}^1_{L,\beta})^\prime), 
        \nonumber \\
        & &&\qquad\text{weakly-star in~} L^\infty(0,\infty;\mathcal{H}^1_{K,\alpha}), 
        \nonumber \\
        & &&\qquad\text{weakly in~} L^2_\mathrm{uloc}(0,\infty;\mathcal{W}^{2,6}) , 
        \\
        \label{CONV:MT}
        (\mu_n,\theta_n) &\rightarrow (\widetilde\mu,\widetilde\theta) 
        &&\qquad\text{weakly in~} L^2_\mathrm{uloc}(0,\infty;\mathcal{H}^1_{L,\beta}), 
        \\
        \label{CONV:FG}
        \big(F^\prime(\phi_n),G^\prime(\psi_n)\big) &\rightarrow \big(F^\prime(\widetilde\phi), G^\prime(\widetilde\psi)\big) 
        &&\qquad\text{weakly in~} L^2_\mathrm{uloc}(0,\infty;\mathcal{L}^6).
    \end{alignat}
    
    Now, we consider the function
    \begin{align*}
        H:[0,\infty) \to \R, \quad
        H(t)\coloneqq &E_{\mathrm{free}}\big(\phi(t),\psi(t)\big)
            - \int_0^t\intO \phi\bv\cdot \Grad\mu \dxr 
        \\
        &\quad - \int_0^t\intG \psi\bw \cdot \Gradg\theta \dGr.
    \end{align*}
    It directly follows from the energy inequality \eqref{WS:ENERGY} that $H$ is non-increasing. 
    Moreover, recalling that $|\phi|< 1$ a.e.~in $Q_0$ and $|\psi|< 1$ a.e.~on $\Sigma_0$, we use Hölder's inequality and Young's inequality to deduce that
    \begin{align}
        \label{EST:F:BEL}
        \begin{split}
        H(t) 
        &\ge E_{\mathrm{free}}\big(\phi(t),\psi(t)\big)
            - \int_0^t \norm{(\bv,\bw)}_{\mathbf{L}^2(\Omega)^2 \times \mathbf{L}^2(\Gamma)}
                \norm{(\Grad\mu,\Gradg\theta)}_{\mathbf{L}^2(\Omega) \times \mathbf{L}^2(\Gamma)} \dr 
        \\
        &\ge - \norm{(\mu,\theta)}_{L^\infty(T_\mathrm{dec},\infty;\mathcal{H}^1)} 
                \int_{T_\mathrm{dec}}^\infty \mathrm{e}^{ar} \norm{(\bv(r),\bw(r))}_{\mathbf{L}^2(\Omega) \times \mathbf{L}^2(\Gamma)} \dr
        \\
        &\qquad
            - \frac{1}{2}\int_0^{T_\mathrm{dec}} \norm{(\bv,\bw)}_{\mathbf{L}^2(\Omega) \times \mathbf{L}^2(\Gamma)}^2\dr
            - \frac{1}{2} \int_0^{T_\mathrm{dec}} \norm{(\mu,\theta)}_{\mathcal{H}^1}^2 \dr
        \end{split}
    \end{align}
    for all $t\ge 0$. Here, $a \geq 0$ and $T_\mathrm{dec}>0$ are as introduced in \ref{D5}.
    In view of Proposition~\ref{PROP:WS} and Theorem~\ref{thm:highreg}, we have
    \begin{align*}
        (\mu,\theta) \in L^2(0,T_\mathrm{dec};\mathcal{H}^1) \cap L^\infty(T_\mathrm{dec},\infty;\mathcal{H}^1).
    \end{align*}
    Together with the assumptions \ref{D2} and \ref{D5} on the velocity fields, we conclude from \eqref{EST:F:BEL} that $H$ is bounded from below.
    Consequently, there exists $E_* \in\R$ such that $H(t)\to E_*$ as $t\to \infty$. 
    In particular, for any $t\in [0,\infty)$, this entails that
    \begin{equation*}
        H(t+t_n)\to E_*\quad\text{as~}n\to \infty.
    \end{equation*}
    
    Using the convergences \eqref{CONV:VEL}--\eqref{CONV:MT} along with the weak-strong convergence principle, we find that
    \begin{align}
        \label{CONV:CONV}
        \int_0^{t}\intO \phi_n\bv_n\cdot \Grad\mu_n \dxr + \int_0^{t}\intG \psi_n\bw_n \cdot \Gradg\theta_n \dGr  \to 0,\quad\text{as }n\to\infty. 
    \end{align}
    By the definition of $(\phi_n,\psi_n)$, we infer that, for every $t\ge 0$, 
    \begin{align}
        \label{CONV:EN:1}
        \begin{split}
        &E_{\mathrm{free}}\big(\phi_n(t),\psi_n(t)\big)
        \\
        &\quad =\; H(t+t_n) 
            + \int_0^{t}\intO \phi_n\bv_n\cdot \Grad\mu_n \dxr 
            + \int_0^{t}\intG \psi_n\bw_n \cdot \Gradg\theta_n \dGr
        \;\to\; E_*
        \end{split}
    \end{align}
    as $n\to\infty$.
    Next, we recall that $(\phi_n,\psi_n,\mu_n,\theta_n)$ satisfies the energy inequality 
    \begin{align}
        \label{WS:ENERGY:N}
        \begin{aligned}
            &E_{\mathrm{free}}\big(\phi_n(t),\psi_n(t)\big) 
                + \int_0^t\intO\abs{\Grad \mu_n}^2 \dxr
                + \int_0^t\intG\abs{\Gradg \theta_n}^2 \dGr
            \\
            & \quad 
                +\chi(L)\int_0^t \intG (\beta\theta_n-\mu_n)^2 \dGr		
            \\
            & \le E_{\mathrm{free}}(\phi_n(0),\psi_n(0))	
                + \int_0^t\intO \phi_n\bv_n \cdot \Grad\mu_n \dxr
                + \int_0^t\intG \psi_n\bw_n \cdot \Gradg\theta_n \dGr
        \end{aligned}
    \end{align}
    for all $t\ge 0$.
    Using the convergences \eqref{CONV:MT}, \eqref{CONV:CONV} and \eqref{CONV:EN:1}, 
    we use the weak lower semicontinuity of the involved norms to conclude that
    \begin{align*}
        &E_* + \int_0^t\intO\abs{\Grad \widetilde\mu}^2 \dxr
    	   + \int_0^t\intG\abs{\Gradg \widetilde\theta}^2 \dGr
           + \chi(L)\int_0^t \intG (\beta\widetilde\theta-\widetilde\mu)^2 \dGr
        \\[2ex]
        &\quad
        \le \underset{n\to\infty}{\lim}\, E_{\mathrm{free}}\big(\phi_n(t),\psi_n(t)\big) 
        \\
        &\qquad
            + \underset{n\to\infty}{\liminf}\,\Bigg[ \int_0^t\intO\abs{\Grad \mu_n}^2 \dxr
            + \int_0^t\intG\abs{\Gradg \theta_n}^2 \dGr
        \\
        &\qquad\qquad\qquad\qquad
            +\chi(L)\int_0^t \intG (\beta\theta_n-\mu_n)^2 \dGr	\Bigg]	
        \\[1ex]
        &\quad
        \le \underset{n\to\infty}{\limsup}\, \Bigg[ E_{\mathrm{free}}\big(\phi_n(t),\psi_n(t)\big) 
            + \int_0^t\intO\abs{\Grad \mu_n}^2 \dxr
            + \int_0^t\intG\abs{\Gradg \theta_n}^2 \dGr
        \\
        &\qquad\qquad\qquad\qquad
            +\chi(L)\int_0^t \intG (\beta\theta_n-\mu_n)^2 \dGr	\Bigg]	
        \\[2ex]
        &\quad 
        \le \underset{n\to\infty}{\limsup}\; E_{\mathrm{free}}(\phi_n(0),\psi_n(0))	
        = E_*\,.
    \end{align*}
    Consequently,
    \begin{align*}
        \int_0^t\intO\abs{\Grad \widetilde\mu}^2 \dxr
    	   + \int_0^t\intG\abs{\Gradg \widetilde\theta}^2 \dGr
           + \chi(L)\int_0^t \intG (\beta\widetilde\theta-\widetilde\mu)^2 \dGr \leq 0
    \end{align*}
    for all $t\ge 0$. This readily entails that $\widetilde \mu$ is spatially constant in $\Omega$ and $\widetilde \theta$ is spatially constant on $\Sigma$. As $(\widetilde\phi,\widetilde\psi,\widetilde\mu,\widetilde\theta)$ is the weak solution of system \eqref{conCH*} with trivial velocity fields to the initial data $(\phi_*,\psi_*)$, it immediately follows that
    \begin{align*}
        \big(\widetilde \phi(t),\widetilde\psi(t)\big) = (\phi_*,\psi_*) \quad\text{for all $t\ge 0$,}
    \end{align*}
    and thus, by comparison, $\widetilde\mu,\widetilde \theta$ are constant also in time.
    Consequently, $(\phi_*,\psi_*) \in \mathcal{N}_{m}^{K,L}$. Finally, recalling that $\sigma>\frac{d}{2}$, we use \eqref{CONV:INI} to deduce that
    \begin{align*}
        E(\phi_n(0),\psi_n(0)) = E(\phi(t_n),\psi(t_n)) \to E(\phi_*,\psi_*)
    \end{align*}
    as $n\to\infty$. Because of \eqref{CONV:EN:1}, it follows that
    \begin{align*}
        E(\phi_*,\psi_*) = E_*
    \end{align*}
    due to the uniqueness of the limit. We point out that, by construction, $E_*$ may depend on $(\phi_0,\psi_0)$, but is independent of the choice of $(\phi_*,\psi_*) \in \oml$.
    Hence, the proof is complete.
\end{proof}

\medskip\pagebreak[3]

The next lemma shows that every stationary solution of \eqref{conCH*} satisfies a strict separation property. 

\begin{lemma}\label{LM.boundedstat}
	Suppose that the assumptions \ref{ASSUMP:1}--\ref{ASSUMP:2}, \ref{S1}-\ref{S3} and \ref{D1} are fulfilled. 
    
    Then for any $(\phi,\psi)\in \mathcal{N}_{m}^{K,L} $ there exists $\delta_\star = \delta_\star(\phi,\psi)\in(0,1)$ such that 
    \begin{align}
        \label{SEPPROP:SS}
        \norm{\phi}_{C(\overline{\Om})} \leq 1 - \delta_\star, \qquad \norm{\psi}_{C(\Ga)} \leq 1 - \delta_\star.
    \end{align}
\end{lemma}

\begin{proof}
    As any $(\phi,\psi)\in \mathcal{N}_{m}^{K,L} $ is a solution to the semilinear bulk-surface elliptic system \eqref{sta3}, the strict separation property \eqref{SEPPROP:SS} follows from 
    \cite[Proposition~5.3]{Giorgini2025}.
\end{proof}

\medskip

To show that any weak solution of system \eqref{conCH*} converges to a stationary point, the following lemma is crucial.

\begin{lemma}[{\L}ojasiewicz--Simon inequality]\label{LEM:LS}
Suppose that the assumptions \ref{ASSUMP:1}--\ref{ASSUMP:2}, \ref{S1}-\ref{S3} and \ref{D1} are fulfilled.
In addition, we assume that $F_1, G_1$ are real analytic functions on $(-1,1)$, and $F_2, G_2$ are real analytic functions on $\R$. 

Then for any $(\phi_\infty, \psi_\infty)\in\omega^{K,L}(\phi_0,\psi_0)$, there exist constants $\varpi^\ast\in (0,\tfrac 12)$, $b^\ast > 0$, and $C>0$, such that
\begin{align*}
   C \left\|\mathbf{P}_L\begin{pmatrix}
        -\Lap\zeta + F^\prime(\zeta), \\
        -\Lapg\xi + G^\prime(\xi) + \alpha\deln\zeta
    \end{pmatrix}\right\|_{\mathcal{L}^2} \geq \abs{E_{\mathrm{free}}(\zeta, \xi) - E_{\mathrm{free}}(\phi_\infty, \psi_\infty)}^{1 - \varpi^\ast}
\end{align*}
for all $(\zeta,\xi)\in\mathcal{W}^2_{K,L,m}$ satisfying $\norm{(\zeta - \phi_\infty, \xi - \psi_\infty)}_{\mathcal{H}^2}\leq b^\ast$.
Here, $\mathbf{P}_L$ denotes the projection of $\mathcal{L}^2$ onto
\begin{equation*}
    \begin{cases} 
        \{\scp{\phi}{\psi}\in\mathcal{L}^2 : \beta\abs{\Om}\meano{\phi} + \abs{\Ga}\meang{\psi} = 0 \} &\text{if~} L\in[0,\infty), \\
        \{\scp{\phi}{\psi}\in\mathcal{L}^2: \meano{\phi} = \meang{\psi} = 0 \} &\text{if~}L=\infty.
        \end{cases} 
\end{equation*}
\end{lemma}

\begin{proof}
    For $K=0$, $L\in [0,\infty]$ and $\alpha=1$, the result can be found in \cite[Lemma~5.2]{Lv2024a}, and for $K\in [0,\infty)$, $L=\infty$ and $\alpha=1$, it was established in \cite[Lemma~5.2]{Lv2024b}. For general $K\in [0,\infty]$ and $\alpha\in [-1,1]$ as specified by \ref{ASSUMP:1}, the statement can be verified analogously.
\end{proof}

\medskip

\begin{theorem}[Convergence to an equilibrium]
    Suppose that the assumptions \ref{ASSUMP:1}--\ref{ASSUMP:2}, \ref{S1}-\ref{S3} and \ref{D1}-\ref{D3} are fulfilled. Moreover, suppose that \ref{D5} holds with $a>0$.
    In addition, we assume that $F_1, G_1$ are real analytic on $(-1,1)$ and $F_2, G_2$ are real analytic on $\R$. 
    Let $(\phi_0,\psi_0) \in \mathcal{X}_{K,L,m}$, and let $(\phi,\psi,\mu,\theta)$ denote the corresponding weak solution to \eqref{conCH*} on $[0,\infty)$. 
    
    Then there exists $(\phi_\infty,\psi_\infty)\in \mathcal{N}^{K,L}_m$ such that
    \begin{align}
        \lim_{t\rightarrow\infty} \norm{(\phi(t),\psi(t)) - (\phi_\infty,\psi_\infty)}_{\mathcal{H}^2} = 0.\label{H2cv}
    \end{align}
\end{theorem}

\begin{proof}
Since $\oml$ is a compact subset of $\mathcal{W}^\sigma_{K,L,m}$ with $\sigma\in(\frac{d}{2},2)$ (see Lemma~\ref{Lemma:oml}), the embedding $\mathcal{W}^\sigma_{K,L,m}\emb C(\overline\Om)\times C(\Ga)$ implies that $\oml$ is also a compact subset of $C(\overline\Om)\times C(\Ga)$. 
Then, as $\oml\subset  \mathcal{N}_{m}^{K,L}$ and any pair $(\phi,\psi)\in \mathcal N^{K,L}_m$ (see Lemma~\ref{LEM:oml}) is strictly separated from the pure phases (see Lemma~\ref{LM.boundedstat}), we can argue by contradiction (see, for instance, \cite{AbelsWilke} or \cite[Proof of Lemma 3.11]{GGPS}) to show that there exists $\delta\in(0,1)$ such that
\begin{align*}
    \norm{\phi_*}_{C(\overline{\Om})}\leq 1-\delta,\quad \norm{\psi_*}_{C(\Gamma)}\leq 1-\delta\qquad\text{for all~} (\phi_*,\psi_*)\in \oml.
\end{align*}
Here, it is crucial that the width of the separation layer $\delta$ is independent of $(\phi_*,\psi_*)\in \oml$.

Moreover, we know from Lemma~\ref{Lemma:oml} that 
\begin{align}
    \label{CONV:WSIG}
    \dist_{\mathcal W^\sigma_{K,L,m}}((\phi(t),\psi(t)),\oml)\to 0\qquad\text{as~}t\to \infty,
\end{align}
with $\sigma\in\big(\frac d2,2\big)$.

As the embedding $\mathcal W^{s}_{K,L,m}\emb C(\overline\Om)\times C(\Ga)$ is compact, there $T_S>0$ such that 
\begin{align}
    \norm{(\phi(t),\psi(t))}_{C(\overline\Om)\times C(\Gamma)}\leq 1-\frac{\delta}{2},\qquad\text{for all~} t\geq T_S.\label{unifsep}
\end{align}
Thanks to this uniform separation property, we can apply Corollary~\ref{COR:H3} to infer that 
\begin{align}
\norm{(\phi,\psi)}_{L^\infty(T_S,\infty;\mathcal H^3)}\leq C.
    \label{regH3}
\end{align}
In particular, this entails that $\oml$ is a bounded subset of $\mathcal H^3$.

Now, to show that $\oml$ is a singleton, we need to make use of the additional assumption \ref{D5}. Using Remark~\ref{control}, we deduce that
\begin{align*}
    &\int_{T_{\mathrm{dec}}}^\infty \norm{(\bv,\bw)}_{\mathcal{L}^2}^{\frac{1-2\varpi}{1-\varpi}}\ds\\
    &\quad\leq \left(\int_{T_\mathrm{dec}}^\infty 
        \mathrm{e}^{as} \bignorm{\big(\bv(s),\bw(s) \big)}_{\LL^2} 
        \ds\right)^{\frac{1-2\varpi}{1-\varpi}} \left(\int_{T_\mathrm{dec}}^\infty 
        \mathrm{e}^{-a\frac{1-2\varpi}{\varpi}s} 
        \ds\right)^{\frac{\varpi}{1-\varpi}} \\
        &\quad\leq \Big(\frac{\varpi}{a(1-2\varpi)}\Big)^{\frac{\varpi}{1-\varpi}} \left(\int_{T_\mathrm{dec}}^\infty 
        \mathrm{e}^{as} \bignorm{\big(\bv(s),\bw(s) \big)}_{\LL^2} 
        \ds\right)^{\frac{1-2\varpi}{1-\varpi}}
\end{align*}
for all $\varpi\in(0,\tfrac12)$.
This entails that
\begin{align}\label{nonnincr}
    \norm{(\bv,\bw)}_{L^{\frac{1-2\varpi}{1-\varpi}}(T_{\mathrm{dec}},\infty;\mathcal{L}^2)} \leq \Big(\frac{\varpi}{a(1-2\varpi)}\Big)^{\frac{\varpi}{1-2\varpi}}\int_{T_\mathrm{dec}}^\infty 
        \mathrm{e}^{as} \bignorm{\big(\bv(s),\bw(s) \big)}_{\LL^2} 
        \ds
\end{align}
for all $\varpi\in(0,\tfrac12)$.
Notice that the assumptions \ref{D3} and \ref{D5} readily imply that $\bv(t)\to \mathbf 0$ in $\bL^2(\Om)$ and $\mathbf w(t)\to \mathbf 0$ in $\bL^2(\Gamma)$ as $t\to \infty$.

Now, thanks to \eqref{CONV:WSIG}, \eqref{regH3} and the compact embedding $\mathcal H^3  \hookrightarrow\mathcal H^2$, we have 
\begin{align}
\dist_{\mathcal H^2}\big((\phi(t),\psi(t)), \oml\big)\to 0 \qquad\text{as~}t\rightarrow\infty.
    \label{distH2}
\end{align}
Moreover, as $\oml$ is a compact subset of $\mathcal{W}^\sigma_{K,L,m}$ (see Lemma~\ref{Lemma:oml}) and a bounded subset of $\mathcal H^3$, we conclude by interpolation that $\oml$ is a compact subset of $\mathcal H^2$.

Now, for any $(\phi_*,\psi_*) \in \oml$, let $\varpi^* \in (0,\tfrac 12)$ and $b^*>0$ be as obtained by Lemma~\ref{LEM:LS}. Then, we obviously have 
\begin{align*}
    \oml\subset \bigcup_{(\phi_*,\psi_*) \in \oml} B_{b^*}(\phi_*,\psi_*).
\end{align*}
where $B_{b^*}(\phi_*,\psi_*)$ denotes the open ball in $\mathcal H^2$ with center $(\phi_*,\psi_*)$ and radius $b_*$.
As $\oml$ is a compact subset of $\mathcal H^2$, there exist finitely many points $\{(\phi^*_i,\psi^*_i)\}_{i=1,...,M}$ with corresponding radii $\{b^*_i\}_{i=1,...,M}$
such that 
\begin{align*}
    \oml\subset \bigcup_{i=1}^M B_{b^*_i}(\phi^*_i,\psi^*_i)
\end{align*}
and the {\L}ojasiewicz--Simon inequality is fulfilled on each of these balls with an exponent $\varpi_i^* \in (0,\tfrac 12)$. This means that for every $i\in\{1,...,M\}$, there exists a constant $C_i>0$ such that 
\begin{align}
    C_i \left\|{\mathbf{P}_L}
    \begin{pmatrix}
        -\Lap\zeta + F^\prime(\zeta) \\
        -\Lapg\xi + G^\prime(\xi) + \alpha\deln\zeta
    \end{pmatrix}\right\|_{\mathcal{L}^2} \geq \abs{E_{\mathrm{free}}(\zeta, \xi) -E_*}^{1 - \varpi^*_i}\label{LJi}
\end{align}
for all $(\zeta,\xi)\in B_{b^*_i}(\phi^*_i,\psi^*_i)$. Now, we define
\begin{align*}
    C \coloneqq \underset{i=1,...,M}{\max}\, C_i \, >0
    \quad\text{and}\quad
    \varpi^* \coloneqq \underset{i=1,...,M}{\min}\,\varpi^*_i \in \big(0,\tfrac 12\big).
\end{align*}
As the balls $B_{b^*_i}(\phi^*_i,\psi^*_i)$, $i=1,...,M$, are open, there further exists
\begin{align*}
    0 < b < \underset{i=1,...,M}{\min} b^*_i
\end{align*}
such that for all $(\zeta,\xi)\in\mathcal{W}^2_{K,L,m}$, it holds that
\begin{align*}
    (\zeta,\xi) \in \bigcup_{i=1}^M B_{b^*_i}(\phi^*_i,\psi^*_i)
    \quad\text{if}\quad
    \dist_{\mathcal H^2}\big((\zeta,\xi),\oml\big)\leq b.
\end{align*}
Hence, in view of \eqref{LJi}, we conclude that
\begin{align}
    \begin{split}
    &C \left\|{\mathbf{P}_L}
    \begin{pmatrix}
        -\Lap\zeta + F^\prime(\zeta) \\
        -\Lapg\xi + G^\prime(\xi) + \alpha\deln\zeta
    \end{pmatrix}\right\|_{\mathcal{L}^2} 
    \geq \abs{E_{\mathrm{free}}(\zeta, \xi) -E_*}^{1 - \varpi^*}
    \label{LJ}
    \end{split}
\end{align}
for all $(\zeta,\xi)\in\mathcal{W}^2_{K,L,m}$ with 
\begin{align*}
    \dist_{\mathcal H^2}\big((\zeta,\xi),\oml\big)\leq b
    \quad\text{and}\quad
    \abs{E_{\mathrm{free}}(\zeta, \xi) -E_*} \le 1.
\end{align*}

Recalling Lemma~\ref{LEM:oml} and the convergence \eqref{distH2}, we conclude that there exists a time $T_F>\max\{T_S,T_\mathrm{dec}\}$ such that
\begin{align*}
    \dist_{\mathcal H^2}\big((\phi(t),\psi(t)),\oml\big)\leq b 
    \quad\text{and}\quad
    \abs{E_{\mathrm{free}}(\phi(t),\psi(t))-E_*} \le 1
\end{align*}
for all $t\geq T_F$.
Thus, employing \eqref{LJ}, we obtain
\begin{align}
    \begin{split}
   \abs{E_{\mathrm{free}}(\phi(t), \psi(t)) -E_*}^{1 - \varpi^\ast} &\leq C \left\|{\mathbf{P}_L}\begin{pmatrix}
        -\Lap\phi(t) + F^\prime(\phi(t)), \\
        -\Lapg\psi(t) + G^\prime(\psi(t)) + \alpha\deln\phi(t)
    \end{pmatrix}\right\|_{\mathcal{L}^2} 
    \\[1ex]
    &\leq C\norm{(\mu(t),\theta(t))}_{L,\beta}. \label{LJf}
    \end{split}
\end{align}
for all $t\geq T_F$. Now, using \cite[Lemma 7.1]{FS}, the assertion can be established by following the line of argument in \cite{GrasselliPoiatti}. To this end, let $s,t\geq T_F$ be arbitrary. Then, recalling \eqref{unifsep}, we use the energy inequality \eqref{WS:ENERGY} to deduce that
\begin{align*}
   &\int_s^t \norm{(\mu,\theta)}_{L,\beta}^2\dr+\frac12 \int_s^t\norm{(\bv,\bw)}_{\mathcal{L}^2}^2\dr
   \\[1ex]
   &\quad\leq E_{\mathrm{free}}(\phi(s),\psi(s))-E_{\mathrm{free}}(\phi(t),\psi(t)) 
   \\
   &\qquad +\int_s^t\intO \phi\bv\cdot \Grad\mu \dxr
			+ \int_s^t\intG \psi\bw \cdot \Gradg\theta \dGr+\frac12 \int_s^t\norm{(\bv(\tau),\bw(\tau))}_{\mathcal{L}^2}^2\dr 
    \\[1ex]
    &\quad\leq E_{\mathrm{free}}(\phi(s),\psi(s))-E_{\mathrm{free}}(\phi(t),\psi(t))
    \\
    &\qquad+ \int_s^t\norm{(\bv,\bw)}_{\mathcal{L}^2}^2\dr+\frac12\int_s^t \norm{(\mu,\theta)}_{L,\beta}^2\dr.
\end{align*}
After rearranging some terms and elevating to the power $2(1-\varpi^\ast) > 0$, we obtain 
\begin{align*}
   &\left(\frac12\right)^{2(1-\varpi^\ast)}\left(\int_s^t \norm{(\mu,\theta)}_{L,\beta}^2\dr+\int_s^t\norm{(\bv,\bw)}_{\mathcal{L}^2}^2\dr\right)^{2(1-\varpi^\ast)} 
   \\
   &\quad\leq \left( E_{\mathrm{free}}(\phi(s),\psi(s))-E_{\mathrm{free}}(\phi(t),\psi(t))+\int_s^t\norm{(\bv,\bw)}_{\mathcal{L}^2}^2\dr\right)^{2(1-\varpi^\ast)} 
   \\[1ex]
   &\quad\leq 2^{2(\frac12-\varpi^\ast)}\vert E_{\mathrm{free}}(\phi(s),\psi(s))-E_*\vert^{2(1-\varpi^\ast)}+2^{2(\frac12-\varpi^\ast)}\vert E_{\mathrm{free}}(\phi(t),\psi(t))-E_*\vert^{2(1-\varpi^\ast)} 
   \\
   &\qquad +2^{2(\frac12-\varpi^\ast)}\left(\int_s^t\norm{(\bv,\bw)}_{\mathcal{L}^2}^2\dr\right)^{2(1-\varpi^\ast)}.
\end{align*}
Now, we send $t\to\infty$. Recalling that $E_{\mathrm{free}}(\phi(t),\psi(t))\to E_*$ as $t\to \infty$, we apply \eqref{LJf} to infer that
\begin{align*}
   & \left(\frac12\right)^{2(1-\varpi^\ast)}\left(\int_s^\infty \norm{(\mu,\theta)}_{L,\beta}^2\dr+\int_s^\infty\norm{(\bv,\bw)}_{\mathcal{L}^2}^2\dr\right)^{2(1-\varpi^\ast)} \\
   &\quad\leq 2^{2(\frac12-\varpi^\ast)}\vert E_{\mathrm{free}}(\phi(s),\psi(s))-E_*\vert^{2(1-\varpi^\ast)}+2^{2(\frac12-\varpi^\ast)}\left(\int_s^\infty\norm{(\bv,\bw)}_{\mathcal{L}^2}^2\dr\right)^{2(1-\varpi^\ast)} \\
   &\quad\leq 2^{2(\frac12-\varpi^\ast)}C\norm{(\mu(s),\theta(s))}^{2}_{L,\beta}+2^{2(\frac12-\varpi^\ast)}\left(\int_s^\infty\norm{(\bv,\bw)}_{\mathcal{L}^2}^2\dr\right)^{2(1-\varpi^\ast)}.
\end{align*}
Since $T_F>T_\mathrm{dec}$, the mapping $(T_F,\infty)\ni t\mapsto \norm{(\bv(t),\bw(t))}_{\mathcal{L}^2}$ is non-increasing according to assumption \ref{D5}. Thus, using the estimate \eqref{nonnincr}, we deduce that
\begin{align*}
   &2^{2(\frac12-\varpi^\ast)}\left(\int_s^\infty\norm{(\bv,\bw)}_{\mathcal{L}^2}^2\dr\right)^{2(1-\varpi^\ast)} \\
   &\quad\leq 2^{2(\frac12-\varpi^\ast)}\left(\norm{(\bv(s),\bw(s))}_{\mathcal{L}^2}^{\frac{1}{1-\varpi^\ast}}\int_s^\infty\norm{(\bv,\bw)}_{\mathcal{L}^2}^{\frac{1-2\varpi^\ast}{1-\varpi^\ast}}\dr\right)^{2(1-\varpi^\ast)} \\ 
   &\quad\leq 2^{2(\frac12-\varpi^\ast)}C\norm{(\bv(s),\bw(s))}_{\mathcal{L}^2}^{2}.
\end{align*}
Hence, we conclude that
\begin{align*}
   & \left(\int_s^\infty \norm{(\mu,\theta)}_{L,\beta}^2\dr+\int_s^\infty\norm{(\bv,\bw)}_{\mathcal{L}^2}^2\dr\right)^{2(1-\varpi^\ast)} \\
   &\quad\leq C\left(\norm{(\mu(s),\theta(s))}^{2}_{L,\beta}+\norm{(\bv(s),\bw(s))}_{\mathcal{L}^2}^{2}\right)
\end{align*}
for almost all $s\geq T_F$. According to \cite[Lemma 7.1]{FS}, this yields
\begin{align*}
    \norm{(\mu,\theta)}_{L^1(T_*,\infty)}\leq C
\end{align*}
for some $T_*>0$. 
Recalling that $(\bv(t),\bw(t))\to(\mathbf{0},\mathbf{0})$ in $\mathcal{L}^2$ as $t\to \infty$,
we infer from \eqref{conCH*:1} and \eqref{conCH*:3} that 
\begin{align*}
    (\partial_t\phi,\partial_t \psi)\in L^1(T_*,\infty;(\mathcal{H}^1_{L,\beta})^\prime).
\end{align*}
Hence, using the fundamental theorem of calculus, we conclude that, as $t\to\infty$,
\begin{align*}
    (\phi(t),\psi(t))\to (\phi(T_*),\psi(T_*)) + \int_{T_*}^\infty(\delt\phi,\delt\psi)\dr=:(\phi_\infty,\psi_\infty) \quad\text{in $(\mathcal{H}^1_{L,\beta})^\prime$}.
\end{align*}
This allows us to identify $\oml=\{(\phi_\infty,\psi_\infty)\}$, which completes the proof, since \eqref{H2cv} directly comes from \eqref{distH2}.
\end{proof}



\section*{Acknowledgements}
PK and JS were supported by the Deutsche Forschungsgemeinschaft (DFG, German Research Foundation): on the one hand by the DFG-project 524694286, and on the other hand by the RTG 2339 ``Interfaces, Complex Structures, and Singular Limits". Their support is gratefully acknowledged.

AP was supported in part by the Österreichischer Wissenschaftsfonds FWF (FWF, Austrian Science Fund) \href{https://doi.org/10.55776/ESP552}{10.55776/ESP552}, whose support is gratefully acknowledged. AP is member of Gruppo Nazionale per l’Analisi Matematica, la Probabilità e le loro Applicazioni (GNAMPA) of
Istituto Nazionale per l’Alta Matematica (INdAM). 

Parts of this work were done while AP and SY were visiting the Faculty of Mathematics of the University of Regensburg, whose hospitality is greatly appreciated.
AP's visit was funded by the Alexander von Humboldt Foundation, while SY's visit was funded by the Deutscher Akademischer Austauschdienst (DAAD, German Academic Exchange Service). Their support is gratefully acknowledged.

For open access purposes, the authors have applied a CC BY public copyright license to
any author accepted manuscript version arising from this submission.




\footnotesize

\bibliographystyle{abbrv}
\bibliography{KPSY}

\end{document}